\documentclass[10pt]{article}
\usepackage{mathrsfs,amsthm,graphicx,color,verbatim,bbm,amsmath,amsfonts,amssymb,newclude,nicefrac,amsfonts,graphicx,color,enumerate,hyperref,cleveref}
\usepackage[margin=1in]{geometry}
\theoremstyle{plain}
\newtheorem{theorem}{Theorem}[section]
\newtheorem{lemma}[theorem]{Lemma}
\newtheorem{prop}[theorem]{Proposition}
\newtheorem{cor}[theorem]{Corollary}

\newtheorem{setting}[theorem]{Setting}

\numberwithin{equation}{section}

\renewcommand{\rho}{\varrho}
\newcommand{\E}{\mathbb{E}}
\renewcommand{\P}{\mathbb{P}}

\newcommand{\R}{\mathbb{R}}
\newcommand{\N}{\mathbb{N}}
\newcommand{\Z}{\mathbb{Z}}

\renewcommand{\d}{\mathrm{d}}
\newcommand{\norm}[1]{\left\lVert#1\right\rVert}

\newcommand{\ddt}{\tfrac{\d}{\d t}}
\newcommand{\ddtn}[1]{\tfrac{\d^{#1}}{\d t^{#1}}}
\newcommand{\normfac}{\tfrac{1}{\sqrt{2\pi}}}
\newcommand{\elnt}{e^{-\frac{1}{2}[\ln(t)]^2}}
\newcommand{\abs}[1]{\left|{#1}\right|}
\newcommand{\pa}[1]{\left({#1}\right)}

\newcommand{\eps}{\varepsilon}

\renewcommand{\phi}{\varphi}
\newcommand{\oneto}[1]{\{1,2,\dots,{#1}\}}
\newcommand{\zeroto}[1]{\{0,1,\dots,{#1}\}}

\newcommand{\dimin}{\dim_{\mathrm{in}}}
\newcommand{\dimout}{\dim_{\mathrm{out}}}
\newcommand{\relu}{\rho}
\newcommand{\leps}{_{\eps}}
\newcommand{\brac}[1]{\left[{#1}\right]}
\newcommand{\tprod}{\textstyle\prod\limits} 
\newcommand{\C}{\mathcal{C}}

\newcommand{\Nn}{\mathcal{N}}
\renewcommand{\L}{\mathcal{L}}
\newcommand{\M}{\mathcal{M}}
\newcommand{\Nall}{\mathfrak{N}}
\newcommand{\Rr}{R_{\rho}}
\newcommand{\Kmax}{K_{\mathrm{max}}}
\newcommand{\Id}{\mathrm{Id}}
\newcommand{\Idd}{\mathrm{Id}_{\R^d}}
\newcommand{\Pc}{\mathcal{P}}
\newcommand{\fninf}{\norm{f}_{n,\infty}}
\newcommand{\CN}{\mathcal{C}^n}

\makeatletter
\def\maketag@@@#1{\hbox{\m@th\normalfont\normalsize#1}}
\makeatother

\begin{document}

\title{DNN Expression Rate Analysis of High-dimensional PDEs:
	Application to Option Pricing\thanks{This work was performed during visits
of PG at the Seminar for Applied Mathematics and the FIM of ETH Z\"urich,
and completed during the thematic term ``Numerical Analysis of Complex PDE Models in the Sciences''
at the Erwin Schr\"odinger Institute, Vienna, from June-August, 2018.
AJ acknowledges support by the Swiss National Science Foundation under grant No. 175699
DE and PhG are supported in part by the Austrian Science Fund (FWF) under project number P 30148.
}
}

\author{Dennis Elbr\"achter, Philipp Grohs, Arnulf Jentzen, and Christoph Schwab
	}

\maketitle

\begin{abstract}
We analyze approximation rates by deep ReLU networks 
of a class of multi-variate solutions of Kolmogorov equations
which arise in option pricing. 
Key technical devices are deep ReLU architectures 
capable of efficiently approximating tensor products. Combining this with results concerning the approximation of well behaved 
(i.e. fulfilling some smoothness properties) univariate functions, 
this provides insights into rates of deep ReLU approximation of 
multi-variate functions with tensor structures. 
We apply this in particular to the model problem given by the 
price of a European maximum option on a basket of $d$ assets
within the Black-Scholes model for European maximum option pricing.
We prove that the solution to the $d$-variate option pricing problem can be
approximated up to an $\eps$-error by a deep ReLU network with depth $\mathcal{O}\big(\ln(d)\ln(\eps^{-1})+\ln(d)^2\big)$ and
$\mathcal{O}\big(d^{2+\frac{1}{n}}\eps^{-\frac{1}{n}}\big)$ non-zero weights, where $n\in \N$ is arbitrary (with the constant implied in $\mathcal{O}(\cdot)$ depending on $n$).
The techniques developed in the constructive proof are of independent interest 
in the analysis of the expressive power of deep neural networks 
for solution manifolds of PDEs in high dimension.
\end{abstract}


\section{Introduction}
\label{sec:intro}

\subsection{Motivation}
\label{sec:MotMainRes}
The development of new classification and regression algorithms based on deep neural networks -- coined 
``Deep Learning" -- revolutionized the area of artificial intelligence, machine learning, 
and data analysis \cite{goodfellow2016deep}.
More recently, these methods have been applied to the numerical solution of 
partial differential equations (PDEs for short) \cite{SirignanoSpiliopoulos2017,FujiiTakahashiTakahashi2017,EYu2017,KhooLuYing2017,Labordere2017,BeckEJentzen2017,EHanJentzen2017a,EHanJentzen2017b,mishra2018machine}. 
In these works it has been empirically observed that deep learning-based methods work exceptionally 
well when used for the numerical solution of high-dimensional problems arising in option pricing. 
The numerical experiments carried out in \cite{BeckEJentzen2017,EHanJentzen2017a,EHanJentzen2017b,KolmogorovNumerics} 
in particular suggest that deep learning-based methods may not suffer from the curse of dimensionality 
for these problems, but only few theoretical results exist which support this claim:
 In \cite{schwab2017deep}, a first theoretical result on rates of expression of 
infinite-variate generalized polynomial chaos expansions for solution manifolds of
certain classes of parametric PDEs has been obtained. Furthermore, recent work \cite{GHJv18_2680,BGJ18_2679} shows
that the algorithms introduced in \cite{KolmogorovNumerics} for the numerical solution of Kolmogorov PDEs are free of the curse of dimensionality in terms of network size and training sample complexity. 

Neural networks constitute a parametrized class of functions constructed by successive applications 
of affine mappings and coordinatewise nonlinearities, see \cite{pinkus1999approximation} 
for a mathematical introduction. 
As in \cite{PetVoigt}, we introduce a neural network via a tuple of matrix vector pairs 
$$\Phi=(((A^1_{i,j})_{i,j=1}^{N_1,N_0},(b^1_i)_{i=1}^{N_1}),\dots,
  ((A^L_{i,j})_{i,j=1}^{N_L,N_{L-1}},(b^L_i)_{i=1}^{N_L}))\in \times_{l=1}^L\pa{\R^{N_l\times N_{l-1}}\times\R^{N_l}}$$ 
for given hyperparameters $L\in\N$, $N_0,N_1,\dots,N_L\in\N$.
Given an ``activation function'' $\varrho\in C(\R,\R)$, 
a neural network $\Phi$ then describes a function 
$R_\varrho(\Phi)\in C(\R^{N_0},\R^{N_L})$ that can be evaluated 
by the recursion 
\begin{align}
x_l = \varrho( A_{l} x_{l-1} + b_{1}), l=1,\dots , L-1,\quad 
      \brac{ R_{ \rho }( \Phi ) }( x_0 ) = A_L x_{ L - 1 } + b_L.
\end{align}
The number of nonzero values in the matrix vector tuples defining $\Phi$ 
describe the size of $\Phi$ which will be denoted by $\mathcal{M}(\Phi)$ and the depth of the network $\Phi$, i.e. its number of affine transformations, will be denoted by $\L(\Phi)$.
We refer to \Cref{NNSetting} for a more detailed description.
A popular activation function $\relu$ is the so-called \emph{``Rectified Linear Unit''}
$\mathrm{ReLU}(x)=\max\{x,0\}$ \cite{goodfellow2016deep}.  

An increasing body of research addresses 
the approximation properties (or ``expressive power'')
of deep neural networks, where by ``approximation properties'' we mean the study 
of the optimal tradeoff between the size $\mathcal{M}(\Phi)$ and the approximation error 
$\|u-R_\varrho(\Phi)\|$ of neural networks approximating functions $u$ from a given function class. 
Classical references include 
\cite{hornik1990universal,cybenko1989approximation,barron1993universal,chui1994neural} 
as well as the summary \cite{pinkus1999approximation} and the references therein. 
In these works it is shown that deep neural networks provide optimal approximation rates 
for classical smoothness spaces such as Sobolev spaces or Besov spaces. 
More recently these results have been extended to Shearlet and Ridgelet spaces \cite{bolcskei2017optimal}, 
Modulation spaces \cite{pere}, piecewise smooth functions \cite{PetVoigt} and 
polynomial chaos expansions \cite{schwab2017deep}. 
All these results indicate that all classical approximation methods 
based on sparse expansions can be emulated by neural networks.

\subsection{Contributions and Main Result}
\label{sec:Contrib}
As a first main contribution of this work we show in \Cref{NN2} 
that low-rank functions of the form 
\begin{align}\label{eq:low_rank_form}
 (x_1,\dots , x_d)\in \R^d \mapsto \sum_{s=1}^Rc_s\prod_{j=1}^dh_j^s(x_j),
 \end{align}
with $h_j^s\in C(\R,\R)$ sufficiently regular and $(c_s)_{s=1}^R\subseteq\R$
can be approximated to a given relative precision by deep ReLU neural networks 
of size scaling like $Rd^2$.
In particular, we obtain a dependence on the dimension $d$ that is only polynomial and not exponential, i.e. we avoid the curse of dimensionality. 
In other words, we show that in addition all classical 
approximation methods based on sparse expansions and on more general low-rank structures, 
can be emulated by neural networks.
Since the solutions of several classes of high-dimensional PDEs are 
precisely of this form (see, e.g., \cite{schwab2017deep}),
our approximation results can be directly applied to these problems to 
establish approximation rates for neural network approximations 
that do not suffer from the curse of dimensionality. Note that approximation results for functions of the form \eqref{eq:low_rank_form} have previously been considered in \cite{2017Schmidt-Hieber} in the context of statistical bounds for nonparametric regression.

Moreover, we remark that the networks realizing the product in \eqref{eq:low_rank_form} itself, have a connectivity scaling which is logarithmic in the accuracy $\eps^{-1}$. While we will, for our concrete example, only obtain a spectral connectivity scaling, i.e. like $\eps^{-{\frac{1}{n}}}$ for any $n\in\N$ with the implicit constant depending on $n$, this tensor construction may be used to obtain logarithmic scaling (w.r.t.\@ the accuracy) for $d$-variate functions in cases where the univariate $h_j^s$ can be approximated with a logarithmic scaling.

As a particular application of the tools developed in the present paper,
we provide a mathematical analysis of the rates of expressive power of neural networks
for a particular, high-dimensional PDE which arises in mathematical finance,
namely the pricing of a so-called 
\emph{European maximum Option} (see, e.g., \cite{EuroOpt}).

We consider the particular (and not quite realistic) situation that 
the log-returns of these $d$ assets are uncorrelated, i.e. their log-returns
evolve according to $d$ uncorrelated drifted scalar diffusion processes.

The price of the European maximum Option on this basket of $d$ assets 
can then be obtained as solution of 
the multivariate Black-Scholes equation which reads, 
for the presently considered case of uncorrelated assets,
as

\begin{equation} \label{eq:KolmogEqn}
\textstyle
( \tfrac{ \partial }{ \partial t } u )(t,x)
+
\tfrac{ \mu }{ 2 }
\sum\limits_{ i = 1 }^d
x_i
\big( 
\tfrac{ \partial }{ \partial x_i } u 
\big)(t,x) 
+
\tfrac{ \sigma^2 }{ 2 }
\sum\limits_{ i = 1 }^d
| x_i |^2
\big( \tfrac{ \partial^2 }{ \partial x_i^2 } u \big)(t,x) 
= 0 \;.
\end{equation}
For the European maximum option, \eqref{eq:KolmogEqn} 
is completed with the \emph{terminal condition}
\begin{equation} \label{eq:TermCond}
u(T,x) = 
\varphi(x)
=
\max\{ x_1 - K_1, x_2 - K_2, \dots, x_d - K_d , 0 \}
\end{equation}
for $ x = ( x_1, \dots, x_d ) \in (0,\infty)^d $.
It is well known (see, e.g., \cite{FreidlinBook,hairer} and the references there) that there exists a unique solution of \eqref{eq:KolmogEqn}-\eqref{eq:TermCond}. This solution can be expressed as conditional expectation 
of the function $\varphi(x)$ in \eqref{eq:TermCond}
over suitable sample paths of a $d$-dimensional diffusion.

One main result of this paper is the following result
(stated with completely detailed assumptions below as \Cref{NNMain}),
on expression rates of deep neural networks for the
basket option price $u(0,x)$ for $x\in [a,b]^d$ for some $0<a<b< \infty$.
To render their dependence on the number $d$ of assets in the basket
explicit, we write $u_d$ in the statement of the theorem.

\begin{theorem}
Let 
$n\in\N$, $\mu\in\R$, $T,\sigma,a\in(0,\infty)$, $b\in(a,\infty)$, $(K_i)_{i\in\N}\subseteq[0,\Kmax)$,
and let $u_d\colon(0,\infty)\times[a,b]^d\to\R$, $d\in\N$, 
be the functions which satisfy for every $d\in\N$, and for every $(t,x) \in [0,T]\times (0,\infty)^d$ the equation \eqref{eq:KolmogEqn} with terminal condition \eqref{eq:TermCond}.\\
Then there exist neural networks 
$(\Gamma_{d,\eps})_{\eps\in(0,1],d\in\N}$ which satisfy
\begin{enumerate}[(i)]
    \item
	$\displaystyle\sup_{\eps\in(0,1],d\in\N}\brac{\frac{\L(\Gamma_{d,\eps})}{\max\{1,\ln(d)\}\pa{|\ln(\eps)|+\ln(d)+1}}}<\infty$,
\item
$\displaystyle\sup_{\eps\in(0,1],d\in\N}\brac{\frac{\M(\Gamma_{d,\eps})}{d^{2+\frac{1}{n}}\eps^{-\frac{1}{n}}}}<\infty$, and
	\item
        for every $\eps\in(0,1]$, $d\in\N$,
	\begin{align}
	\sup_{x\in[a,b]^d}\abs{u_d(0,x)-\brac{R_{\mathrm{ReLU}}(\Gamma_{d,\eps})}\!(x)}\leq\eps.
	\end{align}
	\end{enumerate}
\end{theorem} 
Informally speaking, the previous result states that the 
price of a $d$ dimensional European maximum option can, for every $n\in \N$, 
be expressed on cubes $[a,b]^d$ by deep neural networks 
to pointwise accuracy $\eps>0$ with network size bounded 
as $\mathcal{O}(d^{2+1/n} \eps^{-1/n})$ for arbitrary, fixed 
$n\in \N$ and with the constant implied in $\mathcal{O}(\cdot)$ 
independent of $d$ and of $\eps$ (but depending on $n$). In other words, the price of a European maximum option 
on a basket of $d$ assets can be approximated 
(or ``expressed'') by deep ReLU networks 
\emph{with spectral accuracy and without curse of dimensionality}.

The proof of this result is based on a near explicit expression 
for the function $u_d(0,x)$  (see \Cref{sec:formula}).
It uses this expression in conjunction with regularity estimates
in  \Cref{sec:regularity} and 
a neural network quadrature calculus and corresponding error estimates 
(which is of independent interest) in \Cref{QuadSection} to show that the 
function $u_d(0,x)$ possesses an approximate low-rank representation consisting 
of tensor products of cumulative normal distribution functions (\Cref{QM}) 
to which the low-rank approximation result mentioned above can be applied.

Related results have been shown in the recent work \cite{GHJv18_2680} which proves (by completely different methods) that solutions to general Kolmogorov equations with affine drift and diffusion terms can be approximated by neural networks of a size that scales polynomially in the dimension and the reciprocal of the desired accuracy as measured by the $L^p$ norm with respect to a given probability measure. The approximation estimates developed in the present paper only apply to the European maximum option pricing problem for uncorrelated assets but hold with respect to the much stronger $L^\infty$ norm and provide spectral accuracy in $\epsilon$ (as opposed to a low-order polynomial rate obtained in \cite{GHJv18_2680}), which is a considerable improvement. In summary, compared to \cite{GHJv18_2680}, the present paper treats a more restricted problem but achieves stronger approximation results. 

In order to give some context to our approximation results, we remark that solutions to Kolmogorov PDEs may, under reasonable assumptions, be approximated by empirical risk minimization over a neural network hypothesis class. The key here is the Feynman-Kac formula which allows to write the solution to the PDE as the expectation of an associated stochastic process. This expectation can be approximated by Monte-Carlo integration, i.e. one can view it as a neural network training problem where the data is generated by Monte-Carlo sampling methods which, under suitable conditions, are capable of avoiding the curse of dimensionality. For more information on this we refer to \cite{BGJ18_2679}.

While we admit that the European maximum option pricing problem for uncorrelated assets constitutes a rather special problem, the proofs in this paper develop several novel deep neural network approximation results of independent interest that can be applied to more general settings where a low-rank structure is implicit in high-dimensional problems. 
For mostly numerical results on machine learning for pricing American options we refer to \cite{AmericanOption}. 
Lastly we note that after a first preprint of the present paper was submitted, a number of research articles related to this work have appeared \cite{gonon2019uniform,GS20_2862,grohs2020deep,grohs2019deep,hornung2020spacetime,hutzenthaler2019proof,jentzen2018proof,kutyniok2019theoretical,reisinger2019rectified}.

\subsection{Outline}  \label{sec:Outline}
The structure of this article is as follows. 
The following \Cref{sec:formula} provides a derivation of
the semi-explicit formula for the price of European maximum options in a 
standard Black-Scholes setting. 
This formula consists of an integral of a tensor product function. 
In \Cref{sec:regularity} we develop some auxiliary regularity results 
for the cumulative normal distribution that are of independent interest which
will be used later on. 
In \Cref{QuadSection} we show that the integral appearing in the formula of 
\Cref{sec:formula} can be efficiently approximated by numerical quadrature.
\Cref{sec:relucalc} introduces some basic facts related to deep ReLU networks 
and \Cref{sec:basicexpr} develops basic approximation results for the approximation 
of functions which possess a tensor product structure. 
Finally, in \Cref{sec:optionapprox} we show our main result, namely a 
spectral approximation rate for the approximation of European maximum options by 
deep ReLU networks without curse of dimensionality. 
In \Cref{appendix} we collect some auxiliary proofs.
\section{High-dimensional derivative pricing}
\label{sec:formula}
In this section, we briefly review the Black-Scholes differential equation \eqref{eq:KolmogEqn} which arises, among others, as Kolmogorov equation
for multivariate geometric Brownian Motion. This linear, parabolic equation is,
for one particular type of financial contracts (so-called 
``European maximum option'' on a basket of $d$ stocks whose log-returns
are assumed for simplicity as mutually uncorrelated) endowed with the terminal condition \eqref{eq:TermCond} and solved for $(t,x)\in [0,T]\times (0,\infty)^d$.

\begin{prop}
\label{prop:European_max_option}
Let 
$ d \in \N $,
$ \mu \in \R $,
$ \sigma, T, K_1, \dots, K_d, \xi_1, \dots, \xi_d \in (0,\infty) $,
let
$ ( \Omega, \mathcal{F}, \P ) $ be a probability space,
and 
let $ W = ( W^{ (1) }, \dots, W^{ (d) } ) \colon [0,T] \times \Omega \to \R^d $
be a standard Brownian motion and let $u\in C([0,T]\times (0,\infty)^d)$ satisfy \eqref{eq:KolmogEqn} and \eqref{eq:TermCond}.
Then for $x = (\xi_1,\dots , \xi_d)\in (0,\infty)^d$ it holds that
\begin{equation}
\begin{split}
&
  u(0,x) = \E\!\left[ 
    \max_{ i \in \{ 1, 2, \dots, d \} }
    \left(
      \max\!\left\{ 
        \exp\!\big(
          \big[ 
            \mu - \tfrac{ \sigma^2 }{ 2 } 
          \big] T 
          +
          \sigma W_T^{ (i) }
        \big)
        \,
        \xi_i
        - K_i 
        ,
        0
      \right\}
    \right)
  \right]
\\ &
=
  \int_0^{ \infty }
  1 
  - 
  \left[
  \textstyle
  \prod\limits_{ i = 1 }^d
  \left(
  \int_{
    - \infty
  }^{
    \frac{ 1 }{ \sigma \sqrt{ T } }
    \left[
      \ln\left(
        \frac{ 
          y + K_i
        }{
          \xi_i
        }
      \right)
      -
      \left(
        \mu - [ \nicefrac{ \sigma^2 }{ 2 } ] 
      \right)
      T 
    \right]
  }
  \tfrac{ 1 }{ \sqrt{ 2 \pi } }
  \,
  \exp\!\left(
    -
    \frac{ r^2 }{ 2 }
  \right)
  dr
  \right)
  \right]
  dy
  .
\end{split}
\end{equation}
\end{prop}

For the proof of this Proposition, we require the following well-known result.

\begin{lemma}[Complementary distribution function formula]
\label{lem:moment_formula}
Let $ \mu \colon \mathcal{B}( [0,\infty) ) \to [0,\infty] $ be a sigma-finite measure.
Then
\begin{equation}
  \int_0^{ \infty } x \, \mu( dx )
  =
  \int_0^{ \infty }
  \mu( [x,\infty) )
  \, dx
  .
\end{equation}
\end{lemma}
We are now in position to provide a proof of Proposition~\ref{prop:European_max_option}.
\begin{proof}[Proof of Proposition~\ref{prop:European_max_option}]
The first equality follows directly from the Feynman-Kac formula \cite[Corollary 4.17]{hairer}.
We proceed with a proof of the second equality.
Throughout this proof let $X_i \colon \Omega \to \R$, $i \in \{ 1, 2, \dots, d \}$,
be random variables which satisfy for every $ i \in \{ 1, 2, \dots, d \} $ 
\begin{equation}
  X_i
  =
      \exp\!\big(
        \big[ 
          \mu - \tfrac{ \sigma^2 }{ 2 } 
        \big] T 
        +
        \sigma W_T^{ (i) }
      \big)
      \,
      \xi_i
\end{equation}
and let $ Y \colon \Omega \to \R $ be the random variable given by
\begin{equation}
  Y =     
  \max\{ 
      X_1 - K_1
      ,
      \dots 
      ,
      X_d - K_d
      , 0
    \}
  .
\end{equation}
Observe that for every $ y \in (0,\infty) $ it holds 
\begin{equation}
\begin{split}
  \P\!\left(
    Y \geq y
  \right)
& =
  1 - \P\!\left( Y < y \right)
  =
  1 -
  \P\!\left(
    \max_{ i \in \{ 1, 2, \dots, d \} }
    \left(
      X_i - K_i
    \right)
    < y
  \right)
\\ &
=
  1 -
  \P\!\left(
    \cap_{ i \in \{ 1, 2, \dots, d \} }
    \left\{
      X_i - K_i
      < y
    \right\}
  \right)
  =
  1 
  - 
  \textstyle
  \prod\limits_{ i = 1 }^d
  \P\!\left(
    X_i - K_i
    < y
  \right)
\\ &
=
  1 
  - 
  \textstyle
  \prod\limits_{ i = 1 }^d
  \P\!\left(
    X_i 
    < y + K_i
  \right)
\\ & 
=
  1 
  - 
  \textstyle
  \prod\limits_{ i = 1 }^d
  \P\!\left(
    \exp\!\big(
      \big[ 
        \mu - \tfrac{ \sigma^2 }{ 2 } 
      \big] T 
      +
      \sigma W_T^{ (i) }
    \big)
    \,
    \xi_i
    < y + K_i
  \right)
  .
\end{split}
\end{equation}
Hence, we obtain that for every $ y \in (0,\infty) $ it holds 
\begin{equation}
\begin{split}
  \P\!\left(
    Y \geq y
  \right)
& =
  1 
  - 
  \textstyle
  \prod\limits_{ i = 1 }^d
  \P\!\left(
    \exp\!\big(
      \big[ 
        \mu - \tfrac{ \sigma^2 }{ 2 } 
      \big] T 
      +
      \sigma W_T^{ (i) }
    \big)
    < 
    \frac{ 
      y + K_i
    }{
      \xi_i
    }
  \right)
\\
& =
  1 
  - 
  \textstyle
  \prod\limits_{ i = 1 }^d
  \P\!\left(
      \sigma W_T^{ (i) }
    < 
    \ln\!\left(
      \frac{ 
        y + K_i
      }{
        \xi_i
      }
    \right)
    -
      \big[ 
        \mu - \tfrac{ \sigma^2 }{ 2 } 
      \big] T 
  \right)
\\
& =
  1 
  - 
  \textstyle
  \prod\limits_{ i = 1 }^d
  \P\!\left(
    \frac{ 1 }{ \sqrt{T} } 
    W_T^{ (i) }
    < 
    \frac{ 1 }{ \sigma \sqrt{ T } }
    \left[
    \ln\!\left(
      \frac{ 
        y + K_i
      }{
        \xi_i
      }
    \right)
    -
      \big[ 
        \mu - \tfrac{ \sigma^2 }{ 2 } 
      \big] T 
    \right]
  \right)
  .
\end{split}
\end{equation}
This shows that for every $ y \in (0,\infty) $ it holds 
\begin{equation}
\begin{split}
  \P\!\left(
    Y \geq y
  \right)
& =
  1 
  - 
  \left[
  \textstyle
  \prod\limits_{ i = 1 }^d
  \left(
  \int_{
    - \infty
  }^{
    \frac{ 1 }{ \sigma \sqrt{ T } }
    \left[
      \ln\left(
        \frac{ 
          y + K_i
        }{
          \xi_i
        }
      \right)
      -
      \left(
        \mu - [ \nicefrac{ \sigma^2 }{ 2 } ] 
      \right)
      T 
    \right]
  }
  \tfrac{ 1 }{ \sqrt{ 2 \pi } }
  \,
  \exp\!\left(
    -
    \frac{ r^2 }{ 2 }
  \right)
  dr
  \right)
  \right]
  .
\end{split}
\end{equation}
Combining this with Lemma~\ref{lem:moment_formula} 
completes the proof of Proposition~\ref{prop:European_max_option}.
\end{proof}
With Lemma~\ref{lem:moment_formula} and Proposition~\ref{prop:European_max_option}, we may write 
\begin{equation}
\begin{split}
  u(0,x) 
  =
  \E\!\left[ 
    \varphi\!\left( 
      \exp\!\left( 
        \left[ \mu - \nicefrac{ \sigma^2 }{ 2 } \right] T 
        + 
        \sigma W^{ (1) }_T
      \right)
      x_1
      ,
      \ldots 
      ,
      \exp\!\left( 
        \left[ \mu - \nicefrac{ \sigma^2 }{ 2 } \right] T 
        + 
        \sigma W^{ (d) }_T
      \right)
      x_d
    \right)
  \right]
\end{split}
\end{equation}
(``semi-explicit'' formula). 
Let us consider the case 
$ \mu = \sigma^2 / 2 $, $ T = \sigma = 1 $, and $ K_1 = \ldots = K_d = K \in (0,\infty) $.
Then for every $ x = ( x_1, \dots, x_d) \in (0,\infty)^d $ 
\begin{equation}\label{explicitForm}
\begin{split}
  u(0,x) 
  &=
  \E\!\left[ 
    \varphi\!\left( 
      e^{
        W^{ (1) }_T
      }
      x_1
      ,
      \ldots 
      ,
      e^{
        W^{ (d) }_T
      }
      x_d
    \right)
  \right]
  =
  \E\!\left[ 
    \varphi\!\left( 
      e^{
        W^{ (1) }_1
      }
      x_1
      ,
      \ldots 
      ,
      e^{
        W^{ (d) }_1
      }
      x_d
    \right)
  \right]
\\ & 
  =
  \E\!\left[ 
    \max\!\left\{ 
      e^{
        W^{ (1) }_1
      }
      x_1
      - K
      ,
      \ldots 
      ,
      e^{
        W^{ (d) }_1
      }
      x_d
      - K
      , 0
    \right\}
  \right]
\\ & 
  =
  \int_0^{ \infty }
  1
  -
  \left[ 
    \prod_{ i = 1 }^d
    \int_{
      - \infty
    }^{ 
      \ln( \frac{ K + c }{ x_i } ) 
    }
    \tfrac{ 1 }{ \sqrt{ 2 \pi } }
    \,
    \exp\!\left( - \tfrac{ r^2 }{ 2 } \right)
    dr
  \right]
  dc
  .
\end{split}
\end{equation}

\section{Regularity of the Cumulative Normal Distribution}\label{sec:regularity}

Now that we have derived an semi-explicit formula for the solution, we establish 
regularity properties of the integrand function in \eqref{explicitForm}.
This will be required in order to approximate the multivariate
integrals by quadratures (which are subsequently realized by neural networks)
in Section $4$ and to apply the neural network results from Section $6$ to our problem.
To this end, we analyze the derivatives of the factors in the tensor product, 
which essentially are compositions of the cumulative normal distribution 
with the natural logarithm. 
As this function appears in numerous closed-form option pricing
    formulae (see, e.g., \cite{Kwod2ndEd}), the (Gevrey) type regularity estimates
    obtained in this section are of independent interest (they may, for example,
    also be used in the analysis of deep network expression rates
    and of spectral methods for option pricing).

\begin{lemma}\label{Prel0}
    Let $f\colon(0,\infty)\to\R$ be the function which satisfies for every $t\in(0,\infty)$ that 
    \begin{align}
      f(t)=\normfac\int^{\ln(t)}_{-\infty} e^{-\frac{1}{2}r^2}\d r,\label{Prel0fDef}
    \end{align}
    let $g_{n,k}\colon(0,\infty)\to\R$, $n,k\in\N_0$, be the functions 
    which satisfy for every $n,k\in\N_0$, $t\in(0,\infty)$ that
    \begin{align}
      g_{n,k}(t)=t^{-n}\elnt[\ln(t)]^k,\label{Prel0gnkDEF}
    \end{align}    
    and let $(\gamma_{n,k})_{n,k\in\Z}\subseteq\Z$ be the integers which satisfy for every $n,k\in\Z$ that 
    \begin{equation}
      \gamma_{n,k} =
      \begin{cases}1 & \colon n=1,k=0\\ -\gamma_{n-1,k-1}-(n-1)\gamma_{n-1,k}+(k+1)\gamma_{n-1,k+1} & \colon n>1, 0\leq k<n\\ 0 & \colon \mathrm{else}
      \end{cases}.
      \label{Prel0gammankDef}
    \end{equation}  
    Then it holds for every $n\in\N$ that
    \begin{enumerate}[(i)]
      \item\label{P0i} we have that $f$ is $n$-times continuously differentiable and
      \item\label{P0ii} we have for every $t\in(0,\infty)$ that
      \begin{align} 
	f^{(n)}(t)=\normfac\brac{\sum_{k=0}^{n-1}\gamma_{n,k}\,g_{n,k}(t)}.\label{Prel0fn}
      \end{align}    
    \end{enumerate}
  \end{lemma}
  
  \begin{proof}[Proof of \Cref{Prel0}]
    We prove \eqref{P0i} and \eqref{P0ii} by induction on $n\in\N$. 
    For the base case $n=1$ note that 
    \eqref{Prel0fDef}, 
    \eqref{Prel0gnkDEF}, 
    \eqref{Prel0gammankDef}, 
    the fact that the function $\R\ni r\mapsto e^{-\frac{1}{2}r^2}\in(0,\infty)$ is continuous, the fundamental theorem of calculus, and the chain rule yield
    \begin{enumerate}[(A)]
      \item that $f$ is differentiable and
      \item that for every $t\in(0,\infty)$ it holds
      \begin{align}
	f'(t)=\normfac \, \elnt t^{ - 1 } 
	= \normfac \, g_{1,0}(t)
	= \normfac \, \gamma_{ 1, 0 } \, g_{1,0}(t)
	.
	\label{fprime}
      \end{align}
    \end{enumerate}
    This establishes \eqref{P0i} and \eqref{P0ii} in the base case $n=1$.
    For the induction step ${\N\ni n\to n+1\in\{2,3,4,\dots\}}$ note that for every $t\in(0,\infty)$ we have
    \begin{align} 
      \ddt\brac{\elnt} = -t^{-1}\elnt\ln(t).
    \end{align} 
    Combining this and \eqref{Prel0gnkDEF} with the product rule establishes 
    for every $n\in\N$, $k \in \{ 0, 1, \dots, n - 1 \} $, $ t \in (0,\infty) $ that 
    \begin{align}
      \begin{split}\label{gnk}
	  (g_{n,k})'(t)&=\ddt\brac{t^{-n}\elnt[\ln(t)]^k}\\
	  &=-nt^{-(n+1)}\elnt[\ln(t)]^k-t^{-(n+1)}\elnt[ \ln(t) ]^{k+1}\\
	  &\quad +t^{-(n+1)}\elnt k[\ln(t)]^{ \max\{ k - 1 , 0 \} }
	  \\
	  &
	  =-g_{n+1,k+1}(t)-ng_{n+1,k}(t)+kg_{n+1,\max\{ k - 1 , 0 \} }(t).
      \end{split}
    \end{align}
    Hence, we obtain that for every $n\in\N$, $t\in(0,\infty)$ it holds
    \begin{align}
      \begin{split}\label{Prel0NatInd}
	&\sum_{k=0}^{n-1}\gamma_{n,k}(g_{n,k})'(t)\\
	=&\sum_{k=0}^{n-1}\brac{\gamma_{n,k}\pa{-g_{n+1,k+1}(t)-ng_{n+1,k}(t)+kg_{n+1,\max\{ k-1, 0 \} }(t)}}\\
	=&\sum_{k=0}^{n-1}-\gamma_{n,k}\,g_{n+1,k+1}(t)+\sum_{k=0}^{n-1}-n\gamma_{n,k}\,g_{n+1,k}(t)+\sum_{k=1}^{n-1}k\gamma_{n,k}\,g_{n+1,\max\{ k-1, 0 \} }(t)\\
	=&\sum_{k=1}^{n}-\gamma_{n,k-1}\,g_{n+1,k}(t)+\sum_{k=0}^{n-1}-n\gamma_{n,k}\,g_{n+1,k}(t)+\sum_{k=0}^{n-2}(k+1)\gamma_{n,k+1}\,g_{n+1,k}(t).
      \end{split}
    \end{align}
    The fact that for every $n\in\N$ it holds that $\gamma_{n,-1}=\gamma_{n,n}=\gamma_{n,n+1}=0$ 
   and \eqref{Prel0gammankDef} therefore ensure that for every $n\in\N$, $t\in(0,\infty)$ we have
    \begin{align}
      \begin{split}
	\sum_{k=0}^{n-1}\gamma_{n,k}(g_{n,k})'(t)&=\sum_{k=0}^{n}\brac{\pa{-\gamma_{n,k-1}-n\gamma_{n,k}+(k+1)\gamma_{n,k+1}}g_{n+1,k}(t)}\\
	&=\sum_{k=0}^{n}\gamma_{n+1,k}\,g_{n+1,k}(t).
      \end{split}
    \end{align}
    Induction thus establishes \eqref{P0i} and \eqref{P0ii}. 
    The proof of \Cref{Prel0} is thus completed.
  \end{proof} 
Using the recursive formula from above we can now bound the derivatives of $f$. 
Note that the supremum of $f^{(n)}$ is actually attained on the interval $[e^{-4n},1]$ 
and scales with $n$ like $e^{(cn^2)}$ for some $c\in(0,\infty)$.
This can directly be seen by calulating the maximum of the $g_{n,k}$ from \eqref{Prel0gnkDEF}. 
For our purposes, however, it is sufficient to establish that all derivatives of 
$f$ are bounded on $(0,\infty)$.
  
  \begin{lemma}\label{Prel1}
Let $f\colon(0,\infty)\to\R$ be the function which satisfies 
for every $t\in(0,\infty)$ that
\begin{align}
      f(t)=\normfac\int^{\ln(t)}_{-\infty} e^{-\frac{1}{2}r^2}\d r.
    \end{align} 
    Then it holds for every $n\in\N$ that
    \begin{align}
      \sup_{t\in(0,\infty)}\abs{f^{(n)}(t)}
      \leq\max\!\left\{(n-1)!\,2^{n-2}\, ,\sup_{t\in [e^{-4n},1]}\abs{f^{(n)}(t)}\right\}
     <\infty.
    \end{align}      
  \end{lemma}
  
  \begin{proof}[Proof of \Cref{Prel1}]
    Throughout this proof let
    $g_{n,k}\colon(0,\infty)\to\R$, $n,k\in\N_0$, be 
    the functions introduced in \eqref{Prel0gnkDEF}    
    and let $(\gamma_{n,k})_{n,k\in\Z}\subseteq\Z$ be the integers introduced in \eqref{Prel0gammankDef}.
    Then \Cref{Prel0} shows for every $n\in\N$ that
    \begin{enumerate}[(a)]
      \item\label{Prel1a} we have that $f$ is $n$-times continuously differentiable and
      \item we have for every $t\in(0,\infty)$ that
      \begin{align} 
	f^{(n)}(t)=\normfac\brac{\sum_{k=0}^{n-1}\gamma_{n,k}\,g_{n,k}(t)}.\label{Prel1fn}
      \end{align}    
    \end{enumerate}
    In addition, observe that for every $m\in\N$, $t\in(0,e^{-2m}]$ holds 
    ${\tfrac{1}{2}\ln(t)\leq-m}$. 
    This ensures that for every $m\in\N$, $t\in(0,e^{-2m}]\subseteq(0,1]$ we have
    \begin{align} 
      \begin{split}
	\elnt&=e^{\brac{\ln(t)(-\frac{1}{2}\ln(t))}}=\brac{e^{\ln(t)}}^{-\frac{1}{2}\ln(t)}
	=t^{-\frac{1}{2}\ln(t)}=\pa{\tfrac{1}{t}}^{\frac{1}{2}\ln(t)}
	\leq\pa{\tfrac{1}{t}}^{-m}=t^m. \label{elnSquaredEst1}
      \end{split}
    \end{align}
    Moreover, note that the fundamental theorem of calculus implies for every $t\in(0,1]$ that
    \begin{align}
      \begin{split}
	\abs{\ln(t)}&=\abs{\ln(t)-\ln(1)}=\abs{\ln(1)-\ln(t)}
	=\abs{\int_t^1 \frac{1}{s}\,\d s}\leq\abs{\frac{1}{t}(1-t)}\leq t^{-1}.   
      \label{lnEst1}
      \end{split}
    \end{align}
    Combining \eqref{Prel0gnkDEF}, \eqref{Prel1fn}, and \eqref{elnSquaredEst1} therefore establishes 
    that for every $n\in\N$, ${t\in(0,e^{-4n})}{\subseteq(0,1]}$ it holds
    \begin{align}
      \begin{split}\label{an1}
	\abs{f^{(n)}(t)}&=\normfac\abs{\sum_{k=0}^{n-1}\gamma_{n,k}\,g_{n,k}(t)}=\normfac\abs{\sum_{k=0}^{n-1}\gamma_{n,k}t^{-n}\elnt[\ln(t)]^k}\\
	&\leq\normfac\brac{\sum_{k=0}^{n-1}\abs{\gamma_{n,k}}t^{n-k}}\leq\normfac\brac{\sum_{k=0}^{n-1}\abs{\gamma_{n,k}}}.
      \end{split}
    \end{align}
In addition, observe that the fundamental theorem of calculus ensures that for every $t\in[1,\infty)$ we have
    \begin{align} 
      \abs{\ln(t)}=\abs{\ln(t)-\ln(1)}=\abs{\int_1^t \frac{1}{s}\,\d s}\leq\abs{t-1}\leq t.\label{lnEst2}
    \end{align}
    This, \eqref{Prel0gnkDEF}, \eqref{Prel1fn}, and the fact that for every $t\in(0,\infty)$ it holds $|\elnt|\leq1$ imply that for every $n\in\N$, $t\in(1,\infty)$ we have
    \begin{align}
      \begin{split}\label{an2}
	\abs{f^{(n)}(t)}&=\normfac\abs{\sum_{k=0}^{n-1}\gamma_{n,k}\,g_{n,k}(t)}
                         =\normfac\abs{\sum_{k=0}^{n-1}\gamma_{n,k}t^{-n}\elnt[\ln(t)]^k}
      \\
	&\leq\normfac\brac{\sum_{k=0}^{n-1}\abs{\gamma_{n,k}}t^{-n}\abs{\ln(t)}^k}
         \leq\normfac\brac{\sum_{k=0}^{n-1}\abs{\gamma_{n,k}}t^{-n}t^k}\\
	&=\normfac\brac{\sum_{k=0}^{n-1}\abs{\gamma_{n,k}}t^{-n+k}}
         \leq\normfac\brac{\sum_{k=0}^{n-1}\abs{\gamma_{n,k}}}.
      \end{split}
    \end{align}
    Moreover, observe that \eqref{Prel1a} assures that for every $n\in\N$ it holds that the function $f^{(n)}$ is continuous.
    This and the boundedness of the set $[e^{-4n},1]$ ensure that for every $n\in\N$ we have
    \begin{align} 
      \sup_{t\in [e^{-4n},1]}\abs{f^{(n)}(t)}<\infty. \label{bn}
    \end{align}
    Combining this with \eqref{an1} and \eqref{an2} establishes that for every $n\in\N$ we have
    \begin{align} 
      \sup_{t\in(0,\infty)}\abs{f^{(n)}(t)}
      \leq
      \max\!\left\{\normfac\brac{\sum_{k=0}^{n-1}\abs{\gamma_{n,k}}},
               \sup_{t\in [e^{-4n},1]}\abs{f^{(n)}(t)}\right\}<\infty.
    \label{Prel1CombEst}
    \end{align}
    Furthermore, note that \eqref{Prel0gammankDef} implies that for every $n\in\{2,3,4,\dots\}$ it holds
    \begin{align}
      \begin{split}\label{gammankEq}
       \sum_{k=0}^{n-1}\abs{\gamma_{n,k}}
          &=\sum_{k=0}^{n-1}\abs{-\gamma_{n-1,k-1}-(n-1)\gamma_{n-1,k}+(k+1)\gamma_{n-1,k+1}}
        \\
	&\leq\brac{\sum_{k=0}^{n-1}\abs{\gamma_{n-1,k-1}}}
           +\brac{\sum_{k=0}^{n-1}(n-1)\abs{\gamma_{n-1,k}}}+\brac{\sum_{k=0}^{n-1}(k+1)\abs{\gamma_{n-1,k+1}}}\\
	&=\brac{\sum_{k=-1}^{n-2}\abs{\gamma_{n-1,k}}}+\brac{\sum_{k=0}^{n-1}(n-1)\abs{\gamma_{n-1,k}}}
                                     +\brac{\sum_{k=1}^{n}k\abs{\gamma_{n-1,k}}}.
      \end{split}
    \end{align}
    Combining this with the fact that for every $n\in\{2,3,4,\dots\}$, $k\in\Z\backslash\{0,1,\dots,n-2\}$ we have $\gamma_{n-1,k}=0$ 
    implies that for every $n\in\{2,3,4,\dots\}$ it holds
    \begin{align}
     \sum_{k=0}^{n-1}\abs{\gamma_{n,k}}
    =\sum_{k=0}^{n-2}\brac{(1+(n-1)+k)\abs{\gamma_{n-1,k}}}
     \leq (2n-2)\brac{\sum_{k=0}^{n-2}\abs{\gamma_{n-1,k}}}=2(n-1)\brac{\sum_{k=0}^{n-2}\abs{\gamma_{n-1,k}}}.
    \end{align}
    The fact that $\gamma_{1,0}=1$ hence implies that for every $n\in\N$ we have
    \begin{align}
      \sum_{k=0}^{n-1}\abs{\gamma_{n,k}}\leq (n-1)!\,2^{n-1}\brac{\sum_{k=0}^0\abs{\gamma_{1,k}}}=(n-1)!\,2^{n-1}.
    \end{align}
    Combining this and \eqref{Prel1CombEst} ensures that for every $n\in\N$ it holds
    \begin{align}
      \sup_{t\in(0,\infty)}\abs{f^{(n)}(t)}
     \leq\max\!\left\{\normfac(n-1)!\,2^{n-1}\, ,\sup_{t\in [e^{-4n},1]}\abs{f^{(n)}(t)}\right\}
     <\infty.
    \end{align} 
    The proof of \Cref{Prel1} is thus completed.
  \end{proof}
  In the following corollary we estimate the derivatives of the function $x\to f(\tfrac{K+c}{x})$ 
  required to approximate this function by neural networks.
  \begin{cor}\label{PrelNN1}
    Let $n\in\N$, $K\in[0,\infty)$, $c,a\in(0,\infty)$, $b\in(a,\infty)$, let $f\colon(0,\infty)\to\R$ be the function which satisfies for every $t\in(0,\infty)$ that
    \begin{align}
      f(t)=\normfac\int^{\ln(t)}_{-\infty} e^{-\frac{1}{2}r^2}\d r,
    \end{align}
    and let $h\colon[a,b]\to\R$ be the function which satisfies for every $x\in [a,b]$ that
    \begin{align}
      h(x)=f(\tfrac{K+c}{x}).
    \end{align}
    Then it holds
    \begin{enumerate}[(i)]
      \item that $f$ and $h$ are infinitely often differentiable and	
      \item that
      \small
      \begin{align}
       \max_{k\in\zeroto{n}}\sup_{x\in[a,b]}\abs{h^{(k)}\!(x)}
       \leq n2^{n-1} n! \brac{\max_{k\in\zeroto{n}} 
            \sup_{t\in[\frac{K+c}{b},\frac{K+c}{a}]}\abs{f^{(k)}\!(t)}}\max\{a^{-2n},1\}\max\{(K+c)^n,1\}.
      \end{align}
      \normalsize
    \end{enumerate}
  \end{cor}
  
  \begin{proof}[Proof of \Cref{PrelNN1}]
Throughout this proof let $\alpha_{m,j}\in\Z$, $m,j\in\Z$, be the integers which satisfy that for every $m,j\in\Z$ it holds 
    \begin{align}\label{PNN1alpha}
      \alpha_{m,j}=\begin{cases}-1 & \colon m=j=1\\ -(m-1+j)\alpha_{m-1,j}-\alpha_{m-1,j-1} 
               & \colon m>1,\,\, 1\leq j\leq m \\ 0 & \colon \mathrm{else}
     \end{cases}.
    \end{align}
    Note that \Cref{Prel0} and the chain rule ensure that the functions $f$ and $h$ are infinitely often differentiable.
    Next we claim that for every $m\in\N$, $x\in[a,b]$ it holds
    \begin{align}\label{PrelNN1claim}
      h^{(m)}\!(x)=\tfrac{\d^m}{\d x^m}\!\pa{f(\tfrac{K+c}{x})}
      =\sum_{j=1}^m \alpha_{m,j}(K+c)^j x^{-(m+j)}(f^{(j)}\!\big(\tfrac{K+c}{x})\big).
    \end{align}
    We prove \eqref{PrelNN1claim} by induction on $m\in\N$. 
    To prove the base case $m=1$ we note that the chain rule ensures that
    for every $x\in[a,b]$ we have
    \begin{align}
      \tfrac{\d}{\d x}\!\pa{f(\tfrac{K+c}{x})}=-(K+c)x^{-2}\!\pa{f'\!(\tfrac{K+c}{x})}=\alpha_{1,1}(K+c)x^{-2}\!\pa{f'\!(\tfrac{K+c}{x})}.
    \end{align}
    This establishes \eqref{PrelNN1claim} in the base case $m=1$. 
    For the induction step $\N\ni m \to m+1\in\N$ observe that the chain rule implies 
    for every $m\in\N$, $x\in[a,b]$ that
    \small
    \begin{align}\begin{split}
      &\quad\tfrac{\d}{\d x}\!\brac{\sum_{j=1}^m \alpha_{m,j}(K+c)^j x^{-(m+j)}\!\pa{f^{(j)}\!(\tfrac{K+c}{x})}}\\
      &=-\brac{\sum_{j=1}^m \alpha_{m,j}(K+c)^{j+1}x^{-(m+j+2)}\!\pa{f^{(j+1)}\!(\tfrac{K+c}{x})}}-\brac{\sum_{j=1}^m \alpha_{m,j}(K+c)^j(m+j)x^{-(m+j+1)}\!\pa{f^{(j)}\!(\tfrac{K+c}{x})}}\\
      &=-\brac{\sum_{j=2}^{m+1} \alpha_{m,j-1}(K+c)^{j}x^{-(m+j+1)}\!\pa{f^{(j)}\!(\tfrac{K+c}{x})}}-\brac{\sum_{j=1}^m \alpha_{m,j}(K+c)^j (m+j)x^{-(m+j+1)}\!\pa{f^{(j)}\!(\tfrac{K+c}{x})}}\\
      &=\sum_{j=1}^{m+1}(-(m+j)\alpha_{m,j}-\alpha_{m,j-1})(K+c)^jx^{-(m+1+j)}\!\pa{f^{(j)}\!(\tfrac{K+c}{x})}.
    \end{split}\end{align}
    \normalsize
    Induction thus establishes \eqref{PrelNN1claim}. 
    Next note that \eqref{PNN1alpha} ensures that for every $m\in\{2,3,\dots\}$ it holds
    \begin{align}\begin{split}
     \max_{j\in\oneto{m}}\abs{\alpha_{m,j}}&=\max_{j\in\oneto{m}}\abs{-(m-1+j)\alpha_{m-1,j}-\alpha_{m-1,j-1}}\\
     &\leq\brac{\max_{j\in\oneto{m-1}}\abs{(m-1+j)\alpha_{m-1,j}}}+\brac{\max_{j\in\oneto{m-1}}\abs{\alpha_{m-1,j}}}\\
     &\leq(2m-1)\brac{\max_{j\in\oneto{m-1}}\abs{\alpha_{m-1,j}}}\leq2m\brac{\max_{j\in\oneto{m-1}}\abs{\alpha_{m-1,j}}}.
     \end{split}\end{align}
    Induction hence proves that for every $m\in\N$ we have $\max_{j\in\oneto{m}}\abs{\alpha_{m,j}}\leq 2^{m-1}m!$.
    Combining this with \eqref{PrelNN1claim} implies that
    for every $m\in\oneto{n}$, $x\in[a,b]$ we have
    \begin{align}\begin{split}
      \abs{h^{(m)}\!(x)}&=\abs{\sum_{j=1}^m \alpha_{m,j}(K+c)^j x^{-(m+j)}\!\big(f^{(j)}\!(\tfrac{K+c}{x})\big)}\\
      &\leq 2^{m-1}m! \brac{\max_{j\in\oneto{m}}\sup_{t\in[\frac{K+c}{b},\frac{K+c}{a}]}\abs{f^{(j)}\!(t)}}\max\{x^{-2m},1\}\brac{\sum_{j=1}^m (K+c)^j}\\
      &\leq m2^{m-1} m! \brac{\max_{j\in\oneto{m}}\sup_{t\in[\frac{K+c}{b},\frac{K+c}{a}]}\abs{f^{(j)}\!(t)}}\max\{x^{-2m},1\}\max\{(K+c)^m,1\}.      
    \end{split}\end{align}
    Combining this with the fact that $\sup_{x\in[a,b]}\abs{h(x)}=\sup_{t\in[\frac{K+c}{b},\frac{K+c}{a}]}\abs{f(t)}$ establishes that it holds
    \small
    \begin{align}
    \max_{k\in\zeroto{n}}\sup_{x\in[a,b]}\abs{h^{(k)}\!(x)}
    \leq n2^{n-1} n! \brac{\max_{k\in\zeroto{n}}\sup_{t\in[\frac{K+c}{b},\frac{K+c}{a}]}\abs{f^{(k)}\!(t)}}\max\{a^{-2n},1\}\max\{(K+c)^n,1\}.
    \end{align}
    \normalsize
    This completes the proof of \Cref{PrelNN1}.
  \end{proof}
  
  Next we consider the derivatives of the functions $c\mapsto f(\tfrac{K+c}{x_i})$, $i\in\oneto{d}$, 
  and their tensor product, which will be needed in order to approximate approximate 
  the outer integral in \eqref{explicitForm} by composite Gaussian quadrature.
  \begin{cor}\label{Prel3}
  Let $n\in\N$, $K\in[0,\infty)$, $x\in(0,\infty)$, let $f\colon(0,\infty)\to\R$ be the function which satisfies for every $t\in(0,\infty)$ that
    \begin{align}
      f(t)=\normfac\int^{\ln(t)}_{-\infty} e^{-\frac{1}{2}r^2}\d r,
    \end{align}
  and let $g\colon(0,\infty)\to\R$ be the function which satisfies for every $t\in(0,\infty)$ that
    \begin{align}
      g(t)=f\!\pa{\tfrac{K+t}{x}}.
    \end{align}
  Then it holds
    \begin{enumerate}[(i)]
      \item that $f$ and $g$ are infinitely often differentiable and
      \item that
      \begin{align}
	\sup_{t\in(0,\infty)}\abs{g^{(n)}(t)}
        \leq \left[\sup_{t\in(0,\infty)}\abs{f^{(n)}(t)}\right] \abs{x}^{-n}<\infty.
      \end{align}
    \end{enumerate}
  \end{cor}
  
  \begin{proof}[Proof of \Cref{Prel3}]
  Combining \Cref{Prel1} with the chain rule implies that for every $t\in(0,\infty)$ it holds
  \begin{align} 
  \abs{g^{(n)}(t)}=\abs{\tfrac{\d^n}{\d t^n}\big(f(\tfrac{K+t}{x})\big)}=\abs{f^{(n)}\!\pa{\tfrac{K+t}{x}}\tfrac{1}{x^n}}
        \leq \left[\sup_{t\in(0,\infty)}\abs{f^{(n)}(t)}\right] \abs{x}^{-n}<\infty.
  \end{align}
    This completes the proof of \Cref{Prel3}.
  \end{proof}

  \begin{lemma}\label{Prel2}
  Let $d,n\in\N$, $a\in(0,\infty)$, $b\in(a,\infty)$, 
  $K=(K_1,\dots,K_d)\in[0,\infty)^d$, $x=(x_1,\dots,x_d)\in[a,b]^d$, 
  let $f\colon(0,\infty)\to\R$ be the function which satisfies for every $t\in(0,\infty)$ that
    \begin{align}
      f(t)=\normfac\int^{\ln(t)}_{-\infty} e^{-\frac{1}{2}r^2}\d r,
    \end{align}
  and let $F\colon(0,\infty)\to\R$ be the function which satisfies for every $c\in(0,\infty)$ that
    \begin{align}
      F(c)=1-\brac{\tprod_{i=1}^d f\!\pa{\tfrac{K_i+c}{x_i}}}.\label{Prel2Fdef}
    \end{align}
    Then it holds
\begin{enumerate}[(i)]
\item 
 that $f$ and $F$ are infinitely often differentiable and
\item 
 that
\begin{align}
\sup_{c\,\in(0,\infty)}\abs{F^{(n)}(c)}
\leq \left[\max_{k\in\zeroto{n}}\sup_{t\in(0,\infty)}\abs{f^{(k)}(t)}\right]^n d^n a^{-n}<\infty.
\end{align}
\end{enumerate}
\end{lemma} 

\begin{proof}[Proof of \Cref{Prel2}]
Note that \Cref{Prel0} ensures that $f$ and $F$ are infinitely often differentiable. 
Moreover, observe that \eqref{Prel2Fdef} and the general Leibniz rule 
imply for every $c\in(0,\infty)$ that
\begin{align} 
\begin{split}\label{Prel2a}
F^{(n)}(c)&=-\tfrac{\d^n}{\d c^n}\left[\tprod_{i=1}^d f\!\pa{\tfrac{K_i+c}{x_i}}\right]\\
&=-\sum_{\substack{l_1,l_2,\dots,l_d\in\N_0,\\ \sum_{i=1}^d l_i=n}}\brac{\binom{n}{l_1,l_2,\dots,l_d}\tprod_{i=1}^d\pa{\tfrac{\d^{l_i}}{\d c^{l_i}}\brac{f\!\pa{\tfrac{K_i+c}{x_i}}}}}.
\end{split}
\end{align}
Next note that the fact that for every $r\in\R$ it holds that $e^{-\frac{1}{2}r^2}\geq 0$ ensures that 
\begin{align}
 \sup_{t\in(0,\infty)}\abs{f(t)}
= \sup_{t\in(0,\infty)}\abs{\normfac\int^{\ln(t)}_{-\infty} e^{-\frac{1}{2}r^2}\d r}=\abs{\normfac\int^{\infty}_{-\infty} 
     e^{-\frac{1}{2}r^2}\d r}=1.
    \end{align}
\Cref{Prel3} hence establishes that for every $c\in[0,\infty)$, $l_1,\dots,l_d\in\N_0$ 
with $\sum_{i=1}^d l_i=n$ it holds
    \begin{align}
      \begin{split}\label{Prel2b}
	 \abs{\tprod_{i=1}^d\pa{\tfrac{\d^{l_i}}{\d c^{l_i}}\brac{f\!\pa{\tfrac{K_i+c}{x_i}}}}}
	 &\leq\tprod_{i=1}^d\displaystyle\pa{\left[\sup_{t\in(0,\infty)}\abs{f^{(l_i)}(t)}\right] \abs{x_i}^{-l_i}}\\
	 &=\brac{\tprod_{i=1}^d\abs{x_i}^{-l_i}}\brac{\tprod_{i=1}^d\displaystyle\pa{\sup_{t\in(0,\infty)}\abs{f^{(l_i)}(t)}}}\\
	 &\leq\brac{\tprod_{i=1}^d\abs{x_i}^{-l_i}}\brac{\tprod_{\substack{i\in\oneto{d},\\l_i>0}}\displaystyle\pa{\max_{k\in\oneto{n}}\sup_{t\in(0,\infty)}\abs{f^{(k)}(t)}}}\\
	 &\leq\brac{\tprod_{i=1}^d\abs{x_i}^{-l_i}}\brac{\tprod_{\substack{i\in\oneto{d},\\l_i>0}}\displaystyle\max\left\{1,\max_{k\in\oneto{n}}\sup_{t\in(0,\infty)}\abs{f^{(k)}(t)}\right\}}\\
	 &\leq\brac{\tprod_{i=1}^d\abs{x_i}^{-l_i}}\brac{\max\left\{1,\max_{k\in\oneto{n}}\sup_{t\in(0,\infty)}\abs{f^{(k)}(t)}\right\}}^{(l_1+\ldots+l_d)}\\
	 &=\brac{\tprod_{i=1}^d\abs{x_i}^{-l_i}}\brac{\max_{k\in\zeroto{n}}\sup_{t\in(0,\infty)}\abs{f^{(k)}(t)}}^n.
      \end{split}
    \end{align}
Moreover, note that the multinomial theorem ensures that 
    \begin{align}
      \begin{split}
	d^n=\brac{\sum_{i=1}^d 1}^n&=\sum_{\substack{l_1,l_2,\dots,l_d\in\N_0,\\ \sum_{i=1}^d l_i=n}}\brac{\binom{n}{l_1,l_2,\dots,l_d}\tprod_{i=1}^d 1^{l_i}}
	=\sum_{\substack{l_1,l_2,\dots,l_d\in\N_0,\\ \sum_{i=1}^d l_i=n}}\brac{\binom{n}{l_1,l_2,\dots,l_d}}.
      \end{split}
    \end{align}
    Combining this with \eqref{Prel2a}, \eqref{Prel2b}, and the assumption that
    $x\in[a,b]^d$ implies that for every $c\in(0,\infty)$ we have
    \begin{align} 
      \begin{split}
 	\abs{F^{(n)}(c)}&\leq\abs{\sum_{\substack{l_1,l_2,\dots,l_d\in\N_0, 
         \\ 
        \sum_{i=1}^d l_i=n}}\brac{\binom{n}{l_1,l_2,\dots,l_d}\brac{\tprod_{i=1}^d\displaystyle\abs{x_i}^{-l_i}}\left[\max_{k\in\zeroto{n}}\sup_{t\in(0,\infty)}\abs{f^{(k)}(t)}\right]^n}}
        \\
 	&\leq a^{-n}\left[\max_{k\in\zeroto{n}}\sup_{t\in(0,\infty)}\abs{f^{(k)}(t)}\right]^n \abs{\sum_{\substack{l_1,l_2,\dots,l_d\in\N_0,
        \\ 
        \sum_{i=1}^d l_i=n}}\binom{n}{l_1,l_2,\dots,l_d}}\\
	&= a^{-n}\left[\max_{k\in\zeroto{n}}\sup_{t\in(0,\infty)}\abs{f^{(k)}(t)}\right]^n d^n.
      \end{split}
    \end{align}    
    This completes the proof of \Cref{Prel2}. 
  \end{proof}
\section{Quadrature}\label{QuadSection}

To approximate the function $x\mapsto u(0,x)$ from \eqref{explicitForm} by a neural network we
need to evaluate, for arbitrary, given $x$, an expression of the form $\int_0^{\infty} F_x(c)\d c$ with $F_x$ as defined in \Cref{QN}. 
We achieve this by proving in \Cref{QN} that the functions $F_x$ decay sufficiently fast 
for $c\to\infty$, and then employ numerical integration
to show that the definite integral $\int_0^N F_x(c)\d c$ 
can be sufficiently well approximated by a weighted sum 
of $F_x(c_j)$ for suitable quadrature points $c_j\in(0,N)$. 
The representation of such a sum can be realized by neural networks. 
We show in Section 6 and 7 how the functions $x\mapsto F_x(c_j)$ for $(c_j)\in(0,N)$ 
can be realized efficiently due to their tensor product structure.
We start by recalling an error bound 
for composite Gaussian quadrature  which is explicit in the stepsize and
quadrature order.

\begin{lemma}\label{CGQ}
Let $n,M\in\N$, $N\in(0,\infty)$. 
Then there exist real numbers $(c_j)_{j=1}^{nM}\subseteq(0,N)$ 
and ${(w_j)_{j=1}^{nM}\subseteq(0,\infty)}$ 
such that for every $h\in C^{2n}([0,N],\R)$ it holds
\begin{align}
\abs{\int_0^N h(t)\,\d t - \sum_{j=1}^{nM}w_j h(c_j)}
\leq 
\tfrac{1}{(2n)!}N^{2n+1}M^{-2n}\brac{\sup_{\xi\in[0,N]}\abs{h^{(2n)}(\xi)}}.\label{GCQresult}  
    \end{align}
  \end{lemma}
  \begin{proof}[Proof of \Cref{CGQ}]
    Throughout this proof let $h\in C^{2n}([0,N],\R)$ and $\alpha_k\in[0,N]$, $k\in\{0,1,\dots,M\}$, 
    such that for every $k\in\{0,1,\dots,M\}$ it holds $\alpha_k=\tfrac{kN}{M}$.
    Observe that \cite[Theorems 4.17, 6.11, and 6.12]{DLevy} ensure 
    that for every $k\in\{0,1,\dots,M-1\}$ there exist 
    $(\gamma^k_i)_{i=1}^{n}\subseteq(\alpha_k,\alpha_{k+1})$, 
    $(\omega^k_i)_{i=1}^{n}\subseteq(0,\infty)$, 
    and $\xi^k\in[\alpha_k,\alpha_{k+1}]$ such that
    \begin{align}
      \int_{\alpha_k}^{\alpha_{k+1}}h(t)\,\d t-\sum_{i=1}^n\omega^k_i h(\gamma^k_i)
     =\frac{h^{(2n)}(\xi^k)}{(2n)!}\int_{\alpha_k}^{\alpha_{k+1}}\brac{\tprod_{i=1}^n (t-\gamma^k_i)^2} \d t.\label{GCQest1}
    \end{align}
    Next note that for every $k\in\{0,1,\dots,M-1\}$ it holds
    \begin{align}\begin{split}
      \int_{\alpha_k}^{\alpha_{k+1}}\brac{\tprod_{i=1}^n (t-\gamma^k_i)^2} \d t 
     &\leq \int_{\alpha_{k}}^{\alpha_{k+1}}\brac{\tprod_{i=1}^n(\alpha_k-\alpha_{k+1})^2} \d t
      =\brac{\tfrac{N}{M}}^{2n+1}.
    \end{split}\end{align}
    Combining this with \eqref{GCQest1} yields that for every $k\in\{0,1,\dots,M\}$ we have
    \begin{align}\begin{split}
      \abs{\int_{\alpha_k}^{\alpha_{k+1}}h(t)\, \d t-\sum_{i=1}^n\omega^k_i h(\gamma^k_i)}&\leq\frac{\abs{h^{(2n)}(\xi^k)}}{(2n)!}\brac{\tfrac{N}{M}}^{2n+1}
      \leq\tfrac{1}{(2n)!}\brac{\tfrac{N}{M}}^{2n+1}\brac{\sup_{\xi\in[0,N]}\abs{h^{(2n)}(\xi)}}.
    \end{split}\end{align}
    Hence, we obtain 
    \begin{align}\begin{split}\label{GCQmain}
      \abs{ \int_0^N h(t)\, \d t\!-\!\!\sum_{k=0}^{M-1}\!\sum_{i=1}^n\omega^k_i h(\gamma^k_i)}
      &=\abs{\sum_{k=0}^{M-1}\brac{\int_{\alpha_k}^{\alpha_{k+1}}h(t)\,\d t-\sum_{i=1}^n\omega^k_i h(\gamma^k_i)}}\\
      &\leq\sum_{k=0}^{M-1}\pa{\tfrac{1}{(2n)!}\pa{\tfrac{N}{M}}^{2n+1}\brac{\sup_{\xi\in[0,N]}\abs{h^{(2n)}(\xi)}}}\\
      &=\tfrac{1}{(2n)!}N^{2n+1}M^{-2n}\brac{\sup_{\xi\in[0,N]}\abs{h^{(2n)}(\xi)}}.
    \end{split}\end{align}
    Let  $(c_j)_{j=1}^{nM}\subseteq(0,N)$, $(w_j)_{j=1}^{nM}\subseteq(0,\infty)$ such that for every $i\in\oneto{n}$, $k\in\{0,1,\dots,M-1\}$ it holds 
    \begin{align}
      c_{kn+i}=\gamma_i^k\quad\mathrm{and}\quad w_{kn+i}=\omega_i^k.
    \end{align}
    Next observe that
    \begin{align}
      \abs{\int_0^N h(t)\,\d t - \sum_{j=1}^{nM}w_j h(c_j)}
     =\abs{ \int_0^N h(t)\,\d t\!-\!\!\sum_{k=0}^{M-1}\!\sum_{i=1}^n\omega^k_i h(\gamma^k_i)}.
    \end{align}
    This completes the proof of \Cref{CGQ}. 
  \end{proof}

In the following we bound the error due to truncating the domain of integration.

  \begin{lemma}\label{QN}
    Let $d,n\in\N$, $a\in(0,\infty)$, $b\in(a,\infty)$, $K=(K_1,K_2,\dots,K_d)\in[0,\infty)^d$, 
    let $F_x\colon(0,\infty)\to\R$, $x\in[a,b]^d$, be the functions which satisfy
    for every $x=(x_1,x_2,\dots,x_d)\in[a,b]^d$, $c\in(0,\infty)$ that
    \begin{align}
      F_x(c)=1-\prod_{i=1}^d\brac{\normfac\int_{-\infty}^{\ln(\frac{K_i+c}{x_i})}e^{-\frac{1}{2}r^2}\d r},
    \end{align}
    and for every $\eps\in(0,1]$ let $N_{\eps}\in\R$ be given by
    $N_{\eps}=2e^{2(n+1)}(b+1)^{1+\frac{1}{n}}d^{\frac{1}{n}}\eps^{-\frac{1}{n}}$.\!
    Then it holds for every $\eps\in(0,1]$ that
    \begin{align}
      \sup_{x\in[a,b]^d}\abs{\int_{N_{\eps}}^\infty F_x(c)\,\d c}\leq\eps.
    \end{align}
  \end{lemma}
  
  \begin{proof}[Proof of \Cref{QN}]
    Throughout this proof let $g\colon(0,\infty)\to(0,1)$ be the function given by 
    \begin{align}
      g(t)=1-\normfac\int_{-\infty}^{\ln(t)}e^{-\frac{1}{2}r^2}\d r.
    \end{align}
    Note that \cite[Eq.(5)]{ExpBounds} ensures that for every $y\in[0,\infty)$ we have  
    $\tfrac{2}{\sqrt{\pi}}\int_y^{\infty}e^{-r^2}\d r \leq e^{-y^2}$.\ 
    This implies for every $t\in[1,\infty)$ that
    \begin{align}\begin{split}\label{QNT1}
      0<g(t)&=1-\normfac\int_{-\infty}^{\ln(t)}e^{-\frac{1}{2}r^2}\d r
      =\normfac\int_{\ln(t)}^\infty e^{-\frac{1}{2}r^2}\d r
      =\tfrac{1}{\sqrt{\pi}}\int^{\infty}_{\frac{\ln(t)}{\sqrt{2}}}
                     e^{-r^2}\d r\leq\tfrac{1}{2}e^{-\frac{1}{2}[\ln(t)]^2}.
    \end{split}\end{align}
    Furthermore, observe that for every $t\in[e^{2(n+1)},\infty)$ it holds
    \begin{align} 
      \begin{split}
	\elnt&=e^{\brac{\ln(t)(-\frac{1}{2}\ln(t))}}
              =\brac{e^{\ln(t)}}^{-\frac{1}{2}\ln(t)}=t^{-\frac{1}{2}\ln(t)}\leq t^{-(n+1)}.
      \end{split}
    \end{align}
    This, \eqref{QNT1}, and the fact that for every $\eps\in(0,1]$, $c\in[N_{\eps},\infty)$, $x\in[a,b]^d$, $i\in\oneto{d}$ we have $\tfrac{K_i+c}{x_i}\geq\tfrac{c}{b}\geq e^{2(n+1)}\geq 1$ imply that for every $\eps\in(0,1]$, $c\in[N_{\eps},\infty)$, $x\in[a,b]^d$ it holds
    \begin{align}\begin{split}
      \abs{F_x(c)}&=\abs{1-\prod_{i=1}^d\brac{\normfac\int_{-\infty}^{\ln(\frac{K_i+c}{x_i})}e^{-\frac{1}{2}r^2}\d r}}=\abs{1-\prod_{i=1}^d\brac{1-g(\tfrac{K_i+c}{x_i})}}\\
      &\leq\abs{1-\prod_{i=1}^d\brac{1-\tfrac{1}{2}\brac{\tfrac{K_i+c}{x_i}}^{-(n+1)}}}\leq\abs{1-\prod_{i=1}^d\brac{1-\tfrac{1}{2}\brac{\tfrac{c}{b}}^{-(n+1)}}}.
    \end{split}\end{align}
    Combining this with the binomial theorem and the fact that for every $i\in\oneto{d}$ we have
    $\binom{d}{i}\leq\tfrac{d^i}{i!}\leq\tfrac{d^i}{\exp(i\ln(i)-i+1)}\leq\tfrac{(de)^i}{i^i}$ establishes 
    that for every $\eps\in(0,1]$, $c\in[N_{\eps},\infty)$, $x\in[a,b]^d$ it holds
    \begin{align}\begin{split}
      \abs{F_x(c)}&\leq\abs{1-\pa{1-\tfrac{1}{2}\brac{\tfrac{c}{b}}^{-(n+1)}}^d}=\abs{1-\sum_{i=0}^d\brac{\binom{d}{i}\brac{-\tfrac{1}{2}\brac{\tfrac{c}{b}}^{-(n+1)}}^i}}\\
      &\leq\sum_{i=1}^d\brac{\binom{d}{i}\brac{\tfrac{1}{2}}^i\brac{\tfrac{b}{c}}^{(n+1)i}}\leq\sum_{i=1}^d\brac{\tfrac{de}{2i}}^i\brac{\tfrac{b}{c}}^{(n+1)i}\\
      &=\sum_{i=1}^d\brac{\tfrac{e}{2i}}^i\brac{d\brac{\tfrac{b}{c}}^{n+1}}^i\leq 2d\brac{\tfrac{b}{c}}^{n+1}\brac{\sum_{i=1}^d\brac{d\brac{\tfrac{b}{c}}^{n+1}}^{i-1}}\\
      &=2d\brac{\tfrac{b}{c}}^{n+1}\brac{\sum_{i=0}^{d-1}\brac{d\brac{\tfrac{b}{c}}^{n+1}}^{i}}\leq2d\brac{\tfrac{b}{c}}^{n+1}\brac{\sum_{i=0}^{\infty}\brac{d\brac{\tfrac{b}{c}}^{n+1}}^{i}}.
    \end{split}\end{align}
    This, the geometric sum formula, and the fact that for every $\eps\in(0,1]$ 
    it holds that $N_\eps\geq 2bd^{\frac{1}{n}}$ imply 
    that for every $\eps\in(0,1]$, $c\in[N_{\eps},\infty)$, $x\in[a,b]^d$ we have
    \begin{align}
      \abs{F_x(c)}&\leq 2d\brac{\tfrac{b}{c}}^{n+1} \brac{\frac{1}{1-d\brac{\tfrac{b}{c}}^{n+1}}}
    \leq 4d\brac{\tfrac{b}{c}}^{n+1}.
    \end{align}
    Hence, we obtain for every $\eps\in(0,1]$, $x\in[a,b]^d$ that
    \begin{align}\begin{split}
      \abs{\int_{N_{\eps}}^\infty F_x(c)\,\d c}
      &\leq 4db^{n+1} \abs{\int_{N_{\eps}}^\infty c^{-(n+1)}\d c}
       =4db^{n+1}\tfrac{1}{n}(N_{\eps})^{-n}\\
      &=\tfrac{4}{n}db^{n+1}\brac{2e^{2(n+1)}(b+1)^{1+\frac{1}{n}}d^{\frac{1}{n}}\eps^{-\frac{1}{n}}}^{-n}\\
      &=\tfrac{4}{n}db^{n+1}2^{-n}e^{-(2n^2+2n)}(b+1)^{-(n+1)}d^{-1}\eps\\
      &=\tfrac{4}{n}2^{-n}e^{-(2n^2+n)}\brac{\tfrac{b}{b+1}}^{n+1}\eps\leq\eps.
    \end{split}\end{align}
    This completes the proof of \Cref{QN}.
  \end{proof}
  
Next we combine the result above with \Cref{CGQ} in order to derive 
the number of terms needed in order to approximate the integral by a sum to within a prescribed error bound $\eps$.
  
  \begin{lemma}\label{QM}
    Let $n\in\N$, $a\in(0,\infty)$, $b\in(a,\infty)$, $(K_i)_{i\in\N}\subseteq[0,\infty)$, 
    let $F^d_x\colon(0,\infty)\to\R$, $x\in[a,b]^d$, $d\in\N$, be the functions which satisfy for every $d\in\N$, $x=(x_1,x_2,\dots,x_d)\in[a,b]^d$, $c\in(0, \infty)$ that
    \begin{align}
    F^d_x(c)=1-\prod_{i=1}^d\brac{\normfac\int_{-\infty}^{\ln(\frac{K_i+c}{x_i})}e^{-\frac{1}{2}r^2}\d r},
    \end{align}
    and for every $d\in\N$, $\eps\in(0,1]$ let $N_{d,\eps}\in\R$ be given by
    \begin{align}
      N_{d,\eps}=2e^{2(n+1)}(b+1)^{1+\frac{1}{n}}d^{\frac{1}{n}}\brac{\tfrac{\eps}{2}}^{-\frac{1}{n}}.
    \end{align}
Then there exist $Q_{d,\eps}\in\N$, $c^d_{\eps,j}\in(0,N_{d,\eps})$, $w^d_{\eps,j}\in[0,\infty)$, 
$j\in\oneto{Q_{d,\eps}}$, $d\in\N$, $\eps\in(0,1]$, such 
    \begin{enumerate}[(i)]
    \item that
      \begin{align}
	\sup_{\eps\in(0,1], d\in\N}\brac{\frac{Q_{d,\eps}}{d^{1+\frac{2}{n}}\eps^{-\frac{2}{n}}}}<\infty
      \end{align}
      and
    \item that for every $d\in\N$, $\eps\in(0,1]$ it holds $\sum_{j=1}^{Q_{d,\eps}}w^d_{\eps,j}=N_{d,\eps}$ and
      \begin{align}
	\sup_{x\in[a,b]^d}\abs{\int_0^\infty F^d_x(c)\,\d c-\sum_{j=1}^{Q_{d,\eps}}w^d_{\eps,j}F^d_x(c^d_{\eps,j})}\leq\eps.
      \end{align}
    \end{enumerate}
  \end{lemma}
  
  \begin{proof}[Proof of \Cref{QM}]
    Note that \Cref{Prel2} ensures the existence of $S_m\in\R$, $m\in\N$, 
    such that for every $d,m\in\N$, $x\in[a,b]^d$ it holds
      \begin{align}
	\sup_{c\,\in(0,\infty)}\abs{(F^d_x)^{(m)}(c)}\leq S_m d^m.
      \end{align}
    Let $Q_{d,\eps}\in\R$, $d\in\N$, $\eps\in(0,1]$, be given by
    \begin{align}
      Q_{d,\eps}=n\left\lceil\brac{\tfrac{1}{(2n)!}(N_{d,\eps})^{2n+1}S_{2n}d^{2n}\tfrac{2}{\eps}}^{\frac{1}{2n}}\right\rceil.
    \end{align}
    Next observe that \Cref{CGQ} (with $N\leftrightarrow N_{d,\eps}$ in the notation of \Cref{CGQ})
    establishes the existence of $c^d_{\eps,j}\in(0,N_{d,\eps})$, $w^d_{\eps,j}\in[0,\infty)$,  $j\in\oneto{Q_{d,\eps}}$,
    $d\in\N$, $\eps\in(0,1]$, such that for every $d\in\N$, $\eps\in(0,\infty)$, $x\in[a,b]^d$ we have $\sum_{j=1}^{Q_{d,\eps}}w^d_{\eps,j}=N_{d,\eps}$ and
    \begin{align}\begin{split}\label{QMT}
      \abs{\int_0^{N_{d,\eps}}F^d_x(c)\d c-\sum_{j=1}^{Q_{d,\eps}}w^d_{\eps,j}F^d_x(c^d_{\eps,j})} 
      &\leq\tfrac{1}{(2n)!}(N_{d,\eps})^{2n+1}\brac{\tfrac{Q_{d,\eps}}{n}}^{-2n}S_{2n}d^{2n}\\
      &\leq\tfrac{1}{(2n)!}(N_{d,\eps})^{2n+1}\brac{\tfrac{1}{(2n)!}(N_{d,\eps})^{2n+1}S_{2n}d^{2n}\tfrac{2}{\eps}}^{-1}S_{2n}d^{2n}=\tfrac{\eps}{2}.
    \end{split}\end{align}
    Moreover, note that \Cref{QN} (with $N_{d,\frac{\eps}{2}}\leftrightarrow N_{d,\eps}$ in the notation of \Cref{QN}) 
    and \eqref{QMT} imply for every $d\in\N$, $\eps\in(0,1]$, $x\in[a,b]^d$ that
    \begin{align}\begin{split}
      &\quad\abs{\int_0^\infty F^d_x(c)\,\d c-\sum_{j=1}^{Q_{d,\eps}} w^d_{\eps,j}F^d_x(c^d_{\eps,j})}
      \\
      &\leq\abs{\int_0^{N_{d,\eps}}F^d_x(c)\,\d c-\sum_{j=1}^{Q_{d,\eps}}w^d_{\eps,j}F^d_x(c^d_{\eps,j})}
      + \abs{\int^{\infty}_{N_{d,\eps}}F^d_x(c)\,\d c}\\
      &\leq \tfrac{\eps}{2}+\tfrac{\eps}{2}=\eps.
    \end{split}\end{align}
    Furthermore, we have for every $d\in\N$, $\eps\in(0,1]$ that 
    \begin{align}\begin{split}
      Q_{d,\eps}&\leq n\pa{1+\brac{\tfrac{1}{(2n)!}(N_{d,\eps})^{2n+1}S_{2n}d^{2n}\tfrac{2}{\eps}}^{\frac{1}{2n}}}\\
      &=n+n\brac{\tfrac{2S_{2n}}{(2n)!}}^{\frac{1}{2n}}d\eps^{-\frac{1}{2n}}(N_{d,\eps})^{1+\frac{1}{2n}}\\
      &\leq n+n\brac{\tfrac{2S_{2n}}{(2n)!}}^{\frac{1}{2n}}d\eps^{-\frac{1}{2n}}\brac{4e^{2(n+1)}(b+1)^{1+\frac{1}{n}}d^{\frac{1}{n}}\eps^{-\frac{1}{n}}}^{1+\frac{1}{2n}}\\
      &=n+4n\brac{\tfrac{8S_{2n}}{(2n)!}}^{\frac{1}{2n}}e^{2n+3+\frac{1}{n}}\brac{b+1}^{1+\frac{3}{2n}+\frac{1}{2n^2}}d^{1+\frac{1}{n}+\frac{1}{2n^2}}\eps^{-\frac{3}{2n}-\frac{1}{2n^2}}\\
      &\leq nd^{1+\frac{2}{n}}\eps^{-\frac{2}{n}} + 4n\brac{\tfrac{8S_{2n}}{(2n)!}}^{\frac{1}{2n}}e^{2n+3+\frac{1}{n}}\brac{b+1}^{1+\frac{3}{2n}+\frac{1}{2n^2}}d^{1+\frac{2}{n}}\eps^{-\frac{2}{n}}.
    \end{split}
    \end{align}
    This implies 
    \begin{align}
    \sup_{\eps\in(0,1], d\in\N}\brac{\frac{Q_{d,\eps}}{d^{1+\frac{2}{n}}\eps^{-\frac{2}{n}}}}
    \leq
     n + 4n\brac{\tfrac{8S_{2n}}{(2n)!}}^{\frac{1}{2n}}e^{2n+3+\frac{1}{n}}\brac{b+1}^{1+\frac{3}{2n}+\frac{1}{2n^2}}<\infty.
    \end{align}
    The proof of \Cref{QM} is thus completed.
  \end{proof}

\section{Basic ReLU DNN Calculus}\label{sec:relucalc}
In order to talk about neural networks we will, up to some minor changes and additions, 
adopt the notation of P. Petersen and F. Voigtlaender from \cite{PetVoigt}.
This allows us to differentiate between a neural network, defined as a structured set of weights,
and its realization, which is a function on $\R^d$.
Note that this is almost necessary in order to talk about the complexity of neural networks, 
since notions like depth, size or architecture do not make sense for general functions on $\R^d$.
Even if we know that a given function 'is' a neural network, i.e. can be written a series of 
affine transformations and componentwise non-linearities, there are, in general, multiple non-trivially different ways to do so.

Each of these structured sets we consider does however define a unique function.
This enables us to explicitly and unambiguously construct complex neural networks from simple ones, 
and subsequently relate the approximation capability of a given network to its complexity.
Further note that since the realization of neural network is unique we can still speak 
of a neural network approximating a given function when its realization does so.

Specifically, a neural network will be given by its architecture, i.e. number of layers $L$ and layer dimensions\footnote{Often phrased as input dimension $N_0$ and output dimension $N_L$ with $N_l$, 
$l\in\oneto{L-1}$ many neurons in the $l$'th layer.}$N_0,N_1,\dots,N_L$,
as well as the weights determining the affine transformations used to compute each layer from the previous one.
Note that our notion of neural networks does not attach the architecture and weights 
to a fixed activation function, but instead considers the realization of such a neural 
network with respect to a given activation function.
This choice is a purely technical one here, as we always consider networks with ReLU activation function.
 
\begin{setting}[Neural networks]\label{NNSetting}
    For every $L\in\N$, $N_0,N_1,\dots,N_L\in\N$ let $\mathcal{N}_L^{N_0,N_1,\dots,N_L}$ 
    be the set given by
    \begin{align}
      \mathcal{N}_L^{N_0,N_1,\dots,N_L}=\times_{l=1}^L\pa{\R^{N_l\times N_{l-1}}\times\R^{N_l}},
    \end{align}
    let $\mathfrak{N}$ be the set given by
    \begin{align}
      \mathfrak{N}=\bigcup_{\substack {L \in \N,\\N_0, N_1, ..., N_L \in \N } } \mathcal{N}^{ N_0, N_1,\dots, N_L }_L, 
    \end{align}
    let $\L,\M,\M_l,\dimin,\dimout\colon\mathfrak{N}\to\N$, $l\in\oneto{L}$, 
    be the functions which satisfy for every $L\in\N$ and every ${N_0,N_1,\dots,N_L\in\N}$, 
    $\Phi=(((A^1_{i,j})_{i,j=1}^{N_1,N_0},(b^1_i)_{i=1}^{N_1}),\dots,((A^L_{i,j})_{i,j=1}^{N_L,N_{L-1}},(b^L_i)_{i=1}^{N_L}))\in\mathcal{N}^{ N_0, N_1,\dots, N_L }_L$, $l\in\oneto{L}$ 
    $\L(\Phi)=L$, $\dimin(\Phi)=N_0$, $\dimout(\Phi)=N_L$,
    \begin{align}
    \M_l(\Phi)=\sum_{i=1}^{N_l}\brac{\mathbbm{1}_{\R\backslash\{0\}}(b^l_i)
              +\sum_{j=1}^{N_{l-1}}\mathbbm{1}_{\R\backslash\{0\}}(A^l_{i,j})},
    \end{align}
    and 
    \begin{align}
      \M(\Phi)=\sum_{l=1}^L\M_l(\Phi).
    \end{align}

For every $\rho\in C(\R,\R)$ let $\rho^*\colon\cup_{d\in\N}\R^d\to\cup_{d\in\N}\R^d$ 
be the function which satisfies for every $d\in\N$, $x=(x_1,x_2,\dots,x_d)\in\R^d$ 
that $\rho^*(x)=(\rho(x_1),\rho(x_2),\dots,\rho(x_d))$,
and for every $\rho\in\C(\R,\R)$ denote by $R_{\rho}\colon\mathfrak{N}\to\cup_{a,b\in\N}\,C(\R^a,\R^b)$ 
the function which satisfies 
for every $L\in\N$, $N_0,N_1,\dots,N_L\in\N$, $x_0\in\R^{N_0}$, 
and
$\Phi=((A_1,b_1),(A_2,b_2),\dots,(A_L,b_L))\in\mathcal{N}_L^{N_0,N_1,\dots,N_L}$, 
with $x_1\in\R^{N_1},\dots,x_{L-1}\in\R^{N_{L-1}}$ given by
\begin{align}
      x_l = \rho^*( A_{l} x_{l-1} + b_{l})\;,\qquad l = 1,...,L-1\;,
\end{align}
that
\begin{align}
      \brac{ R_{ \rho }( \Phi ) }( x_0 ) = A_L x_{ L - 1 } + b_L\;.
\end{align}
\end{setting}
  
The quantity $\M(\Phi)$ simply denotes the number of non-zero entries of the network $\Phi$, 
which together with its depth $\L(\Phi)$ will be how we measure the 'size' of a given neural network $\Phi$. 
  One could instead consider the number of all weights, i.e. including zeroes, of a neural network. 
Note, however, that for any non-degenerate neural network $\Phi$ the total number of weights is bounded from above by $\M(\Phi)^2+\M(\Phi)$.
Here, the terminology ``degenerate'' refers to a neural network which has neurons that can be 
removed without changing the realization of the NN. 
This implies for any neural network there also exists a non-degenerate one of smaller or equal size, which has the exact same realization. 
Since our primary goal is to approximate $d$-variate functions by networks the size of which 
only depends polynomially on the dimension, the above
means that the qualitatively same results hold regardless of which notion of 'size' is used. 

We start by introducing two basic tools for constructing new neural networks from known ones and, in \Cref{NNconc} and \Cref{NNpar}, consider how the properties of a derived network depend on its parts.
Note that techniques like these have already been used in \cite{PetVoigt} and \cite{2017Schmidt-Hieber}.

The first tool will be the 'composition' of neural networks in \eqref{ConcDef}, which takes two networks and provides a new network whose realization is the composition of the realizations 
of the two constituent functions.

The second tool will be the 'parallelization' of neural networks in \eqref{NNSPPdef}, 
which will be useful when considering linear combinations or tensor products of functions which we can already approximate. 
While parallelization of same-depth networks \eqref{NNSPPcdef} works with arbitrary activation functions, 
we use for the general case that any ReLU network can easily be extended \eqref{NNSPEdef} 
to an arbitrary depth without changing its realization. 
  
\begin{setting}\label{NNSettingPar}
Assume \Cref{NNSetting}, 
for every $L_1,L_2\in\N$, $\Phi^i=\pa{(A_1^i,b_1^i), (A_2^i,b_2^i),\dots,(A^i_{L_i},b^i_{L_i})}\in\Nall$, $i\in\{1,2\}$, 
with $\dimin(\Phi^1)=\dimout(\Phi^2)$ let $\Phi^1\odot\Phi^2\in\Nall$ 
be the neural network given by 
\footnotesize
\begin{align}\label{ConcDef}
      \Phi^1\odot\Phi^2=\pa{(A_1^2,b_1^2),\dots,(A^2_{L_2-1},b^2_{L_2-1}),\pa{\begin{pmatrix}A^2_{L_2}\\-A^2_{L_2}\end{pmatrix},
       \begin{pmatrix}b^2_{L_2}\\-b^2_{L_2}\end{pmatrix}},
      \pa{\begin{pmatrix}A^1_1 & -A^1_1\end{pmatrix},b^1_1},(A^1_2,b^1_2),\dots,(A^1_{L_1},b^1_{L_1})},
    \end{align}
    \normalsize
    for every $d\in\N$, $L\in\N\cap[2,\infty)$ let $\Phi^{\Id}_{d,L}\in\Nall$ be the neural network given by
    \begin{align}\label{NNSPId1}
    \Phi^{\Id}_{d,L}=\pa{\pa{\begin{pmatrix}\Idd \\ -\Idd\end{pmatrix},0},\underbrace{(\Id_{\R^{2d}},0),\dots,(\Id_{\R^{2d}},0)}_{\text{L-2 times}},\pa{\begin{pmatrix}\Idd & -\Idd\end{pmatrix},0}},
    \end{align}
    for every $d\in\N$ let $\Phi^{\Id}_{d,1}\in\Nall$ be the neural network given by
    \begin{align}\label{NNSPId2}
       \Phi^{\Id}_{d,1}=((\Idd,0)),
    \end{align}
    for every $n,L\in\N$, $\Phi^j=((A^j_1,b^j_1),(A^j_2,b^j_2),\dots,(A^j_L,b^j_L))\in\Nall$, $j\in\oneto{n}$, let $\Pc_s(\Phi^1,\Phi^2,\dots,\Phi^n)\in\Nall$ be the neural network which satisfies
    \small
    \begin{align}\label{NNSPPcdef}
       \Pc_s(\Phi^1,\Phi^2,\dots,\Phi^n)=\pa{\pa{\begin{pmatrix}A^1_1&&&\\& A^2_1 &&\\&&\ddots&\\&&&A^n_1\end{pmatrix},\begin{pmatrix}b^1_1\\b^2_1\\\vdots\\ b^n_1\end{pmatrix}},\dots,\pa{\begin{pmatrix}A^1_L&&&\\& A^2_L &&\\&&\ddots&\\&&&A^n_L\end{pmatrix},\begin{pmatrix}b^1_L\\b^2_L\\\vdots\\ b^n_L\end{pmatrix}}},
    \end{align}
    \normalsize
    for every $L,d\in\N$, $\Phi\in\Nall$ with $\L(\Phi)\leq L$, $\dimout(\Phi)=d$, 
    let $\mathcal{E}_L(\Phi)\in\Nall$ be the neural network given by
    \begin{align}\label{NNSPEdef}
      \mathcal{E}_L(\Phi)=\begin{cases}\Phi^{\Id}_{d,L-\L(\Phi)}\odot\Phi & \colon \L(\Phi)<L \\ \Phi & \colon \L(\Phi)=L\end{cases},
    \end{align}
    and for every $n,L\in\N$, $\Phi^j\in\Nall$, $j\in\oneto{n}$ with $\max_{j\in\oneto{n}}\L(\Phi^j)=L$, 
    let $\Pc(\Phi^1,\Phi^2,\dots,\Phi^n)\in\Nall$ denote the neural network given by
    \begin{align}\label{NNSPPdef}
      \Pc(\Phi^1,\Phi^2,\dots,\Phi^n)=\Pc_s(\mathcal{E}_L(\Phi^1),\mathcal{E}_L(\Phi^2),\dots,\mathcal{E}_L(\Phi^n)).
    \end{align}
  \end{setting}

  \begin{lemma}\label{NNconc}
    Assume \Cref{NNSettingPar}, let $\Phi^1,\Phi^2\in\Nall$, and 
    let $\rho\colon\R\to\R$ be the function which satisfies for every $t\in\R$ that $\rho(t)=\max\{0,t\}$. 
    Then 
    \begin{enumerate}[(i)]
    \item\label{NNci} for every $x\in\R^{\dimin(\Phi^2)}$ it holds
      \begin{align}
	[\Rr(\Phi^1\odot\Phi^2)](x)=([\Rr(\Phi^1)]\circ[\Rr(\Phi^2)])(x)=[\Rr(\Phi^1)]([\Rr(\Phi^2)](x)),
      \end{align}
    \item\label{NNcii} 
        $\L(\Phi^1\odot\Phi^2)=\L(\Phi^1)+\L(\Phi^2)$, 
    \item 
        $\M(\Phi^1\odot\Phi^2)\leq \M(\Phi^1)+\M(\Phi^2)+\M_1(\Phi^1)+\M_{\L(\Phi^2)}(\Phi^2)
                \leq 2(\M(\Phi^1)+\M(\Phi^2))$,
    \item $\M_1(\Phi^1\odot\Phi^2)=\M_1(\Phi^2)$,
    \item $\M_{\L(\Phi^1\odot\Phi^2)}(\Phi^1\odot\Phi^2)=\M_{\L(\Phi^1)}(\Phi^1)$,
    \item $\dimin(\Phi^1\odot\Phi^2)=\dimin(\Phi^2)$, 
    \item\label{NNcvii} 
          $\dimout(\Phi^1\odot\Phi^2)=\dimout(\Phi^1)$,
    \item\label{NNcviii}
	  for every $d,L\in\N$, $x\in\R^d$ it holds that $[\Rr(\Phi^{\Id}_{d,L})](x)=x$, and
    \item\label{NNcix}
	  for every $L\in\N$, $\Phi\in\Nall$ with $\L(\Phi)\leq L$, $x\in\R^{\dimin(\Phi)}$ it holds that $[\Rr(\mathcal{E}_L(\Phi))](x)=[\Rr(\Phi)](x)$.
    \end{enumerate}
  \end{lemma}  
  
  \begin{proof}[Proof of \Cref{NNconc}]
  For every $i\in\{1,2\}$ let $L_i\in\N$, $N^i_1,N^i_2,\dots,N^i_{L_i}$, ${(A^i_l,b^i_l)\in\R^{N^i_l\times N^i_{l-1}}\times\R^{N^i_l}}$, 
      $l\in\oneto{L_i}$ such that $\Phi^i=((A^i_1,b^i_1),\dots,(A^i_{L_i},b^i_{L_i}))$.
  Furthermore,  let $(A_l,b_l)\in\R^{N_l\times N_{l-1}}\times\R^{N_l}$, $l\in\oneto{L_1+L_2}$, 
  be the matrix-vector tuples which satisfy $\Phi_1\odot\Phi_2=((A_1,b_1),\dots,(A_{L_1+L_2},b_{L_1+L_2}))$ 
  and let $r_l\colon\R^{N_0}\to\R^{N_l}$, $l\in\oneto{L_1+L_2}$, 
  be the functions which satisfy for every $x\in\R^{N_0}$ that
  \begin{align}\label{NNcrlDef}
    r_l(x)=\begin{cases}
	    \rho^*(A_1x+b_1) & \colon l=1\\
	    \rho^*(A_l r_{l-1}(x)+b_l) & \colon 1<l<L_1+L_2\\
	    A_l r_{l-1}(x)+b_l & \colon l=L_1+L_2
           \end{cases}.
  \end{align}
  Observe that for every $l\in\oneto{L_2-1}$ holds $(A_l,b_l)=(A^2_l,b^2_l)$. 
  This implies that for every $x\in\R^{N_0}$ holds 
  \begin{align}
    A^2_{L_2}r_{L_2-1}(x)+b^2_{L_2}=[\Rr(\Phi_2)](x).
  \end{align}
  Combining this with \eqref{ConcDef} implies for every $x\in\R^{N_0}$ that
  \begin{align}\begin{split}\label{NNcT1}
  r_{L_2}(x)&=\rho^*(A_{L_2}r_{L_2-1}(x)+b_{L_2})=\rho^*\left(\begin{pmatrix}A^2_{L_2}\\-A^2_{L_2}\end{pmatrix}r_{L_2-1}(x)+\begin{pmatrix}b^2_{L_2}\\-b^2_{L_2}\end{pmatrix}\right)\\
    &=\rho^*\left(\begin{pmatrix}A^2_{L_2}r_{l-1}(x) +b^2_{L_2}\\ -A^2_{L_2}r_{l-1}(x)-b^2_{L_2}\end{pmatrix}\right)=\begin{pmatrix}\rho^*([\Rr(\Phi^2)](x))\\\rho^*(-[\Rr(\Phi^2)](x))\end{pmatrix}
  \end{split}\end{align}
  In addition, for every $d\in\N$, $y=(y_1,y_2,\dots,y_d)\in\R^d$ holds
  \begin{align}\label{NNcT2}
    \rho^*(y)-\rho^*(-y)=(\rho(y_1)-\rho(-y_1),\rho(y_2)-\rho(-y_2),\dots,\rho(y_d)-\rho(-y_d))=y.
  \end{align}
  This, \eqref{ConcDef}, and \eqref{NNcT1} ensure that for every $x\in\R^{N_0}$ holds
  \begin{align}\begin{split}
    r_{L_2+1}(x)
    &=A_{L_2+1}
     \begin{pmatrix}\rho^*([\Rr(\Phi^2)](x)) \\
                    \rho^*(-[\Rr(\Phi^2)](x))
     \end{pmatrix}+b_{L_2+1}\\
    &=A^1_1\rho^*([\Rr(\Phi^2)](x))-A^1_1 \rho^*(-[\Rr(\Phi^2)](x))+b_{L_2+1}\\
    &=A^1_1[\Rr(\Phi^2)](x)+b^1_1.
  \end{split}\end{align}
  Combining this with \eqref{NNcrlDef} establishes \eqref{NNci}. 
  Moreover, \eqref{NNcii}-\eqref{NNcvii} follow directly from \eqref{ConcDef}. Furthermore, \eqref{NNSPId1}, \eqref{NNSPId2}, and \eqref{NNcT2} imply \eqref{NNcviii}. 
  Finally, \eqref{NNcix} follows from \eqref{NNSPEdef} and \eqref{NNcviii}.
  This completes the proof of \Cref{NNconc}.
  \end{proof}

  \begin{lemma}\label{NNpar}
    Assume \Cref{NNSettingPar}, let $\rho\colon\R\to\R$ be the function which satisfies 
    for every $t\in\R$ that $\rho(t)=\max\{0,t\}$, let $n\in\N$, let $\phi^j\in\Nall$, $j\in\oneto{n}$,
    let $d_j\in\N$, $j\in\oneto{n}$, be given by $d_j=\dimin(\phi^j)$, 
    let $D\in\N$ be given by $D=\sum_{j=1}^n d_j$, and let $\Phi\in\Nall$ be given by
    $\Phi=\Pc(\phi^1,\phi^2,\dots,\phi^n)$. 
     Then 
    \begin{enumerate}[(i)]
    \item\label{NNpi} for every $x\in\R^D$ it holds 
    \small
      \begin{align}
	[\Rr(\Phi)](x)=\pa{[\Rr(\phi^1)](x_1,\dots,x_{d_1}),[\Rr(\phi^2)](x_{d_1+1},\dots,x_{d_1+d_2}),\dots,[\Rr(\phi^n)](x_{D-d_n+1},\dots,x_{D})},
      \end{align}
     \normalsize
    \item\label{NNpii} 
                    $\L(\Phi)=\max_{j\in\oneto{n}}\L(\phi^j)$,
    \item\label{NNpiii} 
                    $\M(\Phi)\leq2\pa{\sum_{j=1}^n\M(\phi^j)}+4\pa{\sum_{j=1}^n \dimout(\phi^j)}\max_{j\in\oneto{n}}\L(\phi^j)$,
    \item\label{NNpiv} 
   $\M(\Phi)=\sum_{j=1}^n\M(\phi^j)$ provided for every $j,j'\in\oneto{n}$ holds $\L(\phi^j)=\L(\phi^{j'})$,
    \item\label{NNpv} 
    $\M_{\L(\Phi)}(\Phi)\leq\sum_{j=1}^n \max\{2\dimout(\phi^j),\M_{\L(\phi^j)}(\phi^j)\}$,
    \item\label{NNpv2} 
    $\M_1(\Phi)=\sum_{j=1}^n\M_1(\phi^j)$,
    \item\label{NNpvi} 
    $\dimin(\Phi)=\sum_{j=1}^n\dimin(\phi^j)$, and
    \item\label{NNpvii} $\dimout(\Phi)=\sum_{j=1}^n\dimout(\phi^j)$.
    \end{enumerate}
  \end{lemma}  
  
\begin{proof}[Proof of \Cref{NNpar}]
    Observe that \Cref{NNconc} implies that for every $j\in\oneto{n}$ holds
    \begin{align}
      \Rr(\mathcal{E}_{\L(\Phi)}(\phi^j))=\Rr(\phi^j).
\end{align}
Combining this with \eqref{NNSPPcdef} and \eqref{NNSPPdef} establishes \eqref{NNpi}. 
Furthermore, note that that \eqref{NNpii}, \eqref{NNpv2}, \eqref{NNpvi}, and \eqref{NNpvii} 
follow directly from \eqref{NNSPPcdef} and \eqref{NNSPPdef}.
Moreover, \eqref{NNSPPcdef} demonstrates that for every 
$m\in\N$, $\psi_i\in\Nall$, $i\in\oneto{m}$, 
with $\forall i,i'\in\oneto{m}\colon\L(\psi^i)=\L(\psi^{i'})$ holds
    \begin{align}\label{NNSP2}
      \M(\Pc_s(\psi^1,\psi^2,\dots,\psi^m))=\sum_{i=1}^m\M(\psi^i).
    \end{align}
This establishes \eqref{NNpiv}. 
Next, observe that \Cref{NNconc}, \eqref{NNSPEdef}, 
and the fact that for every $d\in$, $L\in\N$ holds $\M(\Phi^{\Id}_{d,L})\leq 2dL$ 
imply that for every $j\in\oneto{n}$ we have
\begin{align}\begin{split}
      \M(\mathcal{E}_{\L(\Phi)}(\phi^j))
      &\leq 2\M(\Phi^{\Id}_{\dimout(\phi^j),\L(\Phi)-\L(\phi^j)})+2\M(\phi^j)\\
      &\leq 4\dimout(\phi^j)\L(\Phi)+2\M(\phi^j).
      \end{split}
\end{align}
Combining this with \eqref{NNSP2} establishes \eqref{NNpiii}. 
In addition, note that \eqref{NNSPId1}, \eqref{NNSPId2}, and \eqref{NNSPEdef} 
ensure for every $j\in\oneto{n}$ that
\begin{align}
  \M_{\L(\Phi)}(\mathcal{E}_{\L(\Phi)}(\phi^j)) 
    \leq\max\{2\dimout(\phi^j),\M_{\L(\phi^j)}(\phi^j)\}.
\end{align}
Combining this with \eqref{NNSPPcdef} establishes \eqref{NNpv}.
    The proof of \Cref{NNpar} is thus completed.
  \end{proof}

 \section{Basic Expression Rate Results}\label{sec:basicexpr}
  
Here we begin by establishing an expression rate result for a very simple function, namely $x\mapsto x^2$ on $[0,1]$.
Our approach is based on the observation by M. Telgarsky \cite{Telgarsky}, that neural networks with ReLU activation function can efficiently compute high-frequent sawtooth functions, and the idea of D. Yarotsky in \cite{Yarotsky} to use this in order to approximate the function $x\mapsto x^2$ by networks computing its linear interpolations.
This can then be used to derive networks capable of efficiently approximating $(x,y)\mapsto xy$, which leads to tensor products as well as polynomials and subsequently smooth function.
Note that \cite{Yarotsky} uses a slightly different notion of neural networks, where connections between non-adjacent layers are permitted. This does, however, only require a technical modification of the proof, which does not significantly change the result.
Nonetheless, the respective proofs are provided in the appendix for completeness.
	
\begin{lemma}\label{NNsquare}
Assume \Cref{NNSetting} and let $\rho\colon\R\to\R$ be the ReLU activation function given by $\rho(t)=\max\{0,t\}$.
Then there exist neural networks $(\sigma_{\eps})_{\eps\in(0,\infty)}\subseteq\Nall$ 
such that for every $\eps\in(0,\infty)$ 
    \begin{enumerate}[(i)]
    \item\label{NNsi}
        $\L(\sigma\leps)\leq\begin{cases}\tfrac{1}{2}\abs{\log_2(\eps)}+1 & \colon \eps<1\\1 & \colon \eps\geq 1\end{cases}$,
    \item\label{NNsii}
        $\M(\sigma\leps)\leq\begin{cases}15(\tfrac{1}{2}\abs{\log_2(\eps)}+1) & \colon \eps< 1\\0 & \colon \eps\geq 1\end{cases}$,
    \item\label{NNsiii}
	$\sup_{t\in[0,1]}\abs{t^2-\brac{R_{\relu}(\sigma_{\eps})}\!(t)}\leq \eps$,
    \item\label{NNt=0}
       $[R_{\relu}(\sigma_{\eps})]\!(0) = 0$.
    \end{enumerate}
  \end{lemma}

We can now derive the following result on approximate multiplication by neural networks, by observing that $xy=2B^2(|(x+y)/2B|^2-|x/2B|^2-|y/2B|^2)$ for every $B\in(0,\infty)$, $x,y\in\R$. 
  
\begin{lemma}\label{NNmult}
 Assume \Cref{NNSetting}, let $B\in(0,\infty)$, and let $\rho\colon\R\to\R$ be the ReLU activation function given by $\rho(t)=\max\{0,t\}$.
Then there exist neural networks $(\mu_{\eps})_{\eps\in(0,\infty)}\subseteq\Nall$ 
which satisfy for every $\eps\in(0,\infty)$ that
    \begin{enumerate}[(i)]
    \item\label{NNmi}
        $\L(\mu\leps)\leq\begin{cases}\tfrac{1}{2}\log_2(\tfrac{1}{\eps})+\log_2(B)+6
                 & \colon \eps < B^2\\1 & \colon \eps\geq B^2\end{cases}$,

    \item\label{NNmii}
        $\M(\mu\leps)\leq\begin{cases}45\log_2(\tfrac{1}{\eps})+90\log_2(B)+259
                            & \colon \eps < B^2\\0 & \colon \eps\geq B^2\end{cases}$,

    \item\label{NNmiii}
	$\sup_{(x,y)\in[-B,B]^2}\abs{xy-\brac{R_{\relu}(\mu_{\eps})}\!(x,y)}\leq \eps$,
    \item\label{NNmiv}

$\M_1(\mu_{\eps})=8,\ \M_{\L(\mu_{\eps})}(\mu_{\eps})=3$, and

    \item\label{NNmv}
	for every $x\in\R$ it holds that $R_\relu[\mu_{\eps}](0,x) = R_\relu[\mu_{\eps}](x,0)=0$.
    \end{enumerate}
  \end{lemma}
Next we extend this result to products of any number of factors
by hierarchical, pairwise multiplication.  
  
\begin{theorem}\label{NNbigMult}
Assume \Cref{NNSetting}, let $\rho\colon\R\to\R$ be the ReLU activation function given by $\rho(t)=\max\{0,t\}$, let $m\in\N\cap [2,\infty)$, and let $B\in[1,\infty)$.
Then there exists a constant $C\in\R$ (which is independent of $m$, $B$) and neural networks ${(\Pi_{\eps})_{\eps\in(0,\infty)}\subseteq\Nall}$ 
which satisfy 
    \begin{enumerate}[(i)]
    \item\label{NNbMi}
 	$\L(\Pi_{\eps})\leq C\ln(m)\pa{\abs{\ln(\eps)}+m\ln(B)+\ln(m)}$,
    \item\label{NNbMii}
 	$\M(\Pi_{\eps})\leq C m\pa{\abs{\ln(\eps)}+m\ln(B)+\ln(m)}$,
    \item\label{NNbMiii}
	$\displaystyle\sup_{x\in[-B,B]^m}\abs{\brac{\prod_{j=1}^m x_j}-\brac{R_{\relu}(\Pi_{\eps})}\!(x)}\leq \eps$, and
    \item\label{NNbMiv}
      $R_\relu\left[\Pi_{\eps}\right](x_1,x_2,\dots,x_m)=0$, if there exists $i\in\{1,2,\dots,m\}$ with $x_i=0$.
    \end{enumerate}
  \end{theorem}

\begin{proof}[Proof of \Cref{NNbigMult}]
Throughout this proof assume \Cref{NNSettingPar}, 
let $l=\lceil\log_2 m\rceil$, and let $\theta\in\Nn^{1,1}_1$ be the neural network given by $\theta=(0,0)$,
let $(A,b)\in\R^{l\times m}\times\R^{l}$ be the matrix-vector tuple given by 
\begin{align}
A_{i,j}=\begin{cases}1 & \colon i=j, j\leq m \\ 0 & \colon \mathrm{else} \end{cases}\quad\mathrm{and}\quad 
b_i=\begin{cases} 0 & \colon i\leq m\\ 1 & \colon i>m\end{cases}.
\end{align}
Let further $\omega\in\Nn^{m,2^l}_2$ be the neural network given by $\omega=((A,b))$.
Note that \Cref{NNmult} (with $B^m$ as $B$ in the notation of \Cref{NNmult}) ensures that 
there exist neural networks $(\mu_{\eta})_{\eta\in(0,\infty)}\subseteq\Nall$ 
such that for every $\eta\in(0,\brac{B^m}^2)$ it holds 
    \begin{enumerate}[(A)]
    \item\label{NNbMA}
      $\L(\mu_{\eta})\leq\tfrac{1}{2}\log_2(\tfrac{1}{\eta})+\log_2(B^m)+6$,

    \item\label{NNbMB}
      $\M(\mu_{\eta})\leq 45\log_2(\tfrac{1}{\eta})+90\log_2(B^m)+259$,
      
    \item\label{NNbMC}
	$\displaystyle \sup_{x,y\in[-B^m,B^m]}\abs{xy-\brac{R_{\relu}(\mu_{\eta})}\!(x,y)}\leq \eta$,
   
    \item\label{NNbMD} 
     $\M_1(\mu_{\eta})=8,\ \M_{\L(\mu_{\eta})}(\mu_{\eta})=3$, and

    \item\label{NNbME}
      for every $x\in\R$ it holds that $R_\relu[\mu_{\eta}](0,x) = R_\relu[\mu_{\eta}](x,0)=0$.
    \end{enumerate}
Let $(\nu_{\eps})_{\eps\in(0,\infty)}\subseteq\Nall$ 
be the neural networks which satisfy for every $\eps\in(0,\infty)$ 
  \begin{align}\label{NNbMnuDef}
      \nu\leps=\mu_{m^{-2}B^{-2m}\eps}.
  \end{align}
Observe that \eqref{NNbMA} implies that for every $\eps\in(0,B^m)\subseteq(0,m^2 B^{4m})$ it holds 
    \begin{align}\begin{split}\label{NNbMnuLest}
      \L(\nu\leps)&\leq\tfrac{1}{2}\log_2(\tfrac{1}{m^{-2}B^{-2m}\eps})+\log_2(B^m)+6\\
      &=\tfrac{1}{2}(\log_2(\tfrac{1}{\eps})+2\log_2(m)+2m\log_2(B))+m\log_2(B)+6\\
      &=\tfrac{1}{2}\log_2(\tfrac{1}{\eps})+2m\log_2(B)+\log_2(m)+6.
    \end{split}\end{align}
In addition, note that \eqref{NNbMB} implies that for every $\eps\in(0,B^m)\subseteq(0,m^2 B^{4m})$ 
    \begin{align}\begin{split}\label{NNbMnuMest}
      \M(\nu\leps)&\leq 45\log_2(\tfrac{1}{m^{-2}B^{-2m}\eps})+90\log_2(B^m)+259\\
      &=45\log_2(\tfrac{1}{\eps})+180m\log_2(B)+90\log_2(m)+259.
    \end{split}\end{align}    
Furthermore, \eqref{NNbMC} implies that for every $\eps\in(0,B^m)\subseteq(0,m^2 B^{4m})$ holds
    \begin{align}\label{NNbMnuEst}
      \sup_{x,y\in[-B^m,B^m]}\abs{xy-\brac{R_{\relu}(\nu_{\eta})}\!(x,y)}\leq m^{-2}B^{-2m}\eps.
    \end{align}    
Let $\pi_{k,\eps}\in\Nall$, $\eps\in(0,\infty)$, $k\in\N$, 
be the neural networks which satisfy for every $\eps\in(0,\infty)$, $k\in\N$
    \begin{align}\label{NNbMpiDef}
      \pi_{k,\eps}=\begin{cases}\nu\leps & \colon k=1\\
             \nu\leps\odot \Pc(\pi_{k-1,\eps},\pi_{k-1,\eps}) & \colon k>1
            \end{cases}
    \end{align}
and let $(\Pi\leps)_{\eps\in(0,\infty)}\subseteq\Nall$ be neural networks given by 
    \begin{align}\label{NNbMPiDef}
      \Pi\leps=\begin{cases}\pi_{l,\eps}\odot\omega & \colon \eps<B^m \\ \theta & \colon\eps\geq B^m \end{cases}.
    \end{align}
Note that for every $\eps\in(B^m,\infty)$ it holds 
  \begin{align}\begin{split}\label{NNbMiiiepsL}
      \sup_{x\in[-B,B]^m}\abs{\brac{\tprod_{j=1}^m x_j}-\brac{R_{\relu}(\Pi_{\eps})}\!(x)}
      &=\sup_{x\in[-B,B]^m}\abs{\brac{\tprod_{j=1}^m x_j}-\brac{R_{\relu}(\theta)}\!(x)}\\
      &=\sup_{x\in[-B,B]^m}\abs{\brac{\tprod_{j=1}^m x_j}-0} 
      =B^m\leq\eps.
    \end{split}\end{align} 
    We claim that for every $k\in\oneto{l}$, $\eps\in(0,B^m)$ it holds 
    \begin{enumerate}[(a)]
    \item\label{NNbMClaima}
      that
      \begin{align}\label{NNbMClaimEq}
	\sup_{x\in[-B,B]^{(2^k)}}\abs{\brac{\tprod_{j=1}^{2^k}x_j}-[R_{\relu}(\pi_{k,\eps})](x)}
            \leq 4^{k-1} m^{-2} B^{(2^k-2m)}\eps,
      \end{align}
    \item\label{NNbMClaimb}
      that $\L(\pi_{k,\eps})\leq k\L(\nu\leps)$, and
    \item\label{NNbMClaimc}
      that $\M(\pi_{k,\eps})\leq (2^k-1)\M(\nu\leps)+(2^{k-1}-1)20$. 
    \end{enumerate}  
We prove \eqref{NNbMClaima}, \eqref{NNbMClaimb}, and \eqref{NNbMClaimc} by induction on $k\in\oneto{l}$.
Observe that \eqref{NNbMnuEst} and the fact that $B\in[1,\infty)$ establishes \eqref{NNbMClaima} for $k=1$. 
Moreover, note that \eqref{NNbMpiDef} establishes \eqref{NNbMClaimb} and \eqref{NNbMClaimc} in the base case $k=1$. 

For the induction step $\oneto{l-1}\ni k\to k+1\in\{2,3,\dots,l\}$ 
note that \Cref{NNconc}, \Cref{NNpar}, \eqref{NNbMnuEst} and \eqref{NNbMpiDef} 
imply that for every $k\in\oneto{l-1}$, $\eps\in(0,B^m)$ 
\begin{align}\begin{split}\label{NNbMT3}
&\quad\sup_{x\in[-B,B]^{(2^{k+1})}}\abs{\brac{\prod_{j=1}^{2^{k+1}}x_j}-[R_{\relu}(\pi_{k+1,\eps})](x)}
\\
&=\sup_{x,x'\in[-B,B]^{(2^k)}}\abs{\brac{\prod_{j=1}^{2^k}x_j}\!\!\brac{\prod_{j=1}^{2^k}x'_j}-[R_{\relu}(\pi_{k+1,\eps})]\pa{(x,x')}}
\\
&=\sup_{x,x'\in[-B,B]^{(2^k)}}\abs{\brac{\prod_{j=1}^{2^k}x_j}\!\!\brac{\prod_{j=1}^{2^k}x'_j}-[R_{\relu}(\nu\leps)]\pa{[R_{\relu}(\pi_{k,\eps})](x),[R_{\relu}(\pi_{k,\eps})](x')}}
\\
&\leq\sup_{x,x'\in[-B,B]^{(2^k)}}\abs{\brac{\prod_{j=1}^{2^k}x_j}\!\!\brac{\prod_{j=1}^{2^k}x'_j}-\pa{[R_{\relu}(\pi_{k,\eps})](x)}\pa{[R_{\relu}(\pi_{k,\eps})](x')}}
\\
&\quad\,+\!\!\!\!\!\!\sup_{x,x'\in[-B,B]^{(2^k)}}\abs{\pa{[R_{\relu}(\pi_{k,\eps})](x)}\pa{[R_{\relu}(\pi_{k,\eps})](x')}-[R_{\relu}(\nu\leps)]\pa{[R_{\relu}(\pi_{k,\eps})](x),[R_{\relu}(\pi_{k,\eps})](x')}}
\\
&\leq\sup_{x,x'\in[-B,B]^{(2^k)}}\abs{\brac{\prod_{j=1}^{2^k}x_j}\!\!\brac{\prod_{j=1}^{2^k}x'_j}-\pa{[R_{\relu}(\pi_{k,\eps})](x)}\pa{[R_{\relu}(\pi_{k,\eps})](x')}} + m^{-2}B^{-2m}\eps.
\end{split}
\end{align}
Next,  for  every $c,\delta\in(0,\infty)$, $y,z\in[-c,c]$, $\tilde{y},\tilde{z}\in\R$ 
with $\abs{y-\tilde{y}}, \abs{z-\tilde{z}}\leq\delta$ it holds
    \begin{align}\label{NNbMT2}
      \abs{yz-\tilde{y}\tilde{z}}\leq 2(\abs{y}+\abs{z})\delta  + \delta^2\leq 2c\delta+\delta^2.
    \end{align}
    Moreover, for every $k\in\oneto{l}$ 
    \begin{align}\label{NNbMT1}
      4^{k-1}\leq 4^{l-1}=4^{\lceil\log_2 m\rceil-1}\leq 4^{\log_2 m}=m^2.
    \end{align}
    The fact that $B\in[1,\infty)$ therefore ensures that for every $k\in\oneto{l-1}$, $\eps\in(0,B^m)$ 
    \begin{align}\begin{split}
     \brac{4^{k-1} m^{-2} B^{(2^k-2m)}\eps}^2
      =\brac{4^{k-1} m^{-2} B^{(2^{k+1}-2m)}\eps}\brac{4^{k-1} m^{-2} B^{-2m}\eps}
      \leq\brac{4^{k-1} m^{-2} B^{(2^{k+1}-2m)}\eps}.
    \end{split}\end{align}
    This and \eqref{NNbMT2} imply that for every 
     $k\in\oneto{l-1}$, $\eps\in(0,B^m)$, $x,x'\in[-B,B]^{(2^k)}$
    \begin{align}\begin{split}
      &\quad\abs{\brac{\tprod_{j=1}^{2^k}x_j}\!\!\brac{\tprod_{j=1}^{2^k}x'_j}-\pa{[R_{\relu}(\pi_{k,\eps})](x)}\pa{[R_{\relu}(\pi_{k,\eps})](x')}}\\
      &\leq 2B^{(2^k)}4^{k-1} m^{-2} B^{(2^k-2m)}\eps+\brac{4^{k-1} m^{-2} B^{(2^k-2m)}\eps}^2\\
      &\leq 3\brac{4^{k-1} m^{-2} B^{(2^{k+1}-2m)}\eps}.
    \end{split}\end{align}
    Combining this, \eqref{NNbMT3}, and the fact that $B\in[1,\infty)$ 
    demonstrates that for every $k\in\oneto{l-1}$, $\eps\in(0,B^m)$
    \begin{align}\begin{split}
      &\quad\sup_{x\in[-B,B]^{(2^{k+1})}}\abs{\brac{\tprod_{j=1}^{2^{k+1}}x_j}-[R_{\relu}(\pi_{k+1,\eps})](x)}\\
      &\leq 3\brac{4^{k-1} m^{-2} B^{(2^{k+1}-2m)}\eps}+m^{-2}B^{-2m}\eps\\\
      &\leq 4^k m^{-2} B^{(2^{k+1}-2m)}\eps.
    \end{split}\end{align}
    This establishes the claim \eqref{NNbMClaima}. 
    Moreover, \Cref{NNconc} and \Cref{NNpar} imply 
    for every $k\in\oneto{l-1}$, $\eps\in(0,B^m)$ with $\L(\pi_{k,\eps})\leq k\L(\nu\leps)$ holds
    \begin{align}\begin{split}
      \L(\pi_{k+1,\eps})&=\L(\nu\leps)+\max\{\L(\pi_{k,\eps}),\L(\pi_{k,\eps})\}\\
      &\leq \L(\nu\leps) + k\L(\nu\leps)=(k+1)\L(\nu\leps).
    \end{split}\end{align}
This establishes the claim \eqref{NNbMClaimb}. 
Furthermore, \Cref{NNconc}, \Cref{NNpar}, \eqref{NNbMB}, and \eqref{NNbMD}
imply for every $k\in\oneto{l-1}$, $\eps\in(0,B^m)$ 
with $\M(\pi_{k,\eps})\leq (2^k-1)\M(\nu\leps)+(2^{k-1}-1)20$ holds
\begin{align}\begin{split}
\M(\pi_{k+1,\eps})&\leq\M(\nu\leps)+(\M(\pi_{k,\eps})+\M(\pi_{k,\eps}))+\M_1(\nu_{\eps})+\M_{\L(\Pc(\pi_{k,\eps},\pi_{k,\eps}))}(\Pc(\pi_{k,\eps},\pi_{k,\eps})) \\
  &\leq\M(\nu\leps)+2\M(\pi_{k,\eps})+14+2\M_{\L(\nu_{\eps})}(\nu_{\eps})\leq\M(\nu\leps)+2\M(\pi_{k,\eps})+20\\
  &\leq \M(\nu\leps)+2((2^k-1)\M(\nu\leps)+(2^{k-1}-1)20)+20\\
  &=(2^{k+1}-1)\M(\nu\leps)+(2^k-1)20.
    \end{split}\end{align}
This establishes the claim \eqref{NNbMClaimc}. 

Combining \eqref{NNbMClaima} with \Cref{NNconc} and \eqref{NNbMPiDef} 
implies for every $\eps\in(0,B^m)$ the bound
    \small
    \begin{align}\begin{split}
      \sup_{x\in[-B,B]^m}\abs{\brac{\prod_{j=1}^m x_j}-\brac{R_{\relu}(\Pi_{\eps})}\!(x)}&\leq\!\!\!\!\!\sup_{x\in[-B,B]^{(2^l)}}\!\abs{\brac{\prod_{j=1}^{2^l} x_j}-\brac{R_{\relu}(\pi_{l,\eps})}\!(x)}\\
      &\leq 4^{l-1} m^{-2} B^{(2^l-2m)}\eps\\
      &\leq 4^{\lceil\log_2(m)\rceil-1} m^{-2} B^{(2^{\lceil\log_2(m)\rceil}-2m)}\eps\\
      &\leq 4^{\log_2(m)} m^{-2} B^{(2^{\log_2(m)+1}-2m)}\eps\\
      &\leq \brac{2^{\log_2(m)}}^2 m^{-2}B^{(2m-2m)}\eps\leq\eps.
    \end{split}\end{align} 
    \normalsize
This and \eqref{NNbMiiiepsL} establish that 
the neural networks $(\Pi\leps)_{\eps\in(0,\infty)}$ satisfy \eqref{NNbMiii}.
Combining \eqref{NNbMClaimb} with \Cref{NNconc}, \eqref{NNbMnuLest}, and \eqref{NNbMPiDef} 
ensures that for every $\eps\in(0,B^m)$ 
    \small
    \begin{align}\begin{split}
      \L(\Pi\leps)&=\L(\pi_{l,\eps})+\L(\omega)\leq l\L(\nu\leps)+1\leq (\log_2(m)+1)\L(\nu\leps)+1\\
      &\leq\log_2(m)\log_2(\tfrac{1}{\eps})+4\log_2(m)m\log_2(B)+2[\log_2(m)]^2+12\log_2(m)+1.
    \end{split}\end{align}
    \normalsize
    and that for every $\eps\in(B^m,\infty)$ it holds $\L(\Pi\leps)=\L(\theta)=1$. 
This establishes that the neural networks $(\Pi\leps)_{\eps\in(0,\infty)}$ satisfy \eqref{NNbMi}. 
Furthermore, note that \eqref{NNbMClaimc}, \Cref{NNconc}, \eqref{NNbMnuLest}, and \eqref{NNbMPiDef} 
demonstrate that for every $\eps\in(0,B^m)$ 
    \begin{align}\begin{split}
      \M(\Pi\leps)&\leq 2(\M(\pi_{l,\eps}) + \M(\omega))\leq 2\brac{(2^l-1)\M(\nu\leps)+(2^{l-1}-1)20}+4m\\
      &\leq 2^{l+1}\M(\nu\leps)+(2^l)20+4m\leq 4m\M(\nu\leps)+44m\\
      &\leq 180m\log_2(\tfrac{1}{\eps})+720m^2\log_2(B)+360m\log_2(m)+1080m.
    \end{split}\end{align}
and that for every $\eps\in(B^m,\infty)$ holds $\M(\Pi\leps)=\M(\theta)=0$. 
  This establishes that the neural networks $(\Pi\leps)_{\eps\in(0,\infty)}$ 
  satisfy \eqref{NNbMii}. 
  Note that \eqref{NNbMiv} follows from \eqref{NNbME} by construction.
  The proof of \Cref{NNbigMult} is thus completed.
  \end{proof}
	
With the above established, 
it is quite straightforward to get the following result for the approximation of tensor products. 
Note that the exponential term $B^{m-1}$ in \eqref{NN2iii} 
is unavoidable as result from multiplying $m$ many inaccurate values of magnitude $B$. 
For our purposes this will not be an issue since the functions 
we consider are bounded in absolute value by $B=1$. 
This is further not an issue in cases, where the $h_j$ can be 
approximated by networks whose size scales logarithmically with $\eps$. 

 \begin{prop}\label{NN2}
Assume \Cref{NNSettingPar},
let $\rho\colon\R\to\R$ be the ReLU activation function given by $\rho(t)=\max\{0,t\}$,
let $B\in[1,\infty)$, $m\in\N$, for every $j\in\oneto{m}$ 
let $d_j\in\N$, $\Omega_j\subseteq \R^{d_j}$, and $h_j:\Omega_j\to[-B,B]$, 
let $(\Phi^j_{\eps})_{\eps\in(0,\infty)}\in\Nall$, $j\in\oneto{m}$, 
be neural networks which satisfy for every $\eps\in(0,\infty)$, $j\in\oneto{m}$ 
    \begin{align}\label{NN2Ass}
      \sup_{t\in\Omega_j}\abs{h_j(x)-\brac{\Rr(\Phi^j_{\eps})}(x)}\leq\eps,
    \end{align}
let $\Phi^{\Pc}_{\eps}\in\Nall$, $\eps\in(0,\infty)$ be given by $\Phi^{\Pc}_{\eps}=\Pc(\Phi^1_{\eps},\Phi^2_{\eps},\dots,\Phi^m_{\eps})$, 
and let $L_{\eps}\in\N$, $\eps\in(0,\infty)$ 
be given by $L_{\eps}=\max_{j\in\oneto{m}}\L(\Phi^j_{\eps})$.
\\
Then there exists a constant $C\in\R$ ( which is independent of $m,B,\eps$) and neural networks 
$(\Psi_{\eps})_{\eps\in(0,\infty)}\subseteq\Nall$ 
which satisfy
    \begin{enumerate}[(i)]
    \item\label{NN2i}
	$\L(\Psi_{\eps})\leq C\ln(m)\pa{\abs{\ln(\eps)}+m\ln(B)+\ln(m)}+L_{\eps}$,  
    \item \label{NN2ii}
	$\M(\Psi\leps) \leq C m\pa{\abs{\ln(\eps)}+m\ln(B)+\ln(m)}+\M(\Phi^{\Pc}_{\eps})+\M_{L_{\eps}}(\Phi^{\Pc}_{\eps})$,
      and
    \item \label{NN2iii}
      $\displaystyle \sup_{t=(t_1,t_2,\dots,t_m)\in\times_{j=1}^m \Omega_j}\abs{\brac{\tprod_{j=1}^m h_j(t_j)}-\brac{R_{\relu}(\Psi_{\eps})}\!(t)}
       \leq 3mB^{m-1}\eps.$
    \end{enumerate}
  \end{prop}

  \begin{proof}[Proof of \Cref{NN2}]
   In the case of $m=1$ the neural networks $(\Phi^1_{\eps})_{\eps\in(0,\infty)}\in\Nall$ 
   satisfy \eqref{NN2i}, \eqref{NN2ii}, and \eqref{NN2iii} by assumption.  
   Throughout the remainder of this proof assume $m\geq 2$, 
   and let $\theta\in\Nn^{1,1}_1$ denote the trivial neural network $\theta=(0,0)$.
   Observe that \Cref{NNbigMult} (with $\eps\leftrightarrow\eta$, $C'\leftrightarrow C$ in the 
   notation \Cref{NNbigMult}) ensures that there exist $C'\in\R$ 
   and neural networks $(\Pi_{\eta})_{\eta\in(0,\infty)}\subseteq\Nall$ 
   which satisfy for every $\eta\in(0,\infty)$ that
    \begin{enumerate}[(a)]
    \item\label{NNbMa}
 	$\L(\Pi_{\eta})\leq C'\ln(m)\pa{\abs{\ln(\eta)}+m\ln(B)+\ln(m)}$,
    \item\label{NNbMb}
 	$\M(\Pi_{\eta})\leq C' m\pa{\abs{\ln(\eta)}+m\ln(B)+\ln(m)}$,
      and       
    \item\label{NNbMc}
	$\displaystyle\sup_{x\in[-B,B]^m}\abs{\brac{\prod_{j=1}^m x_j}-\brac{R_{\relu}(\Pi_{\eta})}\!(x)}\leq \eta$.
    \end{enumerate} 
    Let $(\Psi_{\eps})_{\eps\in(0,\infty)}\subseteq\Nall$ be the neural networks 
    which satisfy for every $\eps\in(0,\infty)$ that
    \begin{align}
      \Psi_{\eps}=\begin{cases}\Pi_{\eps}\odot \Pc(\Phi^1_{\eps},\Phi^2_{\eps}, \dots, \Phi^m_{\eps}) & \colon \eps<\tfrac{B}{2m}\\ \theta & \colon \eps\geq\tfrac{B}{2m}\end{cases}.
    \end{align}
    Note that for every $\eps\in(0,\tfrac{B}{2m})$ 
    \begin{align}\begin{split}
      \max_{\overset{x\in[-B,B]^m,x'\in\R^m}{\norm{x'-x}_{\infty}\leq\eps}}\abs{\prod_{j=1}^m x'_j-\prod_{j=1}^m x_j}
      &=(B+\eps)^m-B^m=\sum_{k=1}^m \binom{m}{k}B^{m-k}\eps^k\leq\eps\sum_{k=1}^m\frac{m^k}{k!}B^{m-k}\eps^{k-1}\\
      &\leq\eps\sum_{k=1}^m\frac{m^k}{k!}B^{m-k}\pa{\frac{B}{2m}}^{k-1}=mB^{m-1}\eps\sum_{k=1}^m \frac{1}{2^{k-1}k!}\\
      &\leq 2mB^{m-1}\eps.
    \end{split}\end{align}
    Combining this with \Cref{NNconc}, \Cref{NNpar}, \eqref{NN2Ass}, and \eqref{NNbMc} 
    implies that for every $\eps\in(0,\tfrac{B}{2m})$, $t=(t_1,t_2,\dots,t_m)\in\Omega$ it holds
    \footnotesize
    \begin{align}\begin{split}\label{NN2errEst}
      \abs{\brac{\tprod_{j=1}^m h_j(t_j)}-\brac{R_{\relu}(\Psi_{\eps})}\!(t)}
      &=\abs{\brac{\tprod_{j=1}^m h_j(t_j)}-\brac{R_{\relu}(\Pi_{\eps}\odot \Pc(\Phi^1_{\eps},\Phi^2_{\eps},\dots,\Phi^m_{\eps}))}\!(t)}\\
      &\leq\abs{\brac{\tprod_{j=1}^m h_j(t_j)}-\brac{\tprod_{j=1}^m \brac{\Rr(\Phi^j_{\eps})}(t_j)}}\\
      &\quad+\abs{\brac{\tprod_{j=1}^m \brac{\Rr(\Phi^j_{\eps})}(t_j)}-\brac{\Rr(\Pi_{\eps})}\pa{[\Rr(\Phi^1_{\eps})](t_1),\dots,[\Rr(\Phi^m_{\eps})](t_j)}}\\
      &\leq 2mB^{m-1}\eps+\eps\leq 3mB^{m-1}\eps.
    \end{split}\end{align}
    \normalsize
    Moreover, for every 
    $\eps\in[\tfrac{B}{2m},\infty)$, $t=(t_1,t_2,\dots,t_m)\in\Omega$ 
    it holds that
    \begin{align}\begin{split}
    \abs{\brac{\tprod_{j=1}^m h_j(t_j)}-\brac{R_{\relu}(\Psi_{\eps})}\!(t)}
    &=\abs{\brac{\tprod_{j=1}^m h_j(t_j)}-\brac{R_{\relu}(\theta)}\!(t)}\\
    &=\abs{\brac{\tprod_{j=1}^m h_j(t_j)}}\leq B^m\leq 2mB^{m-1}\eps.
    \end{split}\end{align}
    This and \eqref{NN2errEst} establish that the neural networks $(\Psi_{\eps})_{\eps,c\,\in(0,\infty)}$ satisfy \eqref{NN2iii}.
    Next observe that \Cref{NNconc}, \Cref{NNpar}, and \eqref{NNbMa} 
    demonstrate that for every $\eps\in(0,\tfrac{B}{2m})$ 
    \begin{align}\begin{split}
      \L(\Psi_{\eps})&=\L(\Pi_{\eps}\odot \Pc(\Phi^1_{\eps},\Phi^2_{\eps}, \dots, \Phi^m_{\eps}))=\L(\Pi_{\eps})+\max_{j\in\oneto{m}}\L(\Phi^j_{\eps})\\
      &\leq C'\ln(m)\pa{\abs{\ln(\eps)}+m\ln(B)+\ln(m)}+L_{\eps}.
    \end{split}\end{align}
    This and the fact that for every $\eps\in[\tfrac{B}{2m},\infty)$ 
    it holds that $\L(\Psi_{\eps})=\L(\theta)=1$ establish that 
    the neural networks $(\Psi_{\eps})_{\eps,c\,\in(0,\infty)}$ satisfy \eqref{NN2i}.
    Furthermore note that \Cref{NNconc}, \Cref{NNpar}, and \eqref{NNbMb} 
    ensure that for every $\eps\in(0,\tfrac{B}{2m})$ 
    \begin{align}\begin{split}
      \M(\Psi_{\eps})&=\M(\Pi_{\eps}\odot \Pc(\Phi^1_{\eps},\Phi^2_{\eps}, \dots, \Phi^m_{\eps}))\\
      &\leq 2\M(\Pi_{\eps})+\M(\Pc(\Phi^1_{\eps},\Phi^2_{\eps}, \dots, \Phi^m_{\eps}))+\M_{\L(\Pc(\Phi^1_{\eps},\Phi^2_{\eps}, \dots, \Phi^m_{\eps}))}(\Pc(\Phi^1_{\eps},\Phi^2_{\eps}, \dots, \Phi^m_{\eps}))\\
      &\leq 2C' m\pa{\abs{\ln(\eps)}+m\ln(B)+\ln(m)}+\M(\Phi^{\Pc}_{\eps})+\M_{L_{\eps}}(\Phi^{\Pc}_{\eps}).
    \end{split}\end{align}
This and the fact that for every $\eps\in[\tfrac{B}{2m},\infty)$ 
it holds that $\M(\Psi_{\eps})=\M(\theta)=0$ 
imply the neural networks $(\Psi_{\eps})_{\eps,c\,\in(0,\infty)}$ satisfy \eqref{NN2ii}.
The proof of \Cref{NN2} is completed. 
  \end{proof}
  
 Another way to use the multiplication results is to consider the approximation 
of smooth functions by polynomials. 
This can be done for functions of arbitrary dimension using the 
multivariate Taylor expansion (see \cite{Yarotsky} and \cite[Thm. 2.3]{Mhaskar}). 
Such a direct approach, however, yields networks whose size depends exponentially 
on the dimension of the function. 
As our goal is to show that high-dimensional functions with a tensor product structure
can be approximated by networks with only polynomial dependence on the dimension, 
we only consider univariate smooth functions here. 
In the appendix we present a detailed and explicit construction of this 
Taylor approximation by neural networks. 
In the following results we employ an auxiliary parameter $r$, so that the bounds on the depth and connectivity of the networks may be stated for all $\eps\in(0,\infty)$. Note that this parameter does not influence the construction of the networks themselves.
\begin{theorem}\label{NNsmoothFunctions}
Assume \Cref{NNSetting}, let $n\in\N$, $r\in(0,\infty)$, 
let $\rho\colon\R\to\R$ be the ReLU activation function given by $\rho(t)=\max\{0,t\}$, and let $B^n_1\subseteq C^n([0,1],\R)$ 
be the set given by
    \begin{align}
      B^n_1=\left\{f\in C^n([0,1],\R)\colon \max_{k\in\{0,1,\dots,n\}}\brac{\sup_{t\in[0,1]}\abs{f^{(k)}(t)}}\leq 1\right\}.
    \end{align}
Then there exist neural networks 
$(\Phi_{f,\eps})_{f\in B^n_1,\eps\in(0,\infty)}\subseteq\Nall$ 
which satisfy
    \begin{enumerate}[(i)]
    \item\label{NNsFi1} 
	$\displaystyle\sup_{f\in B^n_1,\eps\in(0,\infty)}\brac{\frac{\L(\Phi_{f,\eps})}{\max\{r,\abs{\ln(\eps)}\}}}<\infty$, 
    \item\label{NNsFi2} 
	$\displaystyle \sup_{f\in B^n_1,\eps\in(0,\infty)}\brac{\frac{\M(\Phi_{f,\eps})}{\eps^{-\frac{1}{n}}\max\{r,|\ln(\eps)|\}}}<\infty$,
      and
    \item\label{NNsFii}
      for every $f\in B^n_1$, $\eps\in(0,\infty)$ that
      \begin{align}
	\sup_{t\in[0,1]}\abs{f(t)-\brac{R_{\relu}(\Phi_{f,\eps})}\!(t)}\leq \eps.
      \end{align}
    \end{enumerate}
  \end{theorem}

For convenience of use we also provide the following more general corollary.

  \begin{cor}\label{Yarotsky2}
    Assume \Cref{NNSetting}, let $r\in(0,\infty)$ and let $\rho\colon\R\to\R$ be the ReLU activation function given by $\rho(t)=\max\{0,t\}$. Let further the set $\CN$ be given by $\mathcal{C}^n=\cup_{[a,b]\subseteq\R_+}C^n([a,b],\R)$,
    and let $\norm{\cdot}_{n,\infty}\colon \CN \to [0,\infty)$ satisfy for every $[a,b]\subseteq\R_+$, $f\in C^n([a,b],\R)$ 
    \begin{align}
      \norm{f}_{n,\infty}=\max_{k\in\{0,1,\dots,n\}}\brac{\sup_{t\in[a,b]}\abs{f^{(k)}(t)}}.
    \label{Y2ninfnorm1}
    \end{align}
    Then there exist neural networks $\pa{\Phi_{f,\eps}}_{f\in \CN,\eps\in(0,\infty)}\subseteq\Nall$ which satisfy 
    \begin{enumerate}[(i)]
    \item\label{Y2i}
	$\displaystyle \sup_{f\in\CN, \eps\in(0,\infty)}\brac{\frac{\L(\Phi_{f,\eps})}{\max\{r,|\ln(\frac{\eps}{\max\{1,b-a\}\norm{f}_{n,\infty}})|\}}}<\infty$,       
    \item\label{Y2ii}
	$\displaystyle \sup_{f\in\CN, \eps\in(0,\infty)}\brac{\frac{\M(\Phi_{f,\eps})}{\max\{1,b-a\}\norm{f}_{n,\infty}^{\frac{1}{n}}\eps^{-\frac{1}{n}}\max\{r,|\ln(\frac{\eps}{\max\{1,b-a\}\norm{f}_{n,\infty}})|\}}}<\infty$,      
      and
    \item \label{Y2iii}
      for every $[a,b]\subseteq\R_+$, $f\in C^n([a,b],\R)$, $\eps\in(0,\infty)$ that
      \begin{align}
	\sup_{t\in[a,b]}\abs{f(t)-\brac{R_{\relu}(\Phi_{f,\eps})}\!(t)}\leq\eps.
      \end{align}
    \end{enumerate}
  \end{cor}

\section{DNN Expression Rates for High-Dimensional Basket prices} \label{sec:optionapprox}

Now that we have established a number of general expression rate results, 
we can apply them to our specific problem. Using the regularity result \eqref{PrelNN1} we obtain the following. 

 \begin{cor}\label{NN1}
	Assume \Cref{NNSetting}, let $n\in\N$, $r\in(0,\infty)$, $a\in(0,\infty)$, $b\in(a,\infty)$, 
        let $\rho\colon\R\to\R$ be the ReLU activation function given by $\rho(t)=\max\{0,t\}$,
	let $f\colon(0,\infty)\to\R$ be as defined in \eqref{Prel0fDef}, and let $h_{c,K}\colon[a,b]\to\R$, $c\in(0,\infty)$, $K\in[0,\infty)$, denote 
        the functions which satisfy for every $c\in(0,\infty)$, $K\in[0,\infty)$, $x\in [a,b]$ that
	\begin{align}
	h_{c,K}(x)=f(\tfrac{K+c}{x}).
	\end{align}
	Then there exist neural networks $\pa{\Phi_{\eps,c,K}}_{\eps,c\,\in(0,\infty),K\in[0,\infty)}\subseteq\Nall$ which satisfy 
	\begin{enumerate}[(i)]
		\item\label{NN1ia}
		$\displaystyle\sup_{\eps,c\in(0,\infty),K\in[0,\infty)}\brac{\frac{\L(\Phi_{\eps,c,K})}{\max\{r,|\ln(\eps)|\}+\max\{0,\ln(K+c)\}}}<\infty$,   
		\item \label{NN1ib}
		$\displaystyle\sup_{\eps,c\,\in(0,\infty),K\in[0,\infty)}\brac{\frac{\M(\Phi_{\eps,c,K})}{(K+c+1)^{\frac{1}{n}}\eps^{-\frac{1}{n^2}}}}<\infty$,      
		and
		\item \label{NN1ii}
		for every $\eps,c\in(0,\infty)$, $K\in[0,\infty)$ that
		\begin{align}
		\sup_{x\in[a,b]}\abs{h_{c,K}(x)-\brac{R_{\relu}(\Phi_{\eps,c,K})}\!(x)}\leq\eps.
		\end{align}
	\end{enumerate}
\end{cor}

\begin{proof}[Proof of \Cref{NN1}]
	We observe \Cref{PrelNN1} ensures the existence of a constant $C\in\R$ with
	\begin{align}\label{NN1T1}
	\max_{k\leq n}\sup_{x\in[a,b]}\abs{h_{c,K}^{(k)}(x)}\leq C\max\{(K+c)^n,1\}.
	\end{align}
	Moreover, observe for every $\eps,c\in(0,\infty)$, $K\in[0,\infty)$ it holds
	\begin{align}\begin{split}\label{NN1T2}
	&\quad\max\{r,|\ln(\tfrac{\eps}{\max\{1,b-a\}C\max\{(K+c)^n,1\}})|\}\\
	&\leq\max\{r,\abs{\ln(\eps)}\}+|\ln(\max\{1,b-a\})|+\abs{\ln(C\max\{(K+c)^n,1\})}\\
	&\leq\max\{r,\abs{\ln(\eps)}\}+\ln(\max\{1,b-a\})+\abs{\ln(C)}+\abs{\ln(\max\{(K+c)^n,1\})}\\
	&\leq\max\{r,\abs{\ln(\eps)}\}+\ln(\max\{1,b-a\})+\abs{\ln(C)}+n\max\{\ln(K+c),0\}\\
	&\leq n(1+\max\{1,\tfrac{1}{r}\}(|\ln(C)|+\ln(\max\{1,b-a\})))(\max\{r,\abs{\ln(\eps)}\}+\max\{\ln(K+c),0\}).
	\end{split}\end{align}
Furthermore, note for every $\eps,c\in(0,\infty)$, $K\in[0,\infty)$ it holds
\begin{align}\begin{split}
\brac{\frac{\eps}{\max\{1,b-a\}C\max\{(K+c)^n,1\}}}^{-\frac{1}{2n^2}}
 &=[\max\{1,b-a\}]^{-\frac{1}{2n^2}}\eps^{-\frac{1}{2n^2}}C^{\frac{1}{2n^2}}\max\{(K+c)^{\frac{1}{2n}},1\}\\
 &\leq[\max\{1,b-a\}]^{-\frac{1}{2n^2}} C^{\frac{1}{2n^2}}(K+c+1)^{\frac{1}{2n}}\eps^{-\frac{1}{2n^2}}.
\end{split}\end{align} 
Combining this, \eqref{NN1T1}, \eqref{NN1T2} with \Cref{Y1tech} and \Cref{Yarotsky2} 
(with $n\leftrightarrow 2n^2$ in the notation of \Cref{Yarotsky2}) completes the proof of \Cref{NN1}.
\end{proof}

We can then employ \Cref{NN2} in order to approximate the required tensor product.    

  \begin{cor}\label{NN3}
    Assume \Cref{NNSetting}, let $\rho\colon\R\to\R$ be the ReLU activation function given by $\rho(t)=\max\{0,t\}$, let $n\in\N$, $a\in(0,\infty)$, $b\in(a,\infty)$, $(K_i)_{i\in\N}\subseteq[0,\Kmax)$, 
    and consider, for $h_{c,K}\colon[a,b]\to\R$, $c\in(0,\infty)$, $K\in[0,\Kmax)$, 
    the functions which are, for every $c\in(0,\infty)$, $K\in[0,\Kmax)$, $x\in[a,b]$, given by  
    \begin{align}
      h_{c,K}(x)=\normfac\int^{\ln(\frac{K+c}{x})}_{-\infty} e^{-\frac{1}{2}r^2}\d r.
    \end{align}
    For any $c\in(0, \infty)$, $d\in\N$ let the function $F^d_c(x)\colon [a,b]^d\to\R$ be given by
    \begin{align}
      F^d_c(x)=1-\brac{\tprod_{i=1}^d h_{c,K_i}(x_i)}.
    \end{align}
    Then there exist neural networks $(\Psi^d_{\eps,c})_{\eps,c\,\in(0,\infty),d\in\N}\subseteq\Nall$ 
    which satisfy
    \begin{enumerate}[(i)]
    \item\label{NN3i}
      $\displaystyle\sup_{\eps,c\,\in(0,\infty),d\in\N}\brac{\frac{\L(\Psi^d_{\eps,c})}{\max\{1,\ln(d)\}(\abs{\ln(\eps)}+\ln(d)+1)+\ln(c+1)}}<\infty$,
    \item \label{NN3ii}
	$\displaystyle\sup_{\eps,c\,\in(0,\infty),d\in\N}\brac{\frac{\M(\Psi^d_{\eps,c})}{(c+1)^{\frac{1}{n}}d^{1+\frac{1}{n}}\eps^{-\frac{1}{n}}}}<\infty$,
      and
    \item \label{NN3iii}
      for every $\eps,c\,\in(0,\infty)$, $d\in\N$ that
      \begin{align}
	\sup_{x\in[a,b]^d}\abs{F^d_c(x)-\brac{R_{\relu}(\Psi^d_{\eps,c})}\!(x)}\leq\eps.
      \end{align}
    \end{enumerate}
  \end{cor}

  \begin{proof}[Proof of \Cref{NN3}]
    Throughout this proof assume \Cref{NNSettingPar}.
    Property \Cref{NN1} ensures there exist constants $b_L,b_M\in(0,\infty)$ and 
    neural networks $\pa{\Phi^i_{\eta,c}}_{\eta,c\,\in(0,\infty)}\subseteq\Nall$, $i\in\N$ 
    such that for every $i\in\N$ it holds
    \begin{enumerate}[(a)]
    \item\label{NN3NN1a}
	$\displaystyle\sup_{\eta,c\in(0,\infty)}\brac{\frac{\L(\Phi^i_{\eta,c})}{\max\{1,|\ln(\eta)|\}+\max\{0,\ln(\Kmax+c)\}}}<b_L$,   
    \item\label{NN3NN1b}
	$\displaystyle\sup_{\eta,c\,\in(0,\infty)}\brac{\frac{\M(\Phi^i_{\eta,c})}{(\Kmax+c+1)^{\frac{1}{n}}\eta^{-\frac{1}{n^2}}}}<b_M$,      
      and
    \item\label{NN3NN1c}
      for every $\eta,c\in(0,\infty)$ that
      \begin{align}
	\sup_{x\in[a,b]}\abs{h_{c,K_i}(x)-\brac{R_{\relu}(\Phi^i_{\eta,c})}\!(x)}\leq\eta.
      \end{align}
    \end{enumerate}  
    Furthermore, for every $c\in(0,\infty)$, $i\in\N$, $x\in[a,b]$ holds 
    \begin{align}\begin{split}\label{NN3hEst}
      \abs{h_{c,K_i}(x)}=\abs{\normfac\int^{\ln(\frac{K_i+c}{x})}_{-\infty} e^{-\frac{1}{2}r^2}\d r}\leq\normfac\abs{\int^{\infty}_{-\infty} e^{-\frac{1}{2}r^2}\d r}=1.
    \end{split}\end{align}
    Combining this with \eqref{NN3NN1a} and \Cref{NN2} and \Cref{NNpar} implies 
    there exist $C\in\R$ and neural networks 
    $(\psi^d_{\eta,c})_{\eta\in(0,\infty)}\subseteq\Nall$, $c\in(0,\infty)$, $d\in\N$, 
    such that for every $c\in(0,\infty)$, $d\in\N$ it holds
    \begin{enumerate}[(A)]
    \item\label{NN3NN2A}
	$\displaystyle\L(\psi^d_{\eta,c})\leq C\ln(d)\pa{\abs{\ln(\eta)}+\ln(d)}+\max_{i\in\oneto{d}}\L(\Phi^i_{\eta,c})$,  
    \item \label{NN3NN2B}
	$\displaystyle\M(\psi^d_{\eta,c})\leq C d\pa{\abs{\ln(\eta)}+\ln(d)}
	+4\sum_{i=1}^d \M(\Phi^i_{\eta,c})+8d\max_{i\in\oneto{d}}\L(\Phi^i_{\eta,c})$,      
      and
    \item \label{NN3NN2C}
      for every $\eta\in(0,\infty)$ that 
      \begin{align}
	\sup_{x\in[a,b]^d}\abs{\brac{\tprod_{i=1}^d h_{c,K_i}(x_i)}-\brac{R_{\relu}(\psi^d_{\eta,c})}\!(x)}\leq 3d\eta.
      \end{align}
    \end{enumerate}
    Let $\lambda\in\Nn_1^{1,1}$ be the neural network given by $\lambda=\pa{(-1,1)}$, let $\theta\in\Nn^{1,1}_1$ be the neural network given by $\theta=(0,0)$, and let $(\Psi^d_{\eps,c})_{\eps,c\,\in(0,\infty),d\in\N}\subseteq\Nall$ be the neural networks given by
    \begin{align}\label{NN3Psidef}
      \Psi^d_{\eps,c}=\begin{cases}\lambda\odot\psi^d_{\nicefrac{\eps}{(3d)},c} & \colon \eps\leq 2\\ \theta & \colon \eps>2\end{cases}.
    \end{align}
    Observe that this and \eqref{NN3NN2B} imply  for every
    $\eps\in(0,2]$, $c\,\in(0,\infty)$, $d\in\N$, $x\in[a,b]^d$ it holds
    \begin{align}\begin{split}\label{NN3errEst1}
      \abs{F^d_c(x)-\brac{R_{\relu}(\Psi^d_{\eps,c})}\!(x)}&=\abs{\pa{1-\brac{\tprod_{i=1}^d h_{c,K_i}(x_i)}}-\pa{1-\brac{R_{\relu}(\psi^d_{\nicefrac{\eps}{(3d)},c})}\!(x)}}\\
      &\leq 3d\tfrac{\eps}{3d}=\eps.
    \end{split}\end{align}
    \normalsize
    Moreover, \eqref{NN3Psidef} and \eqref{NN3hEst} ensure for every $\eps\in(2,\infty)$, $c\,\in(0,\infty)$, $d\in\N$, $x\in[a,b]^d$ 
    it holds 
    \begin{align}\begin{split}
      \abs{F^d_c(x)-\brac{R_{\relu}(\Psi^d_{\eps,c})}\!(x)}&=\abs{\pa{1-\brac{\tprod_{i=1}^d h_{c,K_i}(x_i)}}}\\
    \end{split}\end{align}
    This and \eqref{NN3errEst1} establish the neural networks $(\Psi^d_{\eps,c})_{\eps,c\,\in(0,\infty),d\in\N}$ satisfy \eqref{NN3iii}.
    Next observe that for every $c\,\in(0,\infty)$ it holds
    \begin{align}\begin{split}\label{NN3T4}
      \max\{0,\ln(\Kmax+c)\}&\leq\max\{0,\ln(\max\{1,\Kmax\}+\max\{1,\Kmax\}c)\}\\
      &=\ln(\max\{1,\Kmax\}(1+c))=\ln(\max\{1,\Kmax\})+\ln(1+c)\\
      &\leq\ln(c+1)+|\ln(\Kmax)|.
    \end{split}\end{align}
    Hence, we obtain that for every $\eps,c\,\in(0,\infty)$, $d\in\N$ it holds
    \begin{align}\begin{split}\label{NN3T1}
      &\quad\max\{1,|\ln(\tfrac{\eps}{3d})|\}+\max\{0,\ln(\Kmax+c)\}\\
      &\leq |\ln(\eps)|+\ln(d)+\ln(3)+\ln(c+1)+|\ln(\Kmax)|\\
      &\leq (\ln(3)+|\ln(\Kmax)|)\brac{\max\{1,\ln(d)\}(|\ln(\eps)|+\ln(d)+1)+\ln(c+1)}.
    \end{split}\end{align}
    In addition, for every $\eps,c\,\in(0,\infty)$, $d\in\N$ it holds
    \begin{align}
      C\ln(d)\pa{\abs{\ln(\tfrac{\eps}{3d})}+\ln(d)}\leq 4C\brac{\max\{1,\ln(d)\}(|\ln(\eps)|+\ln(d)+1)+\ln(c+1)}.
    \end{align}
    \normalsize
    Combining this with \Cref{NNconc}, \eqref{NN3NN1a}, \eqref{NN3NN2A}, and \eqref{NN3T1} yields
    \begin{align}\begin{split}\label{NN3T2}
      &\quad\sup_{\substack{\eps\in(0,2],c\,\in(0,\infty),\\d\in\N}}\brac{\frac{\L(\Psi^d_{\eps,c})}{\max\{1,\ln(d)\}(\abs{\ln(\eps)}+\ln(d)+1)+\ln(c+1)}}\\      
      &\leq\sup_{\substack{\eps\in(0,2],c\,\in(0,\infty),\\d\in\N}}\brac{\frac{1+C\ln(d)\pa{\abs{\ln(\frac{\eps}{3d})}+\ln(d)}+\max_{i\in\oneto{d}}\L(\Phi^i_{\nicefrac{\eps}{(3d)},c})}{\max\{1,\ln(d)\}(\abs{\ln(\eps)}+\ln(d)+1)+\ln(c+1)}}\\
      &\leq 2+4C+(\ln(3)+|\ln(\Kmax)|)b_L<\infty.
    \end{split}\end{align}
    \normalsize
    Moreover,  \eqref{NN3Psidef} shows
    \small
    \begin{align}\begin{split}
      &\quad\sup_{\substack{\eps\in(2,\infty),c\,\in(0,\infty),\\d\in\N}}\brac{\frac{\L(\Psi^d_{\eps,c})}{\max\{1,\ln(d)\}(\abs{\ln(\eps)}+\ln(d)+1)+\ln(c+1)}}\\ 
      &=\sup_{\substack{\eps\in(2,\infty),c\,\in(0,\infty),\\d\in\N}}\brac{\frac{1}{\max\{1,\ln(d)\}(\abs{\ln(\eps)}+\ln(d)+1)+\ln(c+1)}}<\infty. 
    \end{split}\end{align}
    \normalsize
    This and \eqref{NN3T2} establish that $(\Psi^d_{\eps,c})_{\eps,c\,\in(0,\infty),d\in\N}$ satisfy \eqref{NN3i}.
    Next observe \Cref{Y1tech} implies that
    \begin{itemize}
    \item
      for every $\eps\in(0,2]$ it holds 
      \begin{align}\label{NN3Eeps}
	|\ln(\eps)|\leq\brac{\sup_{\delta\in[\exp(-2n^2),2]}\ln(\delta)}\eps^{-\frac{1}{n}}=2n^2\eps^{-\frac{1}{n}},
      \end{align}
    \item  
      for every $d\in\N$ it holds 
      \begin{align}\label{NN3Ed}
	\ln(d)\leq\brac{\max_{k\in\oneto{\exp(2n^2)}}\ln(k)}d^{\frac{1}{n}}=2n^2d^{\frac{1}{n}},
      \end{align}
    \item
      and for every $c\in(0,\infty)$ it holds
      \begin{align}\label{NN3Ec}
	\ln(c+1)\leq\brac{\sup_{t\in(0,\exp(2n^2-1)]}\ln(t+1)}(c+1)^{\frac{1}{n}}=2n^2(c+1)^{\frac{1}{n}}.
      \end{align}
    \end{itemize}
    For every $m\in\N$, $x_i\in[1,\infty)$, $i\in\oneto{m}$, it holds 
    \begin{align}\label{NN3StP}
      \sum_{i=1}^m x_i\leq \tprod_{i=1}^m(x_i+1)\leq 2^m\tprod_{i=1}^m x_i.
    \end{align}
    Combining this with \eqref{NN3Eeps}, \eqref{NN3Ed}, and \eqref{NN3Ec} 
    shows for every $\eps\in(0,2]$, $d\in\N$, $c\in(0,\infty)$ it holds
    \begin{align}\begin{split}\label{NN3Tt1}
      2C d(|\ln(\tfrac{\eps}{3d})|+\ln(d))&\leq 2C d (|\ln(\eps)|+2\ln(d)+\ln(3)+\ln(c+1))\\
      &\leq 4n^2C d (2\eps^{-\frac{1}{n}}+2d^{\frac{1}{n}}+\ln(3)+(c+1)^{\frac{1}{n}})\\
      &\leq 1024n^2 C (c+1)^{\frac{1}{n}} d^{1+\frac{1}{n}} \eps^{-\frac{1}{n}}.
    \end{split}\end{align}
    Furthermore, note \eqref{NN3T4}, \eqref{NN3Eeps}, \eqref{NN3Ed}, \eqref{NN3Ec}, and \eqref{NN3StP} 
    ensure for every $\eps\in(0,2]$, $d\in\N$, $c\in(0,\infty)$ it holds 
    \begin{align}\begin{split}\label{NN3Tt2}
      &\quad 16d(\max\{1,|\ln(\tfrac{\eps}{3d})|\}+\max\{0,\ln(\Kmax+c)\})\\
      &\leq 16d(|\ln(\eps)|+\ln(d)+\ln(3)+\ln(c+1)+|\ln(\Kmax)|)\\
      &\leq 32n^2d(2\eps^{-\frac{1}{n}}+d^{\frac{1}{n}}+(c+1)^{\frac{1}{n}}+\ln(3)+|\ln(\Kmax)|)\\
      &\leq 2048n^2(\ln(3)+|\ln(\Kmax)|)(c+1)^{\frac{1}{n}} d^{1+\frac{1}{n}} \eps^{-\frac{1}{n}}.
    \end{split}\end{align}
    In addition, observe that for every $\eps\in(0,2]$, $d\in\N$, $c\in(0,\infty)$ it holds 
    \begin{align}\label{NN3Tt3}
      4d(\Kmax+c+1)^{\frac{1}{n}}(\tfrac{\eps}{3d})^{-\frac{1}{n^2}}\leq 96\max\{1,\Kmax\}(c+1)^{\frac{1}{n}}d^{1+\frac{1}{n}}\eps^{-\frac{1}{n}}.
    \end{align}
    Combining this with \Cref{NNconc}, \eqref{NN3NN1a}, \eqref{NN3NN1b}, \eqref{NN3NN2B}, \eqref{NN3Tt1}, and \eqref{NN3Tt2} yield
    \begin{align}\begin{split}\label{NN3T5}
    &\quad\sup_{\substack{\eps\in(0,2],c\,\in(0,\infty),\\d\in\N}}\brac{\frac{\M(\Psi^d_{\eps,c})}{(c+1)^{\frac{1}{n}}d^{1+\frac{1}{n}}\eps^{-\frac{1}{n}}}}\\
      &\leq\sup_{\substack{\eps\in(0,2],c\,\in(0,\infty),\\d\in\N}}\brac{\frac{\displaystyle 4+2Cd(|\ln(\tfrac{\eps}{3d})|+\ln(d))+8\sum_{i=1}^d\M(\Phi^i_{\nicefrac{\eps}{(3d)},c})+16d\max_{i\in\oneto{d}}\L(\Phi^i_{\nicefrac{\eps}{(3d)},c})}{(c+1)^{\frac{1}{n}}d^{1+\frac{1}{n}}\eps^{-\frac{1}{n}}}}\\
      &\leq 8+1024n^2C+96\max\{1,\Kmax\}b_M+2048n^2(\ln(3)+|\ln(\Kmax)|)b_L<\infty.
    \end{split}\end{align}
    \normalsize
    Furthermore, note that \eqref{NN3Psidef} ensures 
    \begin{align}\begin{split}
      \sup_{\substack{\eps\in(2,\infty),c\,\in(0,\infty),\\d\in\N}}\brac{\frac{\M(\Psi^d_{\eps,c})}{(c+1)^{\frac{1}{n}}d^{1+\frac{1}{n}}\eps^{-\frac{1}{n}}}}=\sup_{\substack{\eps\in(2,\infty),c\,\in(0,\infty),\\d\in\N}}\brac{\frac{\M(\theta)}{(c+1)^{\frac{1}{n}}d^{1+\frac{1}{n}}\eps^{-\frac{1}{n}}}}=0.
    \end{split}\end{align}
    \normalsize
    This and \eqref{NN3T5} establish that the neural networks $(\Psi^d_{\eps,c})_{\eps,c\,\in(0,\infty),d\in\N}$ satisfy \eqref{NN3ii}.
    Thus the proof of \Cref{NN3} is completed.
  \end{proof}

Finally, we add the quadrature estimates from Section $4$ to achieve approximation with networks whose size only depends polynomially on the dimension of the problem.
  \begin{theorem}\label{NNMain}
    Assume \Cref{NNSetting}, let $\rho\colon\R\to\R$ be the ReLU activation function given by $\rho(t)=\max\{0,t\}$, let $n\in\N$, $a\in(0,\infty)$, $b\in(a,\infty)$, $(K_i)_{i\in\N}\subseteq[0,\Kmax)$,
    and let $F_d\colon(0,\infty)\times[a,b]^d\to\R$, $d\in\N$, 
    be the functions which satisfy for every $d\in\N$, $c\in(0, \infty)$, $x\in[a,b]^d$ 
    \begin{align}
      F_d(c,x)=1-\prod_{i=1}^d \brac{\normfac\displaystyle\int^{\ln(\frac{K_i+c}{x_i})}_{-\infty} e^{-\frac{1}{2}r^2}\d r}.
    \end{align}
    Then there exists neural networks $(\Gamma_{d,\eps})_{\eps\in(0,1],d\in\N}\in\Nall$ which satisfy
    \begin{enumerate}[(i)]
    \item\label{NNMi}
	$\displaystyle\sup_{\eps\in(0,1],d\in\N}\brac{\frac{\L(\Gamma_{d,\eps})}{\max\{1,\ln(d)\}\pa{|\ln(\eps)|+\ln(d)+1}}}<\infty$,
    \item \label{NNMii}
	$\displaystyle\sup_{\eps\in(0,1],d\in\N}\brac{\frac{\M(\Gamma_{d,\eps})}{d^{2+\frac{1}{n}}\eps^{-\frac{1}{n}}}}<\infty$,
      and
    \item \label{NNMiii}
      for every $\eps\in(0,1]$, $d\in\N$ that
      \begin{align}
	\sup_{x\in[a,b]^d}\abs{\int_0^{\infty}F_d(c,x)\d c-\brac{R_{\relu}(\Gamma_{d,\eps})}\!(x)}\leq\eps.
      \end{align}
    \end{enumerate}
  \end{theorem}
  
  \begin{proof}[Proof of \Cref{NNMain}]
    Throughout this proof assume \Cref{NNSettingPar}, let $S_{b,n}\in\R$ be given by
    \begin{align}
      S_{b,n}=2e^{2(4n+1)}(b+1)^{1+\frac{1}{4n}}
    \end{align}
    and let $N_{d,\eps}\in\R$, $d\in\N$, $\eps\in(0,1]$, be given by
    \begin{align}
      N_{d,\eps}=S_{b,n}d^{\frac{1}{4n}}\brac{\tfrac{\eps}{4}}^{-\frac{1}{4n}}.
    \end{align}
    Note \Cref{QM} (with $4n\leftrightarrow n$, $F_x^d(c)\leftrightarrow F_d(x,c)$, 
    $N_{d,\frac{\eps}{2}}\leftrightarrow N_{d,\eps}$, $Q_{d,\frac{\eps}{2}}\leftrightarrow Q_{d,\eps}$ in the notation of \Cref{QM}) 
    ensures that there exist $Q_{d,\eps}\in\R$, $c^d_{\eps,j}\in(0,N_{d,\eps})$, $w^d_{\eps,j}\in[0,\infty)$, $j\in\oneto{Q_{d,\eps}}$, $d\in\N$, $\eps\in(0,1]$ with
    \begin{align}\label{NNMQM1}
      \sup_{\eps\in(0,1], d\in\N}\brac{\frac{Q_{d,\eps}}{d^{1+\frac{1}{2n}}\eps^{-\frac{1}{2n}}}}<\infty
    \end{align}
    and for every $d\in\N$, $\eps\in(0,1]$ it holds 
    \begin{align}\label{NNMQM2}
      \sup_{x\in[a,b]^d}\abs{\int_0^\infty F_d(c,x)\d c-\sum_{j=0}^{Q_{d,\eps}} w^d_{\eps,j}F_d(c^d_{\eps,j},x)}\leq\tfrac{\eps}{2}
    \end{align}
    and
    \begin{align}\label{NNMsum}
      \sum_{j=1}^{Q_{d,\eps}}w^d_{\eps,j}=N_{d,\eps}.
    \end{align}
    Furthermore, \Cref{NN3} (with $4n\leftrightarrow n$, $F_{c^d_{\eps,j}}^d(x)\leftrightarrow F_d(x,c^d_{\eps,j})$) 
    ensures there exist neural networks $(\Psi^d_{\eps,j})_{\eps\in(0,\infty),d\in\N,j\in\oneto{Q_{d,\eps}}}\subseteq\Nall$ 
    which satisfy
    \begin{enumerate}[(a)]
    \item\label{NNMNN3a}
	$\displaystyle\sup_{\eps\in(0,\infty),d\in\N}\brac{\frac{\max_{j\in\oneto{Q_{d,\eps}}}\L(\Psi^d_{\eps,j})}{\max\{1,\ln(d)\}\pa{|\ln(\frac{\eps}{2N_{d,\eps}})|+\ln(d)+1}+\ln(N_{d,\eps}+1)}}<\infty$,
    \item \label{NNMNN3b}
	$\displaystyle\sup_{\eps\in(0,\infty),d\in\N}\brac{\frac{\max_{j\in\oneto{Q_{d,\eps}}}\M(\Psi^d_{\eps,j})}{(N_{d,\eps}+1)^{\frac{1}{4n}}d^{1+\frac{1}{4n}}\brac{\frac{\eps}{2N_{d,\eps}}}^{-\frac{1}{4n}}}}<\infty$,
      and
    \item\label{NNMNN3c}
      for every $\eps\in(0,\infty)$, $d\in\N$ that
      \begin{align}
	\sup_{x\in[a,b]^d}\abs{F_d(c^d_{\eps,j},x)-\brac{R_{\relu}(\Psi^d_{\eps,j})}\!(x)}\leq\tfrac{\eps}{2N_{d,\eps}}.
      \end{align}
    \end{enumerate}
    Let $\Idd\in\R^{d\times d}$, $d\in\N$, be the matrices given by $\Idd=\mathrm{diag}(1,1,\dots,1)$, 
    let  $\nabla_{d,q}\in\Nn_1^{d,dq}$, $d,q\in\N$, be the neural networks given by 
    \begin{align}
      \nabla_{d,q}=\pa{(\begin{pmatrix}\mathrm{Id}_d \\ \vdots\\ \mathrm{Id}_d\end{pmatrix},0)},
    \end{align}
    let $\Sigma_{d,\eps}\in\Nn_1^{d,1}$, $d\in\N$, $\eps\in(0,1]$, be the neural networks given by 
    \begin{align}
      \Sigma_{d,\eps}=\pa{(\begin{pmatrix}w^d_{\eps,1} & w^d_{\eps,2} & \dots & w^d_{\eps,Q_{d,\eps}}\end{pmatrix},0)},
    \end{align}
    and let $(\Gamma_{d,\eps})_{\eps\in(0,1],d\in\N}\in\Nall$ be the neural networks given by
    \begin{align}
      \Gamma_{d,\eps}=\Sigma_{d,\eps}\odot \Pc(\Psi^d_{\eps,1}, \Psi^d_{\eps,2}, \dots , \Psi^d_{\eps,Q_{d,\eps}}) \odot \nabla_{d,Q_{d,\eps}}.
    \end{align}
    Combining \Cref{NNconc}, \Cref{NNpar}, \eqref{NNMQM2}, \eqref{NNMsum}, and \eqref{NNMNN3c} 
    implies for every $\eps\in(0,\infty)$ and $d\in\N$, $x\in[a,b]^d$ it holds
    \begin{align}\begin{split}
      &\quad\abs{\int_0^{\infty}F_d(c,x)\d c-\brac{R_{\relu}(\Gamma_{d,\eps})}\!(x)}\\
      &\leq\abs{\int_0^\infty F_d(c,x)\d c-\sum_{j=0}^{Q_{d,\eps}} w^d_{\eps,j}F_d(c^d_{\eps,j},x)}+\abs{\sum_{j=0}^{Q_{d,\eps}} w^d_{\eps,j}F_d(c^d_{\eps,j},x)-\brac{R_{\relu}(\Gamma_{d,\eps})}\!(x)}\\
      &\leq\tfrac{\eps}{2}+\abs{\sum_{j=0}^{Q_{d,\eps}} w^d_{\eps,j}F_d(c^d_{\eps,j},x)-\sum_{j=0}^{Q_{d,\eps}}w^d_{\eps,j}\brac{R_{\relu}(\Psi^d_{\eps,j})}\!(x)}\\
      &\leq\tfrac{\eps}{2}+\sum_{j=0}^{Q_{d,\eps}} w^d_{\eps,j}\abs{F_d(c^d_{\eps,j},x)-\brac{R_{\relu}(\Psi^d_{\eps,j})}\!(x)}\leq\tfrac{\eps}{2}+N_{d,\eps}\tfrac{\eps}{2N_{d,\eps}}=\eps.
    \end{split}\end{align}
    \normalsize
    This establishes that the neural networks $(\Gamma_{d,\eps})_{\eps\in(0,1],d\in\N}$ satisfy \eqref{NNMiii}.
    Next, observe for every $\eps\in(0,\infty)$, $d\in\N$ 
    \begin{align}\begin{split}
      &\quad\max\{1,\ln(d)\}\pa{|\ln(\frac{\eps}{2N_{d,\eps}})|+\ln(d)+1}+\ln(N_{d,\eps}+1)\\
      &\leq\max\{1,\ln(d)\}\pa{|\ln(\eps)|+\ln(d)+3\ln(N_{d,\eps})+\ln(2)+1}\\
      &\leq\max\{1,\ln(d)\}\pa{|\ln(\eps)|+\ln(d)+3\pa{\ln(S_{b,n})+\frac{1}{4n}\ln(d)+\frac{1}{4n}|\ln(\eps)|+\frac{1}{4n}\ln(4)}+2}\\
      &\leq\max\{1,\ln(d)\}\pa{4|\ln(\eps)|+4\ln(d)+3\ln(S_{b,n})+8}\\
      &\leq(3\ln(S_{b,n})+8)\max\{1,\ln(d)\}\pa{|\ln(\eps)|+\ln(d)+1}.
    \end{split}\end{align}
    \normalsize
    Combining this with \Cref{NNconc}, \Cref{NNpar}, and \eqref{NNMNN3a} implies
    \begin{align}\begin{split}
      &\quad\sup_{\eps\in(0,1],d\in\N}\brac{\frac{\L(\Gamma_{d,\eps})}{\max\{1,\ln(d)\}\pa{|\ln(\eps)|+\ln(d)+1}}}\\
      &\leq\sup_{\eps\in(0,1],d\in\N}\brac{\frac{\L(\Sigma_{d,\eps})+\max_{j\in\oneto{Q_{d,\eps}}}\L(\Psi^d_{\eps,j})+\L(\nabla_{d,Q_{d,\eps}})}{\max\{1,\ln(d)\}\pa{|\ln(\eps)|+\ln(d)+1}}}\\
      &\leq 2+\sup_{\eps\in(0,1],d\in\N}\brac{\frac{\max_{j\in\oneto{Q_{d,\eps}}}\L(\Psi^d_{\eps,j})}{\max\{1,\ln(d)\}\pa{|\ln(\eps)|+\ln(d)+1}}}\\
      &\leq 2+(3\ln(S_{b,n})+8)\!\!\!\!\!\!\sup_{\eps\in(0,\infty),d\in\N}\brac{\frac{\max_{j\in\oneto{Q_{d,\eps}}}\L(\Psi^d_{\eps,j})}{\max\{1,\ln(d)\}\pa{|\ln(\frac{\eps}{2N_{d,\eps}})|+\ln(d)+1}+\ln(N_{d,\eps}+1)}}\\
      &<\infty.
    \end{split}\end{align}
    \normalsize
    This establishes $(\Gamma_{d,\eps})_{\eps\in(0,1],d\in\N}$ satisfy \eqref{NNMi}.
    In addition, for every $\eps\in(0,\infty)$, $d\in\N$ it holds 
    \begin{align}\begin{split}
      (N_{d,\eps}+1)^{\frac{1}{4n}}d^{1+\frac{1}{4n}}\brac{\frac{\eps}{2N_{d,\eps}}}^{-\frac{1}{4n}}&\leq 4N_{d,\eps}^{\frac{1}{2n}}d^{1+\frac{1}{4n}}\eps^{-\frac{1}{4n}}\\
      &\leq 4\brac{S_{b,n}d^{\frac{1}{4n}}\brac{\tfrac{\eps}{4}}^{-\frac{1}{4n}}}^{\frac{1}{2n}}d^{1+\frac{1}{4n}}\eps^{-\frac{1}{4n}}\\
      &\leq 16S_{b,n}d^{1+\frac{1}{4n}+\frac{1}{4n^2}}\eps^{-(\frac{1}{4n}+\frac{1}{8n^2})}\\
      &\leq16S_{b,n}d^{1+\frac{1}{2n}}\eps^{-\frac{1}{2n}}.
    \end{split}\end{align}
    Combining this with \Cref{NNconc}, \Cref{NNpar}, \eqref{NNMQM1}, \eqref{NNMNN3b}, 
    and the fact that for every $\psi\in\Nall$ which satisfies ${\min_{l\in\oneto{\L(\psi)}}\M_l(\psi)>0}$ 
    it holds $\L(\psi)\leq\M(\psi)$ ensures
    \begin{align}\begin{split}
      &\quad\sup_{\eps\in(0,1],d\in\N}\brac{\frac{\M(\Gamma_{d,\eps})}{d^{(2+\frac{1}{n})}\eps^{-\frac{1}{n}}}}\\
      &\leq\sup_{\eps\in(0,1],d\in\N}\brac{\frac{\displaystyle2\M(\Sigma_{d,\eps})+4\pa{2\sum_{j=1}^{Q_{d,\eps}}\M(\Psi^d_{\eps,j})+4Q_{d,\eps}\max_{j\in\oneto{Q_{d,\eps}}}\L(\Psi^d_{\eps,j})}+4\M(\nabla_{d,Q_{d,\eps}})}{d^{(2+\frac{1}{n})}\eps^{-\frac{1}{n}}}}\\
      &\leq\sup_{\eps\in(0,1],d\in\N}\brac{\frac{24Q_{d,\eps}\max_{j\in\oneto{Q_{d,\eps}}}\M(\Psi^d_{\eps,j})}{d^{(2+\frac{1}{n})}\eps^{-\frac{1}{n}}}}+\sup_{\eps\in(0,1],d\in\N}\brac{\frac{2Q_{d,\eps}+4dQ_{d,\eps}}{d^{(2+\frac{1}{n})}\eps^{-\frac{1}{n}}}}\\
      &\leq 24\pa{\sup_{\eps\in(0,1],d\in\N}\brac{\frac{Q_{d,\eps}}{d^{(1+\frac{1}{2n})}\eps^{-\frac{1}{2n}}}}}\pa{\sup_{\eps\in(0,1],d\in\N}\brac{\frac{\max_{j\in\oneto{Q_{d,\eps}}}\M(\Psi^d_{\eps,j})}{d^{(1+\frac{1}{2n})}\eps^{-\frac{1}{2n}}}}}\\
      &\quad +4\sup_{\eps\in(0,1],d\in\N}\brac{\frac{Q_{d,\eps}}{d^{(1+\frac{1}{n})}\eps^{-\frac{1}{n}}}}\\
      &\leq 24\pa{\sup_{\eps\in(0,1],d\in\N}\brac{\frac{Q_{d,\eps}}{d^{(1+\frac{1}{2n})}\eps^{-\frac{1}{2n}}}}}\pa{1+16S_{b,n}\sup_{\eps\in(0,1],d\in\N}\brac{\frac{\max_{j\in\oneto{Q_{d,\eps}}}\M(\Psi^d_{\eps,j})}{(N_{d,\eps}+1)^{\frac{1}{4n}}d^{1+\frac{1}{4n}}\brac{\frac{\eps}{2N_{d,\eps}}}^{-\frac{1}{4n}}}}}\\
      &<\infty.
    \end{split}\end{align}
    \normalsize
    This establishes the neural networks $(\Gamma_{d,\eps})_{\eps\in(0,1],d\in\N}$ satisfy \eqref{NNMii}. 
    The proof of \Cref{NNMain} is thus completed.
  \end{proof}

\section{Discussion}  
  
While \Cref{NNMain} only establishes formally that the solution of one specific high-dimensional PDE may be approximated by neural networks without curse of dimensionality, the constructive approach also serves to illustrate that neural networks are capable of accomplishing the same for any PDE solution which exhibits a similar low-rank structure. Note here, that the tensor product construction in \Cref{NN2} only introduces a logarithmic dependency on the approximation accuracy. That we end up with a spectral rate in this specific case is due to \Cref{NN2} and \Cref{QM}, i.e. the insufficient regularity of the univariate functions inside the tensor product, as well as the number of terms required by the Gaussian quadrature used to approximate the outer integral. In particular, this means that the approach in \Cref{sec:basicexpr} might also be used to produce approximation results with connectivity growing only logarithmically in the inverse of the approximation error, given that one has a suitably well behaved low-rank structure. 

The present result is a promising step towards higher order, numerical solution of high-dimensional PDEs, which are notoriously troublesome to handle with any of the classical approaches based on discretization of the domain, or with randomized (a.k.a.\@ Monte-Carlo based) arguments. Of course answering the question of approximability can only ensure that there exist networks with a reasonable size-to-accuracy trade-off, whereas for any practical purpose it is also necessary to establish whether and how one can find these networks.

An analysis of the generalization error for linear Kolmogorov equations can be found in \cite{BGJ18_2679}, which concludes that, under reasonable assumptions, the number of required Monte Carlo samples is free of the curse of dimensionality.
Moreover, there are a number of empirical results \cite{KolmogorovNumerics,BeckEJentzen2017,EHanJentzen2017a,EHanJentzen2017b,SirignanoSpiliopoulos2017}, which suggest that the solutions of various high dimensional PDEs may be learned efficiently using standard stochastic gradient descent based methods. However, a satisfying formal analysis of this training procesdoes not seem to be available at the present. 

Lastly we would like to point out that, even though we had a semi-explicit formula available, the ReLU networks we used for approximation were in no way adapted to use this knowledge and have been shown to exhibit excellent approximation properties for, e.g., piecewise smooth functions~\cite{PetVoigt}, affine and Gabor systems~\cite{elbrchter2019deep}, and even fractal structures~\cite{dym2019expression}. So, while a spline dictionary based approach specifically designed for the approximation of this one PDE solution may have similar rates, it would most certainly lack the remarkable universality of neural networks.

\bibliography{blackscholes_new}
\bibliographystyle{acm}

\appendix
\section{Additional Proofs}\label{appendix}

\subsection{Technical Lemma}
\label{sec:TechLem}

  \begin{lemma}\label{Y1tech}
    It holds for every $r\in(0,\infty)$, $t\in(0,\exp(-2r^2)]$ that
    \begin{align}
      \abs{\ln(t)}\leq t^{-\nicefrac{1}{r}}
    \end{align}
    and for every $r\in(0,\infty)$, $t\in[\exp(2r^2),\infty)$ that
    \begin{align}
      \ln(t)\leq t^{\nicefrac{1}{r}}.
    \end{align}
  \end{lemma}
  \begin{proof}[Proof of \Cref{Y1tech}]
    First, observe that for every $r\in(0,\infty)$, $y\in[2r^2,\infty)$ it holds that
    \begin{align}
      \exp\!\pa{\frac{y}{r}}=\sum_{k=0}^{\infty}\brac{\frac{y^k}{k!r^k}}\geq\frac{y^2}{2!r^2}=y\brac{\frac{y}{2r^2}}\geq y.\label{Y1tech1}
    \end{align}
    This implies that for every $r\in(0,\infty)$, $x\in[\exp(2r^2),\infty)$ it holds that
    \begin{align}
      x^{\nicefrac{1}{r}}=\exp\!\pa{\ln\!\pa{x^{\nicefrac{1}{r}}}}=\exp\!\pa{\tfrac{\ln(x)}{r}}\geq\ln(x).
    \end{align}
    Hence, we obtain that for every $r\in(0,\infty)$, $t\in(0,\exp(-2r^2)]\subseteq(0,1]$ it holds that
    \begin{align}
      t^{-\nicefrac{1}{r}}=\brac{\tfrac{1}{t}}^{\nicefrac{1}{r}}\geq\ln(\tfrac{1}{t})=\abs{\ln(t)}.
    \end{align}
    This completes the proof of \Cref{Y1tech}.
  \end{proof}

\subsection{Proof of \Cref{NNsquare}}
\label{sec:PrfNNsquare} 
\begin{proof}[Proof of \Cref{NNsquare}] 
	The proof follows \cite{Yarotsky}. We provide it in order to provide
        values of constants in the bounds on depth and width, and to reveal the dependence on
        the scaling parameter $B$.
	Throughout this proof let $\theta\in\Nn^{1,1}_1$ be the neural network given by $\theta=(0,0)$, 
        let $g_s\colon [0,1]\to[0,1]$, $s\in\N$, be the functions which satisfy for every $s\in\N$, $t\in[0,1]$ that
	\begin{align}\label{NNsgDef}
	g_s(t)=\begin{cases}
	2t & \colon s=1,t<\tfrac{1}{2}\\
	2-2t & \colon s=1,t\geq\tfrac{1}{2}\\
	g_1(g_{s-1}(t)) & \colon s\geq 1
	\end{cases},
	\end{align}
	and let $f_m\colon [0,1]\to[0,1]$, $m\in\N$, 
        be the functions which satisfy for every 
        $m\in\N$, $k\in\zeroto{2^m}$, $x\in\brac{\tfrac{k}{2^m},\tfrac{k+1}{2^m}}$ 
        that
	\begin{align}\label{NNsfmDef}
	f_m(x)=\brac{\frac{2k+1}{2^m}}x-\frac{k^2+k}{2^{2m}}.
	\end{align}
	We claim for every $s\in\N$, $k\in\zeroto{2^{s-1}-1}$ it holds 
	\begin{align}\label{NNsClaim}
	g_s(x)=\begin{cases}
	2^s(x-\tfrac{2k}{2^s}) & \colon x\in\brac{\tfrac{2k}{2^s},{\tfrac{2k+1}{2^s}}}\\
	2^s(\tfrac{2k+2}{2^s}-x) & \colon x\in\brac{\tfrac{2k+1}{2^s},{\tfrac{2k+2}{2^s}}}
	\end{cases}.
	\end{align}
	We now prove \eqref{NNsClaim} by induction on $s\in\N$. 
        Equation \eqref{NNsgDef} establishes \eqref{NNsClaim} in the base case $s=1$. 
	For the induction step $\N\ni s\to s+1\in\{2,3,\dots\}$ observe that \eqref{NNsgDef} 
        implies for every $s\in\N$, $l\in\zeroto{2^{s-1}-1}$ that
	\begin{enumerate}[(a)]
		\item
		it holds for every $x\in\brac{\tfrac{2l}{2^s},\tfrac{2l+(\nicefrac{1}{2})}{2^s}}$ 
		\begin{align}\begin{split}\label{NNsEq1}
		g_{s+1}(x)&=g(g_s(x))=g(2^s(x-\tfrac{2l}{2^s}))=2\brac{2^s(x-\tfrac{2l}{2^s})}\\
		&=2^{s+1}(x-\tfrac{2l}{2^s})=2^{s+1}(x-\tfrac{2(2l)}{2^{s+1}}).
		\end{split}\end{align}
		\item 
		it holds for every $x\in\brac{\tfrac{2l+(\nicefrac{1}{2})}{2^s},\tfrac{2l+1}{2^s}}$ 
		\begin{align}\begin{split}\label{NNsEq2}
		g_{s+1}(x)&=g(g_s(x))=g(2^s(x-\tfrac{2l}{2^s}))=2-2\brac{2^s(x-\tfrac{2l}{2^s})}\\
		&=2-2^{s+1}x+4l=2^{s+1}(\tfrac{4l+2}{2^{s+1}}-x)\\
		&=2^{s+1}(\tfrac{2(2l+1)}{2^{s+1}}-x).
		\end{split}\end{align}  
		\item
		it holds for every $x\in\brac{\tfrac{2l+1}{2^s},\tfrac{2l+(\nicefrac{3}{2})}{2^s}}$ 
		\begin{align}\begin{split}\label{NNsEq3}
		g_{s+1}(x)&=g(g_s(x))=g(2^s(\tfrac{2l+2}{2^s}-x))=2-2\brac{2^s(\tfrac{2l+2}{2^s}-x)}\\
		&=2-2(2l+2)+2^{s+1}x=2^{s+1}x-2(2l+1)\\
		&=2^{s+1}(x-\tfrac{2(2l+1)}{2^{s+1}}).
		\end{split}\end{align}
		\item
		it holds for every $x\in\brac{\tfrac{2l+(\nicefrac{3}{2})}{2^s},\tfrac{2l+2}{2^s}}$ 
		\begin{align}\begin{split}\label{NNsEq4}
		g_{s+1}(x)&=g(g_s(x))=g(2^s(\tfrac{2l+2}{2^s}-x))=2\brac{2^s(\tfrac{2l+2}{2^s}-x)}\\
		&=2^{s+1}(\tfrac{2l+2}{2^s}-x)=2^{s+1}(\tfrac{2(2l+2)}{2^{s+1}}-x).
		\end{split}\end{align} 
	\end{enumerate} 
	Next observe that for every $s\in\N$, $k\in\zeroto{2^s-1}$ there exists $l\in\zeroto{2^{s-1}-1}$ 
        such that
	\begin{align}\label{NNsS1}
	\brac{\tfrac{2k}{2^{s+1}},{\tfrac{2k+1}{2^{s+1}}}}=\brac{\tfrac{2l}{2^s},\tfrac{2l+(\nicefrac{1}{2})}{2^s}}\quad\mathrm{or}\quad\brac{\tfrac{2k}{2^{s+1}},{\tfrac{2k+1}{2^{s+1}}}}=\brac{\tfrac{2l+1}{2^s},\tfrac{2l+(\nicefrac{3}{2})}{2^s}}.
	\end{align}
	Furthermore, for every $s\in\N$, $k\in\zeroto{2^s-1}$ there exists $l\in\zeroto{2^{s-1}-1}$ such that
	\begin{align}\label{NNsS2}
	\brac{\tfrac{2k+1}{2^{s+1}},{\tfrac{2k+2}{2^{s+1}}}}=\brac{\tfrac{2l+(\nicefrac{1}{2})}{2^s},\tfrac{2l+1}{2^s}}\quad\mathrm{or}\quad\brac{\tfrac{2k+1}{2^{s+1}},{\tfrac{2k+2}{2^{s+1}}}}=\brac{\tfrac{2l+(\nicefrac{3}{2})}{2^s},\tfrac{2l+2}{2^s}}.
	\end{align}
	Combining this with \eqref{NNsEq1}, \eqref{NNsEq2}, \eqref{NNsEq3}, \eqref{NNsEq4}, and \eqref{NNsS1} completes the induction step $\N\ni s\to s+1\in\{2,3,\dots\}$ 
	and thus establishes the claim \eqref{NNsClaim}.

	Next, for every $m\in\N$, $k\in\zeroto{2^{m-1}}$ it holds
	\begin{align}\begin{split}\label{NNsfDIFF1}
	f_{m-1}(\tfrac{2k}{2^m})-f_m(\tfrac{2k}{2^m})&=f_{m-1}(\tfrac{k}{2^{m-1}})-f_m(\tfrac{2k}{2^m})
	=\brac{\tfrac{k}{2^{m-1}}}^2-\brac{\tfrac{2k}{2^m}}^2 =0.
	\end{split}\end{align}
	In addition, note that \eqref{NNsfmDef} implies that 
        for every $m\in\N$, $k\in\zeroto{2^m-1}$ it holds
	\begin{align}\begin{split}\label{NNsT3}
	f_{m-1}(\tfrac{2k+1}{2^{m}})
        &=f_{m-1}\pa{\tfrac{k+\frac{1}{2}}{2^{m-1}}}=\brac{\frac{2k+1}{2^{m-1}}}\frac{k+\frac{1}{2}}{2^{m-1}}-\frac{k^2+k}{2^{2(m-1)}}\\
	&=\frac{(2k+1)(k+\frac{1}{2})-(k^2+k)}{2^{2m-2}}=\frac{k^2+k+\frac{1}{2}}{2^{2m-2}} =\frac{4k^2+4k+2}{2^{2m}}
	\end{split}\end{align}    
	and
	\begin{align}\begin{split}\label{NNsfDIFF2}
	f_m(\tfrac{2k+1}{2^{m}})&=\brac{\frac{2(2k+1)+1}{2^m}}\frac{2k+1}{2^{m}}-\frac{(2k+1)^2+(2k+1)}{2^{2m}}
	=\frac{4k^2+4k+1}{2^{2m}}.
	\end{split}\end{align}
	For every $m\in\N$, $k\in\zeroto{2^m-1}$ it holds
	\begin{align}
	f_{m-1}(\tfrac{2k+1}{2^m})-f_m(\tfrac{2k+1}{2^m})=\frac{4k^2+4k+2}{2^{2m}}-\frac{4k^2+4k+1}{2^{2m}}=\frac{1}{2^{2m}}.
	\end{align}
	Combining this with \eqref{NNsClaim}, \eqref{NNsfmDef}, and \eqref{NNsfDIFF1} 
        demonstrates that for every $m\in\N$, $x\in[0,1]$ it holds
	\begin{align}\label{NNsClaim2}
	f_{m-1}(x)-f_m(x)=2^{-2m}g_m(x).
	\end{align}
	The fact that for every $x\in[0,1]$ it holds that $f_0(x)=x$ therefore implies that 
        for every $m\in\N_0$, $x\in[0,1]$ it holds
	\begin{align}\label{NNsfmSum}
	f_m(x)=x-\sum_{s=1}^m 2^{-2s} g_s(x).
	\end{align}
	We observe $f_m$ is the affine, linear interpolant of the twice 
        continuously differentiable function $[0,1]\ni x\mapsto x^2\in[0,1]$ 
        at the points $\tfrac{k}{2^m}$, $k\in\zeroto{2^m}$. This establishes that for every $m\in\N$
	\begin{align}\begin{split}\label{NNsfmEst}
	\sup_{x\in[0,1]}\abs{x^2-f_m(x)}&=\max_{k\in\zeroto{2^m}}\pa{\sup_{x\in\brac{\frac{k}{2^m},\frac{k+1}{2^m}}}\abs{x^2-f_m(x)}}\\
	&\leq\max_{k\in\zeroto{2^m}}\pa{\frac{\brac{\frac{k+1}{2^m}-\frac{k}{2^m}}^2}{8}\max_{x\in\brac{\frac{k}{2^m},\frac{k+1}{2^m}}}\abs{\ddtn{2}\brac{x^2}}}\\
	&\leq\max_{k\in\zeroto{2^m}}\pa{\tfrac{1}{8}\brac{\tfrac{1}{2^m}}^2\max_{x\in\brac{\frac{k}{2^m},\frac{k+1}{2^m}}}\abs{2}}\\
	&=2^{-2m-2}.
	\end{split}\end{align}
	Let $(A_k,b_k)\in\R^{4\times 4}\times\R^4$, $k\in\N$, 
        be the matrix-vector tuples which satisfy for every $k\in\N$
	\begin{align}
	A_k=\begin{pmatrix}
	2 & -4 & 2 & 0\\
	2 & -4 & 2 & 0\\
	2 & -4 & 2 & 0\\
	-2^{-2k+3} & 2^{-2k+4} & -2^{-2k+3} & 1
	\end{pmatrix}
	\quad\mathrm{and}\quad
	b_k=\begin{pmatrix}
	0 \\ -\frac{1}{2} \\ -1 \\0
	\end{pmatrix},
	\end{align}    
	let $\phi_m\in\Nall$, $m\in\N$, be the neural networks which satisfy $\phi_1=(1,0)$ and, for every $m\in\N$,
	\begin{align}
	\phi_m=\pa{\pa{\begin{pmatrix}1 \\ 1 \\ 1 \\ 1 \end{pmatrix},\begin{pmatrix}0 \\ -\frac{1}{2} \\ -1 \\0\end{pmatrix}},(A_2,b_2),\dots,(A_{m-1},b_{m-1}),\pa{\begin{pmatrix}-2^{-2m+3} \\ 2^{-2m+4} \\ -2^{-2m+3} \\ 1\end{pmatrix}^T,0}}.
	\end{align}
	Let further $r^k\colon\R\to\R$, $k\in\N$ denote the function which satisfies for every $x\in\R$ 
	\begin{align}
	(r^1_1(x),r^1_2(x),r^1_3(x),r^1_4(x))=r^1(x)=\rho^*(x,x-\tfrac{1}{2},x-1,x)
	\end{align}  
	and for every $x\in\R$, $k\in\N$ 
	\begin{align}
	(r^k_1(x),r^k_2(x),r^k_3(x),r^k_4(x))=r^k(x)=\rho^*(A_k r_{k-1}(x)+b_k).
	\end{align}
	We claim that for every $k\in\oneto{m-1}$, $x\in[0,1]$ it holds
	\begin{enumerate}[(a)]
		\item\label{NNsNIa} 
		\begin{align}
			2r^k_1(x)-4r^k_2(x)+2r^k_3(x)=g_k(x)
		\end{align}
		and
		\item\label{NNsNIb} 
		\begin{align}
			r^k_4(x)=x-\sum_{j=1}^{k-1} 2^{-2j}g_j(x).	
		\end{align}
	\end{enumerate}
	We prove \eqref{NNsNIa} and \eqref{NNsNIb} by induction over $k\in\oneto{m-1}$. 
        For the base case $k=1$ we note that for every $x\in[0,1]$ it holds 
	\begin{align}\label{NNsT01}
		g_1(x)=2\rho(x)-4\rho(x-\tfrac{1}{2})+2\rho(x-1). 
	\end{align} 
	Hence, we obtain that for every $x\in[0,1]$ it holds
	\begin{align}\label{NNsT02}
		2r^1_1(x)-4r^1_2(x)+2r^1_3(x)=2\rho(x)-4\rho(x-\tfrac{1}{2})+2\rho(x-1)=g_1(x).
	\end{align}
	Furthermore, note that for every $x\in[0,1]$ it holds that $r^1_4(x)=x$. 
        This and \eqref{NNsT02} establish the base case $k=1$.
	For the induction step $\oneto{m-2}\ni k-1\to k\in\{2,3,\dots,m-1\}$ observe that \eqref{NNsT01} 
        ensures for every $x\in[0,1]$, $k\in\{2,3,\dots,m-1\}$, with
	$g_{k-1}(x)=2r^{k-1}_1(x)-4r^{k-1}_2(x)+2r^{k-1}_3(x)$, it holds 
	\begin{align}\begin{split}
	2r^k_1(x)-4r^k_2(x)+2r^k_3(x)=\quad&2\rho(2r^{k-1}_1(x)-4r^{k-1}_2(x)+2r^{k-1}_3(x))\\
	-&4\rho(2r^{k-1}_1(x)-4r^{k-1}_2(x)+2r^{k-1}_3(x)-\tfrac{1}{2})\\
	+&2\rho(2r^{k-1}_1(x)-4r^{k-1}_2(x)+2r^{k-1}_3(x)-1)\\
	=\quad&g_1(2r^{k-1}_1(x)-4r^{k-1}_2(x)+2r^{k-1}_3(x))\\
	=\quad&g_1(g_{k-1}(k))=g_k(x).
	\end{split}\end{align}
	Induction thus establishes \eqref{NNsNIa}.
	Moreover note that \eqref{NNsfmDef} and \eqref{NNsfmSum} for every $k\in\N$, $x\in[0,1]$ it holds
	\begin{align}
		x-\sum_{j=1}^{k-1} 2^{-2j}g_j(x)=f_{k-1}(x)\geq 0.
	\end{align}
	Combining this with \eqref{NNsT01} implies that for every $x\in[0,1]$, $k\in\{2,3,\dots,m-1\}$ with 
	$g_{k-1}(x)=2r^{k-1}_1(x)-4r^{k-1}_2(x)+2r^{k-1}_3(x)$ and 
        $r^{k-1}_4(x)=x-\sum_{j=1}^{k-2} 2^{-2j}g_j(x)$ it holds
	\begin{align}\begin{split}
		r^k_4(x)&=\rho(-2^{-2k+3}r^{k-1}_1(x)+2^{-2k+4}r^{k-1}_2(x)-2^{-2k+3}r^{k-1}_3(x)+r^{k-1}_4(x))\\
		&=\rho(x-\sum_{j=1}^{k-2} 2^{-2j}g_j(x)-g_{k-1}(x))=\rho(x-\sum_{j=1}^{k-1} 2^{-2j}g_j(x))\\
		&=x-\sum_{j=1}^{k-1} 2^{-2j}g_j(x).
	\end{split}\end{align}
	Induction thus establishes \eqref{NNsNIb}.	
	Next observe that \eqref{NNsNIa} and \eqref{NNsNIb} that for every $m\in\N$, $x\in[0,1]$ it holds 
	\begin{align}\begin{split}
	[R_{\rho}(\phi_m)](x)&=-2^{-2m+3}r^{m-1}_1(x)+2^{-2m+4}r^{m-1}_2(x)-2^{-2m+3}r^{m-1}_3(x)+r^{m-1}_4(x)\\
	&=-2^{-2(m-1)}\pa{2r^{m-1}_1(x)-4r^{m-1}_2(x)+2r^{m-1}_3(x)}+x-\sum_{j=1}^{m-2} 2^{-2j}g_j(x)\\
	&=x-\brac{\sum_{j=1}^{m-2} 2^{-2j}g_j(x)}-2^{-2(m-1)}g_{m-1}(x)=x-\sum_{j=1}^{m-1} 2^{-2j}g_j(x).
	\end{split}\end{align} 
	Combining this with \eqref{NNsfmSum} establishes that for every $m\in\N$, $x\in[0,1]$ it holds
	\begin{align}
	[R_{\rho}(\phi_m)](x)=f_{m-1}(x).
	\end{align}
	This and \eqref{NNsfmEst} imply that for every $m\in\N$ it holds
	\begin{align}\label{NNsphimEst}
	\sup_{x\in[0,1]}\abs{x^2-[R_{\rho}(\phi_m)](x)}\leq 2^{-2m}.
	\end{align}
	Furthermore, observe that by construction it holds for every $m\in\N$
	\begin{align}\label{NNsLMphi}
	\L(\phi_m)=m\quad\mathrm{and}\quad \M(\phi_m)=\max\{1,10+15(m-2)\}\leq 15m.
	\end{align}
	Let $(\sigma_{\eps})_{\eps\in(0,\infty)}\subseteq\Nall$ be the neural networks which satisfy for 
        $\eps\in(0,1)$
	\begin{align}
	\sigma_{\eps}=\phi_{\left\lceil\frac{1}{2}\abs{\log_2(\eps)}\right\rceil}
	\end{align}
	and for every $\eps\in[1,\infty)$ that $\sigma\leps=\theta$.
	Observe that for every $\eps\in[1,\infty)$ it holds 
	\begin{align}
	\sup_{x\in[0,1]}\abs{x^2-[R_{\rho}(\sigma\leps)](x)}=\sup_{x\in[0,1]}\abs{x^2-[R_{\rho}(\theta)](x)}\leq 1\leq\eps.
	\end{align}
	In addition note for every $\eps\in(0,1)$ it holds
	\begin{align}\begin{split}
	\sup_{x\in[0,1]}\abs{x^2-[R_{\rho}(\sigma\leps)](x)}&=\sup_{x\in[0,1]}\abs{x^2-[R_{\rho}(\phi_{\left\lceil\frac{1}{2}\abs{\log_2(\eps)}\right\rceil})](x)} \\
	&\leq 2^{-2\left\lceil\frac{1}{2}\abs{\log_2(\eps)}\right\rceil}\leq 2^{-2(\frac{1}{2}\abs{\log_2(\eps)})}
	=2^{\log_2(\eps)}=\eps.
	\end{split}\end{align}
	Moreover, observe that \eqref{NNsLMphi} implies for every $\eps\in(0,1)$ it holds
	\begin{align}
	\L(\sigma\leps)=\L(\phi_{\left\lceil\frac{1}{2}\abs{\log_2(\eps)}\right\rceil})=\left\lceil\tfrac{1}{2}\abs{\log_2(\eps)}\right\rceil
	\end{align}
	and
	\begin{align}
	\M(\sigma\leps)=\M(\phi_{\left\lceil\frac{1}{2}\abs{\log_2(\eps)}\right\rceil})\leq 15\left\lceil\tfrac{1}{2}\abs{\log_2(\eps)}\right\rceil.
	\end{align}
	Furthermore, for every $\eps\in[1,\infty)$ it holds 
        $\L(\sigma\leps)=\L(\theta)=1$ and $\M(\sigma\leps)=\M(\theta)=0$. 
	This completes the proof of \Cref{NNsquare}.    
\end{proof}

\subsection{Proof of \Cref{NNmult}}
\begin{proof}[Proof of \Cref{NNmult}]
Throughout this proof assume \Cref{NNSettingPar}, 
    let $\theta\in\Nn^{1,1}_1$ be the neural network given by $\theta=(0,0)$,
	let $\alpha\in\Nn_2^{2,6,3}$ be the neural network given by
	\begin{align}\begin{split}\label{NNmalphaDef}
	\alpha_1&=\pa{(\begin{pmatrix}
	1 & 1\\ -1 & -1\\ 1 & 0 \\ -1 & 0 \\ 0 & 1 \\ 0 & -1 
	\end{pmatrix},\begin{pmatrix}0 \\ 0 \\ 0 \\ 0 \\ 0 \\ 0\end{pmatrix}),
                      (\tfrac{1}{2B}\begin{pmatrix}
                      1 & 1 & 0 & 0 & 0 & 0 \\
                      0 & 0 & 1 & 1 & 0 & 0 \\
                      0 & 0 & 0 & 0 & 1 & 1
                      \end{pmatrix},\begin{pmatrix} 0 \\ 0 \\ 0\end{pmatrix})},
	\end{split}\end{align}
	and let $\Sigma\in\Nn^{3,1}_1$ be the neural network given by 
        $\Sigma=\pa{(\begin{pmatrix}2B^2 & -2B^2 & -2B^2\end{pmatrix},0)}$. 
	Observe that \Cref{NNsquare} 
        ensures the existence of 
        neural networks $(\sigma_{\eps})_{\eps\in(0,\infty)}\subseteq\Nall$ 
        which satisfy \Cref{NNsquare}, \eqref{NNsi} -- \eqref{NNt=0}.
	Let $(\mu_{\eps})_{\eps\in(0,\infty)}\subseteq\Nall$ be the neural networks which satisfy for every $\eps\in(0,\infty)$ 
	\begin{align}\label{NNmmuDef}
	\mu\leps=\begin{cases}\Sigma\odot \Pc\pa{\sigma_{\nicefrac{\eps}{6B^2}},\sigma_{\nicefrac{\eps}{6B^2}},\sigma_{\nicefrac{\eps}{6B^2}}}\odot\alpha & \colon \eps < B^2\\ \theta & \colon \eps\geq B^2\end{cases}.
	\end{align}   
	Note first that for every $\eps\in[B^2,\infty)$ it holds
	\begin{align}\begin{split}
	\sup_{x,y\in[-B,B]}\abs{xy-\brac{R_{\relu}(\mu_{\eps})}\!(x,y)}&=\sup_{x,y\in[-B,B]}\abs{xy-\brac{R_{\relu}(\theta)}\!(x,y)}=\sup_{x,y\in[-B,B]}\abs{xy-0}=B^2\leq\eps.
	\end{split}\end{align}
	Next observe that for every $(x,y)\in\R^2$ it holds 
	\begin{align}\begin{split}\label{NNmalphaR}
	[R_{\relu}(\alpha)](x,y)=\tfrac{1}{2B}\begin{pmatrix}
	\rho(x+y)+\rho(-(x+y)) \\ \rho(x) + \rho(-x) \\ \rho(y) + \rho(-y) 
	\end{pmatrix}= \tfrac{1}{2B}\begin{pmatrix}
	|x+y| \\ |x| \\ |y|
	\end{pmatrix}.
	\end{split}\end{align}	
	Furthermore, for every $(x,y,z)\in\R^3$ holds $[R_{\relu}(\Sigma)](x,y,z) = 2B^2x - 2B^2y - 2B^2z$.
	Combining this with \Cref{NNconc}, \Cref{NNpar}, \eqref{NNmmuDef}, and \eqref{NNmalphaR} 
        establishes that for every $\eps\in(0,B^2)$, $(x,y)\in[-B,B]^2$ it holds
	\begin{align}\begin{split}\label{NNmRsum}
	[R_{\relu}(\mu\leps)](x,y)&=2B^2\pa{[R_{\relu}(\sigma_{\nicefrac{\eps}{6B^2}})]\pa{\tfrac{\abs{x+y}}{2B}}
	-[R_{\relu}(\sigma_{\nicefrac{\eps}{6B^2}})]\pa{\tfrac{\abs{x}}{2B}}
	-[R_{\relu}(\sigma_{\nicefrac{\eps}{6B^2}})]\pa{\tfrac{\abs{y}}{2B}}}.
	\end{split}\end{align}
	With \Cref{NNsquare}, \Cref{NNt=0}, \eqref{NNmRsum} establishes \eqref{NNmv}.
	In addition note that \Cref{NNsquare} demonstrates for every $\eps\in(0,\infty)$ it holds
	\begin{align}\begin{split}
	&\quad\sup_{z\in[-2B,2B]}\abs{\tfrac{1}{2}z^2-2B^2\brac{[R_{\relu}(\sigma_{\nicefrac{\eps}{6B^2}})]\pa{\tfrac{\abs{z}}{2B}}}}\\
	&=\sup_{z\in[-2B,2B]}\abs{2B^2\brac{\tfrac{\abs{z}}{2B}}^2-2B^2\brac{[R_{\relu}(\sigma_{\nicefrac{\eps}{6B^2}})]\pa{\tfrac{\abs{z}}{2B}}}}\\
	&=2B^2\brac{\sup_{t\in[0,1]}\abs{t^2-\brac{[R_{\relu}(\sigma_{\nicefrac{\eps}{6B^2}})]\pa{t}}}}\leq 2B^2\brac{\frac{\eps}{6B^2}}=\frac{\eps}{3}.
	\end{split}\end{align}
	This and \eqref{NNmRsum} establish that for every $\eps\in(0,B^2)$ it holds
	\begin{align}\begin{split}
	&\quad\sup_{x,y\in[-B,B]}\abs{xy-[R_{\relu}(\mu\leps)](x,y)}\\
	&=\sup_{x,y\in[-B,B]}\abs{\frac{1}{2}\brac{(x+y)^2-x^2-y^2}-[R_{\relu}(\mu\leps)](x,y)}\\
	&\leq \tfrac{\eps}{3}+\tfrac{\eps}{3}+\tfrac{\eps}{3}=\eps.
	\end{split}\end{align}
	Next observe that $\L(\alpha)=2$ and $\L(\Sigma)=1$. 
	Combining this with \Cref{NNconc}, \Cref{NNpar}, and \Cref{NNsquare}\eqref{NNsi} ensures
        for every $\eps\in(0,B^2)$
	\begin{align}\begin{split}
	\L(\mu\leps)&=\L(\Sigma)+\L(\sigma_{\nicefrac{\eps}{6B^2}})+\L(\alpha)\\
	&\leq\tfrac{1}{2}\abs{\log_2(\tfrac{\eps}{6B^2})}+4=\tfrac{1}{2}\log_2(\tfrac{6B^2}{\eps})+4\\
	&\leq\tfrac{1}{2}(\log_2(\tfrac{1}{\eps})+2\log_2(B)+3)+4\\
	&=\tfrac{1}{2}\log_2(\tfrac{1}{\eps})+\log_2(B)+6.
	\end{split}\end{align}
	Combining $\M(\alpha)=14$ and $\M(\Sigma)=3$ with \Cref{NNconc}, \Cref{NNpar}, \Cref{NNsquare}\eqref{NNsii}, and \eqref{NNmalphaDef} demonstrate that for every $\eps\in(0,B^2)$ it holds	
	\begin{align}\begin{split}
	\M(\mu\leps)
	&\leq 2\pa{\M(\Sigma)+3\M(\sigma_{\nicefrac{\eps}{6B^2}})+\M(\alpha)}\\
	&\leq 34 +  90(\tfrac{1}{2}|\log_2(\tfrac{6B^2}{\eps})|+1)\\
	&\leq 45\log_2(\tfrac{1}{\eps})+90\log_2(B) + 259.
	\end{split}\end{align}
	Moreover, for every $\eps\in(B^2,\infty)$ it holds $\L(\mu\leps)=1$ and $\M(\mu\leps)=0$. 
    Next, observe \Cref{NNconc} and \Cref{NNpar} demonstrate that for every $\eps\in(0,\infty)$ it holds that $\M_1(\mu_\eps)=\M_1(\alpha)=8$ and $\M_{\L(\mu_{\eps})}(\mu_{\eps})=\M(\Sigma)=3$.
	This completes the proof of \Cref{NNmult}.

\end{proof}
\subsection{Proof of \Cref{NNsmoothFunctions}}
\begin{proof}[Proof of \Cref{NNsmoothFunctions}]
	Throughout this proof assume \Cref{NNSettingPar}, 
        let $h_{N,j}\colon\R\to\R$, $N\in\N$, $j\in\{0,1,\dots,N\}$, 
        be the functions which satisfy for every $N\in\N$, $j\in\{0,1,\dots,N\}$, $x\in\R$ 
	\begin{align}\label{NNsFhnJdef}
	h_{N,j}(x)=\begin{cases}Nx+1-j & \colon \tfrac{j-1}{N}\leq x \leq \tfrac{j}{N}\\ -Nx+1+j & \colon \tfrac{j}{N} \leq x\leq \tfrac{j+1}{N}\\ 0 & \colon \mathrm{else} \end{cases},
	\end{align}
	let $T_{f,N,j}\colon\R\to\R$, $f\in B^n_1$, $N\in\N$, $j\in\{0,1,\dots,N\}$, 
        be the functions which satisfy for every 
        $f\in B^n_1$, $N\in\N$, $j\in\{0,1,\dots,N\}$, $x\in[0,1]$ 
	\begin{align}\label{NNsFTdef}
	T_{f,N,j}(x)=\sum_{k=0}^{n-1}\frac{f^{(k)}(\tfrac{j}{N})}{k!}(x-\tfrac{j}{N})^k.
	\end{align}
	For every $f\in B^n_1$, let $f_N:\R\to\R$, $N\in\N$ denote 
        functions which satisfy for every $N\in\N$, $x\in[0,1]$ 
	\begin{align}\label{NNsFfNdef}
	f_N(x)=\sum_{j=0}^N h_{N,j}(x)T_{f,N,j}(x).
	\end{align}
	Observe that Taylor's theorem (with Lagrange remainder term) 
        ensures that for every
        $f\in B^n_1$, $N\in\N$, $j\in\{0,1,\dots,N\}$, 
        $x\in[\max\{0,\tfrac{j-1}{N}\},\min\{1,\tfrac{j+1}{N}\}]$ 
	\begin{align}\begin{split}\label{NNsFTaylor}
	\abs{f(x)-T_{f,N,j}(x)}&\leq\tfrac{1}{n!}\abs{x-\tfrac{j}{N}}^n 
        \sup_{\xi\in[\max\{0,\tfrac{j-1}{N}\},\min\{1,\tfrac{j+1}{N}\}]}\abs{f^{(n)}(\xi)}
          \\
	&\leq \tfrac{1}{n!} N^{-n} \max_{k\in\{0,1,\dots,n\}}\brac{\sup_{t\in[0,1]}\abs{f^{(k)}(t)}}
        \leq \tfrac{1}{n!} N^{-n}.
	\end{split}\end{align}
	Moreover, for every $N\in\N$, $x\in[0,1]$, $j\notin\{\lceil Nx\rceil-1,\lceil Nx\rceil\}$ 
        it holds that $h_{N,j}(x)=0$. 
        We obtain for every $N\in\N$ and $x\in[0,1]$ 
	\small
	\begin{align}\label{NNsFT1}
	\sum_{j=0}^N h_{N,j}(x)T_{f,N,j}(x)=h_{N,\lceil Nx\rceil-1}(x)T_{f,N,\lceil Nx\rceil-1}(x)
                      +h_{N,\lceil Nx\rceil}(x)T_{f,N,\lceil Nx\rceil}(x).
	\end{align}
	\normalsize
	Furthermore, \eqref{NNsFhnJdef} implies for every
        $N\in\N$, $j\in\{1,\dots,N-1\}$, $x\in[\tfrac{j-1}{N},\tfrac{j}{N}]$ holds 
	\begin{align}	
	h_{N,j-1}(x)+h_{N,j}(x)=-Nx+1+(j-1) + Nx+1-j=1.
	\end{align}
	Combining this with \eqref{NNsFfNdef}, \eqref{NNsFTaylor}, and \eqref{NNsFT1} 
        establishes that for every $f\in B^n_1$, $N\in\N$, $x\in[0,1]$
	\begin{align}\begin{split}\label{NNsFffnEst}
	&\quad\abs{f(x)-f_N(x)}\\
	&=\abs{f(x)-\sum_{j=0}^N h_{N,j}(x)T_{f,N,j}(x)}\\
	&=\abs{f(x)-\pa{h_{N,\lceil Nx\rceil-1}(x)T_{f,N,\lceil Nx\rceil-1}(x)+h_{N,\lceil Nx\rceil}(x)T_{f,N,\lceil Nx\rceil}(x)}}\\
	&\leq\abs{h_{N,\lceil Nx\rceil-1}(x)f(x)-h_{N,\lceil Nx\rceil-1}(x)T_{f,N,\lceil Nx\rceil-1}(x)}\\
	&\quad+\abs{h_{N,\lceil Nx\rceil}(x)f(x)-h_{N,\lceil Nx\rceil}(x)T_{f,N,\lceil Nx\rceil}(x)}\\
	&=h_{N,\lceil Nx\rceil-1}(x)\abs{f(x)-T_{f,N,\lceil Nx\rceil-1}(x)}
         +h_{N,\lceil Nx\rceil}(x)\abs{f(x)-T_{f,N,\lceil Nx\rceil}(x)}\\
	&\leq h_{N,\lceil Nx\rceil-1}(x)\brac{\tfrac{1}{n!} N^{-n}}
             +h_{N,\lceil Nx\rceil}(x)\brac{\tfrac{1}{n!} N^{-n}}=\tfrac{1}{n!} N^{-n}.
	\end{split}\end{align}
	We now realize this local Taylor approximation using neural networks.
	To this end, note that \Cref{NNbigMult} ensures that there exist
        $C\in\R$ and neural networks 
         $(\Pi_{\eta}^k)_{\eta\in(0,\infty)}$, $k\in\N\cap[2,\infty)$ which satisfy
	\begin{enumerate}[(A)]
		\item\label{NNsFA}
		  $\displaystyle\L(\Pi_{\eta}^k)\leq C\ln(k)\pa{\abs{\ln(\eta)}+k\ln(3)+\ln(k)}$,
		\item\label{NNsFB}
	        $\displaystyle\M(\Pi_{\eta}^k)\leq C k\pa{\abs{\ln(\eta)}+k\ln(3)+\ln(k)}$, 
		\item\label{NNsFC}
		$\displaystyle\sup_{x\in[-3,3]^k}\abs{\brac{\prod_{i=1}^k x_i}-\brac{R_{\relu}(\Pi_{\eta}^k)}\!(x)}
                \leq \eta$ and
		\item\label{NNsFD}
		$R_\relu\left[\Pi_{\eta}^k\right](x_1,x_2,\dots,x_k)=0$, 
                    if there exists $i\in\{1,2,\dots,k\}$ with $x_i=0$.
	\end{enumerate}
To complete the proof, we introduce the following neural networks:
\begin{itemize}
\item 
$\nabla_{N,j,k}\in\Nn^{k,1}_1$, $N\in\N$, $j\in\zeroto{N}$, $k\in\{2,3,\dots,n-1\}$ 
given by
	\begin{align}\label{NNsFnabla}
	\nabla_{N,j,k}=\pa{(\begin{pmatrix}1 \\ \vdots \\ 1\end{pmatrix},
        \begin{pmatrix}-\tfrac{j}{N} \\ \vdots \\ -\tfrac{j}{N} \end{pmatrix})},
	\end{align}
\item 
$\xi_{\eps,N,j}^k\in\Nall$, $\eps\in(0,\infty)$, $N\in\N$, $j\in\zeroto{N}$, $k\in\{1,2,\dots,n-1\}$, 
given by 
	\begin{align}\label{NNsFxi}
	\xi_{\eps,N,j}^k=\begin{cases}(1,0) & \colon k=1\\ \Pi^k_{\nicefrac{\eps}{8e}}\odot\nabla_{N,j,k} & \colon k>1\end{cases},
	\end{align}
\item 
$\Sigma_{f,N,j}\in\Nn^{1,n-1}_1$, $f\in B^n_1$, $N\in\N$, $j\in\zeroto{N}$ 
given by 
	\begin{align}\label{NNsFSigma}
	\Sigma_{f,N,j}=\pa{(\begin{pmatrix}\frac{f^{(n-1)}(\tfrac{j}{N})}{(n-1)!} 
        & \frac{f^{(n-2)}(\tfrac{j}{N})}{(n-2)!} & \dots & \frac{f^{(1)}(\tfrac{j}{N})}{(1)!}\end{pmatrix},f(\tfrac{j}{N}))},
	\end{align}
\item
$\tau_{f,\eps,N,j}\in\Nall$, $f\in B^n_1$, $\eps\in(0,\infty)$, $N\in\N$, $j\in\zeroto{N}$ 
given by
	\begin{align}\label{NNsFtau}
	\tau_{f,\eps,N,j}=\Sigma_{f,N,j}\odot \Pc(\xi_{\eps,N,j}^{n-1}, \xi_{\eps,N,j}^{n-2}, \dots, \xi_{\eps,N,j}^1)\odot\nabla_{1,0,n-1},
	\end{align}
\item
$\chi_{N,j}\in\Nn^{1,3,1}_2$, $N\in\N$, $j\in\zeroto{N}$ given by
\begin{align}\label{NNsFchi}
\chi_{N,j}=\pa{(\begin{pmatrix}1 \\ 1 \\ 1 \end{pmatrix},
\begin{pmatrix}-\nicefrac{(j-1)}{N} \\ -\nicefrac{j}{N} \\ -\nicefrac{(j+1)}{N}\end{pmatrix}),
(\begin{pmatrix}1 & -2 & 1\end{pmatrix},0)}
\end{align}
\item
$\lambda_N\in\Nn^{1,N+1}_1$,  $N\in\N$ given by
	\begin{align}\label{NNsFlambda}
	\lambda_N=\pa{(\begin{pmatrix}1 & \dots & 1\end{pmatrix},0)},
	\end{align}
\item
$\psi_{f,\eps,N,j}\in\Nall$, $f\in B^n_1$, $\eps\in(0,\infty)$, $N\in\N$, $j\in\zeroto{N}$ 
given by
\begin{align}
\psi_{f,\eps,N,j}=\Pi^2_{\nicefrac{\eps}{8}}\odot \Pc(\chi_{N,j},\tau_{f,\eps,N,j}),
\end{align}
\item
$\phi_{f,\eps,N}\in\Nall$, $f\in B^n_1$, $N\in\N$, $\eps\in(0,\infty)$ 
given by
\begin{align}\label{NNsFphi}
\phi_{f,\eps,N}=\lambda_N\odot \Pc\pa{\psi_{f,\eps,N,1}, \psi_{f,\eps,N,2}, \dots, \psi_{f,\eps,N,N}}\odot\nabla_{1,0,2N+2}.
\end{align}
\end{itemize}
With these networks, we 
note \Cref{NNconc}, \Cref{NNpar}, \eqref{NNsFC}, \eqref{NNsFnabla} and \eqref{NNsFxi} 
        ensure that for every 
        $N\in\N$, $\eps\in(0,\infty)$, $j\in\zeroto{N}$, $k\in\{2,3,\dots,n-1\}$ 
	\small
	\begin{align}\begin{split}\label{NNsFxiEst1}
	&\quad\sup_{x\in[0,1]}\abs{(x-\tfrac{j}{N})^k-\brac{R_{\relu}(\xi_{\eps,N,j}^k)}\!(x)}\\
	&\leq\sup_{x\in[0,1]}\abs{(x-\tfrac{j}{N})^k-\brac{R_{\relu}(\Pi^k_{\nicefrac{\eps}{8e}})}(\brac{R_{\relu}(\nabla_{N,j,k})}(x))}\\
	&\leq\sup_{x\in[0,1]}\abs{\brac{\prod_{i=1}^k (x-\tfrac{j}{N})^k}-\brac{R_{\relu}(\Pi^k_{\nicefrac{\eps}{8e}})}(x-\tfrac{j}{N},x-\tfrac{j}{N},\dots,x-\tfrac{j}{N})}\\
	&\leq\sup_{x\in[-1,1]^k}\abs{\brac{\prod_{i=1}^k x_i}-\brac{R_{\relu}(\Pi_{\nicefrac{\eps}{8e}}^k)}\!(x)}\leq\tfrac{\eps}{8e}  
	\end{split}\end{align}
	\normalsize
	and
	\begin{align}\label{NNsFxiEst2}
	\sup_{x\in[0,1]}\abs{(x-\tfrac{j}{N})-\brac{R_\relu(\xi_{\eps,N,j}^1)}(x)}=0.
	\end{align}
	Moreover, \Cref{NNconc}, \Cref{NNpar}, \eqref{NNsFnabla}, \eqref{NNsFxi}, \eqref{NNsFSigma}, and \eqref{NNsFtau} 
        demonstrate that for every $f\in B^n_1$, $N\in\N$, $\eps\in(0,\infty)$, $j\in\zeroto{N}$, $x\in[0,1]$ it holds
	\begin{align}
	\brac{R_{\relu}(\tau_{f,\eps,N,j})}(x)=\sum_{k=1}^{n-1}\brac{\frac{f^{(k)}(\tfrac{j}{N})}{k!}\brac{R_{\relu}(\xi_{\eps,N,j}^k)}(x)}+f(\tfrac{j}{N}).
	\end{align}
	Combining this with \eqref{NNsFTdef}, \eqref{NNsFtau}, \eqref{NNsFxiEst1} and \eqref{NNsFxiEst1} 
        establishes that for every 
        $f\in B^n_1$, $N\in\N$, $\eps\in(0,\infty)$, $j\in\zeroto{N}$, $x\in[0,1]$ 
        it holds
	\small
	\begin{align}\begin{split}\label{NNsFtauEst}
	&\quad\abs{T_{f,N,j}(x)-\brac{R_{\relu}(\tau_{f,\eps,N,j})}(x)}\\
	&=\abs{\pa{\sum_{k=0}^{n-1}
         \frac{f^{(k)}(\tfrac{j}{N})}{k!}(x-\tfrac{j}{N})^k}
           -\pa{\sum_{k=1}^{n-1}\brac{\frac{f^{(k)}(\tfrac{j}{N})}{k!}\brac{R_{\relu}(\xi_{\eps,N,j}^k)}(x)}+f(\tfrac{j}{N})}}\\
	&\leq\sum_{k=1}^{n-1}\pa{\frac{f^{(k)}(\tfrac{j}{N})}{k!}\abs{(x-\tfrac{j}{N})^k-\brac{R_{\relu}(\xi_{\eps,N,j}^k)}(x)}}\\
	&\leq \frac{\eps}{8e}\sum_{k=1}^{n-1}\frac{f^{(k)}(\tfrac{j}{N})}{k!}
         \leq\frac{\eps}{8e}\pa{\sum_{k=1}^{\infty}\frac{1}{k!}}\leq\frac{\eps}{8}.
	\end{split}\end{align}
	\normalsize
	Next,  \eqref{NNsFchi} ensures for every $N\in\N$, $j\in\zeroto{N}$, $x\in[0,1]$ 
	\begin{align}\label{NNsFT4}
	[R_{\relu}(\chi_{N,j})](x)=\rho(x-\tfrac{j-1}{N})-2\rho(x-\tfrac{j}{N})+\rho(x-\tfrac{j+1}{N})=h_{N,j}(x).
	\end{align}
	Now \eqref{NNsFtauEst} and Taylor's Theorem 
         imply for every $f\in B^n_1$, $N\in\N$, $\eps\in(0,1)$, $j\in\zeroto{N}$, $x\in[0,1]$ that
	\begin{align}\begin{split}
	  \abs{[R_{\relu}(\tau_{f,\eps,N,j})](x)}&\leq\abs{[R_{\relu}(\tau_{f,\eps,N,j})](x)-T_{f,N,j}(x)}+\abs{T_{f,N,j}(x)-f(x)}+\abs{f(x)}\\
	  &\leq\frac{\eps}{4(N+1)}+\tfrac{1}{n!} x^n \sup_{t\in[0,1]}|f^{(n)}(t)|+\sup_{t\in[0,1]}\abs{f(t)}\leq 3.
	\end{split}\end{align}
	Combining this with \Cref{NNconc}, \Cref{NNpar}, \eqref{NNsFhnJdef}, \eqref{NNsFC}, \eqref{NNsFtauEst}, 
        and \eqref{NNsFT4} establishes for every $f\in B^n_1$, $N\in\N$, $\eps\in(0,1)$, $j\in\zeroto{N}$, $x\in[0,1]$ the bound
	\begin{align}\begin{split}\label{NNsFxyz}
	&\quad\abs{h_{N,j}(x)T_{f,N,j}(x)-[\Rr(\psi_{f,\eps,N,j})](x,x)}\\
	&\leq\abs{h_{N,j}(x)T_{f,N,j}(x)-[\Rr(\chi_{N,j})](x)[\Rr(\tau_{N,j})](x)}\\
	&\quad+\abs{[\Rr(\chi_{N,j})](x)[\Rr(\tau_{N,j})](x)-[\Rr(\Pi^2_{\nicefrac{\eps}{8}}\circ \Pc(\chi_{N,j},\tau_{f,\eps,N,j}))](x,x)}\\
	&\leq\abs{h_{N,j}(x)T_{f,N,j}(x)-[\Rr(\tau_{N,j})](x)}\\
	&\quad+\abs{[\Rr(\chi_{N,j})](x)[\Rr(\tau_{N,j})](x)-[\Rr(\Pi^2_{\nicefrac{\eps}{8}})]([\Rr(\chi_{N,j}](x),[\Rr(\tau_{f,\eps,N,j})](x))}\\
	&\leq\tfrac{\eps}{8}+\tfrac{\eps}{8}=\tfrac{\eps}{4}.
	\end{split}\end{align}
	Furthermore, note that for every $N\in\N$, $j\in\zeroto{N}$, $x\notin[\tfrac{j-1}{N},\tfrac{j+1}{N}]$ it holds that $h_{N,j}(x)=\chi_{N,j}(x)=0$. 
	Thus \eqref{NNsFD} ensures that for every $f\in B^n_1$, $N\in\N$, $\eps\in(0,1)$, $j\in\zeroto{N}$, $x\in[0,1]$ it holds
	\begin{align}
	 \abs{h_{N,j}(x)T_{f,N,j}(x)-[\Rr(\psi_{f,\eps,N,j})](x,x)}=0.
	\end{align}
	This, \Cref{NNconc}, \Cref{NNpar}, \eqref{NNsFfNdef}, \eqref{NNsFphi}, and \eqref{NNsFxyz} imply that
        for every $f\in B^n_1$, $N\in\N$, $\eps\in(0,1)$, $x\in[0,1]$ it holds
	\begin{align}\begin{split}
	\abs{f_N(x)-[\Rr(\phi_{f,\eps,N})](x)}&=\abs{\sum_{j=0}^N h_{N,j}(x)T_{f,N,j}(x)-\sum_{j=0}^N [\Rr(\psi_{f,\eps,N,j})](x,x)}\\
	&\leq2\max_{j\in\zeroto{N}}\abs{h_{N,j}(x)T_{f,N,j}(x)-[\Rr(\psi_{f,\eps,N,j})](x,x)}\\
	&\leq\tfrac{\eps}{2}.
	\end{split}\end{align}
	Combining this with \eqref{NNsFffnEst} establishes that for every $f\in B^n_1$, $N\in\N$, $\eps\in(0,1)$, $x\in[0,1]$ it holds
	\begin{align}\begin{split}\label{NnsFphiepsNest}
	\abs{f(x)-[\Rr(\phi_{f,\eps,N})](x)}&\leq\abs{f(x)-f_N(x)}+\abs{f_N(x)-[\Rr(\phi_{f,\eps,N})]}
	\leq \tfrac{1}{n!} N^{-n}+\tfrac{\eps}{2}.
	\end{split}\end{align}    
	Let $N_{\eps}\in\N$ satisfy for every $\eps\in(0,\infty)$ 
	\begin{align}
	N_{\eps}=\left\lceil\brac{\tfrac{2}{n!\eps}}^{\nicefrac{1}{n}}\right\rceil,
	\end{align}
	let $\theta\in\Nn^{1,1}_1$ be given by $\theta=(0,0)$, 
        and let $(\Phi_{f,\eps})_{f\in B^n_1,\eps\in(0,\infty)}\subseteq\Nall$ be the neural networks 
        given by
	\begin{align}
	\Phi_{f,\eps}=\begin{cases}\phi_{f,\eps,N\leps} & \colon \eps < 1 \\\theta & \colon \eps\geq 1\end{cases}.
	\end{align}
	Oberve that \eqref{NnsFphiepsNest} implies that for every $f\in B^n_1$, $\eps\in(0,1)$, $x\in[0,1]$
	\begin{align}\begin{split}\label{NNsFT2}
	\abs{f(x)-[\Rr(\Phi_{f,\eps})](x)}&=\abs{f(x)-[\Rr(\phi_{f,\eps,N\leps})](x)}\leq \tfrac{1}{n!} N_{\eps}^{-n} +\tfrac{\eps}{2}
	\leq \tfrac{1}{n!} \brac{\tfrac{n!\eps}{2}}+\tfrac{\eps}{2}=\eps.
	\end{split}\end{align}
	Moreover that for every $f\in B^n_1$, $\eps\in[1,\infty)$, $x\in[0,1]$ it holds
	\begin{align}
	\abs{f(x)-[\Rr(\Phi_{f,\eps})](x)}=\abs{f(x)-[\Rr(\theta)](x)}=\abs{f(x)}\leq 1\leq\eps.
	\end{align}
	This and \eqref{NNsFT2} establish that the neural networks $(\Phi_{f,\eps})_{f\in B^n_1,\eps\in(0,\infty)}$ 
        satisfy \eqref{NNsFii}.

	Next, \Cref{NNconc}, \Cref{NNpar}, \eqref{NNsFA}, \eqref{NNsFnabla}, and \eqref{NNsFxi} 
        imply for every $N\in\N$, $\eps\in(0,\infty)$, $j\in\zeroto{N}$, $k\in\{1,2,\dots,n-1\}$ 
	\begin{align}\begin{split}
	\L(\xi^k_{\eps,N,j})&\leq \max\{1,\L(\Pi^k_{\nicefrac{\eps}{8e}})+\L(\nabla_{N,j,k})\}
	\leq C\ln(k)\pa{|\ln(\tfrac{\eps}{8e})|+k\ln(3)+\ln(k)}+1.
	\end{split}\end{align}
	Combining this with \Cref{NNconc}, \Cref{NNpar}, \eqref{NNsFnabla}, \eqref{NNsFSigma}, \eqref{NNsFtau} 
        shows for every $f\in B^n_1$, $N\in\N$, $\eps\in(0,\infty)$, $j\in\zeroto{N}$ the bound
	\begin{align}\begin{split}
	\L(\tau_{f,\eps,N,j})&\leq\L(\Sigma_{f,N,j})+\brac{\max_{k\in\oneto{n-1}}\L(\xi^k_{\eps,N,j})}+\L(\nabla_{1,0,n-1})\\
	&\leq 3+C\ln(n)\pa{ |\ln(\tfrac{\eps}{8e})|+n\ln(3)+\ln(n)}.
	\end{split}\end{align}
	This, \Cref{NNconc}, \Cref{NNpar}, \eqref{NNsFA}, \eqref{NNsFchi}, \eqref{NNsFlambda}, \eqref{NNsFphi}, and \eqref{NNsFnabla} 
        ensure for every $f\in B^n_1$, $N\in\N$, $\eps\in(0,\infty)$ it holds
	\begin{align}\begin{split}\label{NNbMT4}
	\L(\phi_{f,\eps,N})&\leq\L(\lambda_N)+\brac{\max_{j\in\zeroto{N}}\L(\psi_{f,\eps,N,j})}+\L(\nabla_{1,0,2N+2})\\
	&\leq 2+\brac{\max_{j\in\zeroto{N}}\L(\Pi^2_{\nicefrac{\eps}{8}}\odot \Pc(\chi_{N,j},\tau_{f,\eps,N,j}))}\\
	&\leq 2+\brac{C\ln(2)\pa{|\ln(\tfrac{\eps}{8})|+2\ln(3)+\ln(2)}+\max\{3,\L(\tau_{f,\eps,N,j})\}}\\
	&\leq 5+C\ln(2)\pa{|\ln(\tfrac{\eps}{8})|+\ln(18)}+C\ln(n)\pa{|\ln(\tfrac{\eps}{8e})|+n\ln(3)+\ln(n)}\\
	&\leq 5+C\ln(2)\pa{|\ln(\eps)|+|\ln(8)|+\ln(18)}\\
	&\quad+C\ln(n)\pa{ |\ln(\eps)|+|\ln(8e)|+n\ln(3)+\ln(n)}\\
	&=C\ln(2n)\abs{\ln(\eps)}+C(\ln(2)\ln(144)+\ln(n)(\ln(3)n+\ln(n)+|\ln(8e)|))+5.
	\end{split}\end{align}
	\normalsize
	With the constant $C$ from \eqref{NNbMT4}, define the term $T_1$ by
	\begin{align}
	T_1=C(\ln(2)\ln(144)+\ln(n)(\ln(3)n+\ln(n)+|\ln(8e)|))+5.
	\end{align}
	Observe that \eqref{NNbMT4} implies for every 
        $f\in B^n_1$, $\eps\in(0,1)$ 
	\begin{align}\begin{split}\label{NNsFT3}
	\L(\Phi_{f,\eps})=\L(\phi_{f,\eps,N_{\eps}})=C\ln(2n)\abs{\ln(\eps)}+T_1.
	\end{split}\end{align}
	Hence we obtain 
	\begin{align}\begin{split}\label{NNsFLLest1}
	\sup_{f\in B^n_1,\eps\in(0,e^{-r}]}
        \brac{\tfrac{\L(\Phi_{f,\eps})}{\max\{r,\abs{\ln(\eps)}\}}}
        \leq\sup_{f\in B^n_1,\eps\in(0,e^{-r}]}\brac{\tfrac{C\ln(2n)\abs{\ln(\eps)}+T_1}{\abs{\ln(\eps)}}}
        \leq C\ln(2n)+\tfrac{T_1}{r}<\infty.
	\end{split}\end{align}
	In addition, note that \eqref{NNsFT3} ensures that
	\begin{align}\begin{split}\label{NNsFLLest2}
	\sup_{f\in B^n_1,\eps\in(e^{-r},1)}\brac{\tfrac{\L(\Phi_{f,\eps})}{\max\{r,\abs{\ln(\eps)}\}}}
        \leq\sup_{f\in B^n_1,\eps\in(e^{-r},1)}\brac{\tfrac{C\ln(2n)\abs{\ln(\eps)}+T_1}{r}}
        \leq C\ln(2n)+\tfrac{T_1}{r}<\infty.
	\end{split}\end{align}
	Furthermore 
	\begin{align}\label{NNsFLLest3}
	\sup_{f\in B^n_1,\eps\in[1,\infty)}\!\brac{\tfrac{\L(\Phi_{f,\eps})}{\max\{r,\abs{\ln(\eps)}\}}}
           =\sup_{f\in B^n_1,\eps\in[1,\infty)}\!\brac{\tfrac{1}{\max\{r,\abs{\ln(\eps)}\}}}<\infty.
	\end{align}
	This, \eqref{NNsFLLest1}, and \eqref{NNsFLLest2} establish that
        the neural networks $(\Phi_{f,\eps})_{\eps\in(0,\infty)}$ satisfy \eqref{NNsFi1}.
	Next, \Cref{NNconc}, \eqref{NNsFB}, \eqref{NNsFnabla}, and \eqref{NNsFxi} imply 
        for every $N\in\N$, $\eps\in(0,\infty)$, $j\in\zeroto{N}$, $k\in\{1,2,\dots,n-1\}$
	\begin{align}\begin{split}
	\M(\xi^k_{\eps,N,j})&\leq \max\{1,2(\M(\Pi^k_{\nicefrac{\eps}{8e}})+\M(\nabla_{N,j,k}))\}
	\leq 2(C k\pa{\abs{\ln(\tfrac{\eps}{8e})}+k\ln(3)+\ln(k)} +1)
	\end{split}\end{align}
	Combining this with \Cref{NNconc}, \Cref{NNpar}, \eqref{NNsFnabla}, \eqref{NNsFSigma}, and \eqref{NNsFtau} 
        shows for every $f\in B^n_1$, $N\in\N$, $\eps\in(0,\infty)$, $j\in\zeroto{N}$ it holds
	\begin{align}\begin{split}
	\M(\tau_{f,\eps,N,j})&\leq 2\pa{\M(\Sigma_{f,N,j})+2\pa{\M(\Pc(\xi_{\eps,N,j}^{n-1}, \dots, \xi_{\eps,N,j}^1))+\L(\nabla_{1,0,n-1})}}\\
	&\leq 2n+4\pa{2\brac{\sum_{k=1}^{n-1}\M(\xi_{\eps,N,j}^k)}+4(n-1)\max_{k\in\oneto{n-1}}\L(\xi_{\eps,N,j}^k)}+8(n-1)\\
	&\leq 10n+8(n-1)(2C n\pa{\ln(\tfrac{\eps}{(8e)})|+n\ln(3)+\ln(n)}+2)\\
	&\quad+16(n-1)(C\ln(n)\pa{|\ln(\tfrac{\eps}{8e})|+n\ln(3)+\ln(n)}+1)\\
	&\leq 32n^2C\pa{|\ln(\tfrac{\eps}{8e})|+n\ln(3)+\ln(n)}+42n.
	\end{split}\end{align}
	\normalsize
	Let the term $T_2$ be given by
	\begin{align}
	T_2=128\pa{C+32n^2C+C\ln(n)},
	\end{align}
	and let the term $T_3$ be given by 
	\begin{align}
	T_3=1556+128(C\ln(144)+64n^2C(n\ln(3)+\ln(n))+42n.
	\end{align}
	This, \Cref{NNconc}, \Cref{NNpar}, \eqref{NNsFB}, \eqref{NNsFnabla}, \eqref{NNsFchi}, \eqref{NNsFlambda}, \eqref{NNsFphi}, 
        and the fact that for every $\psi\in\Nall$ with $\min_{l\in\oneto{\L(\psi)}}\M_l(\psi)>0$ it holds that $\L(\psi)\leq\M(\psi)$ 
        ensure that for every $f\in B^n_1$, $N\in\N$, $\eps\in(0,\infty)$ it holds
	\begin{align}\begin{split}
	&\quad\M(\phi_{f,\eps,N})\\
	&\leq 2\pa{\M(\lambda_N)+2\brac{\M(\Pc(\psi_{f,\eps,N,1},\psi_{f,\eps,N,2},\dots,\psi_{f,\eps,N,N}))+\M(\nabla_{1,0,2N+2})}}\\
	&\leq 2(N+1)+8\brac{\sum_{j=0 }^N \M(\psi_{f,\eps,N,j})}+16(N+1)\brac{\max_{j\in\zeroto{N}}\L(\psi_{f,\eps,N,j})}+8(N+1)\\
	&\leq 20N+32(N+1)\max_{j\in\oneto{N}}\M(\psi_{f,\eps,N,j})\\
	&\leq 20N+64N\pa{\M(\Pi^2_{\nicefrac{\eps}{8}})+\M(\Pc(\chi_{N,N},\tau_{f,\eps,N,N}))}\\
	&\leq 20N+128NC \pa{\abs{\ln(\tfrac{\eps}{8})}+2\ln(3)+\ln(2)}\\
	&\quad+64N\pa{2\M(\chi_{N,N})+2\M(\tau_{f,\eps,N,N})+4\max\{\L(\chi_{N,N}),\L(\tau_{f,\eps,N,N})\}}\\
	&\leq 20N +128NC\pa{\abs{\ln(\tfrac{\eps}{8})}+\ln(18)} +1152N\\
	&\quad+128N\pa{32n^2C\pa{|\ln(\tfrac{\eps}{8e})|+n\ln(3)+\ln(n)}+42n}\\
	&\quad+128N\pa{3+C\ln(n)\pa{ |\ln(\tfrac{\eps}{8e})|+n\ln(3)+\ln(n)}}\\
	&=128\pa{C+32n^2C+C\ln(n)}N|\ln(\eps)|\\
	&\quad+ \pa{1556+128(C\ln(144)+64n^2C(n\ln(3)+\ln(n))+42n}N\\
	&=T_2N|\ln(\eps)|+T_3N.
	\end{split}\end{align}
	\normalsize
	Combining this with \Cref{Y1tech} demonstrates that for every $f\in B^n_1$, $\eps\in(0,\exp(-2n^2)]$ it holds
	\begin{align}\begin{split}\label{NNsFMPhi}
	\M(\Phi_{f,\eps})&=\M(\phi_{f,\eps,N_{\eps}})\leq T_2N_{\eps}|\ln(\eps)|+T_3N_{\eps}\\
	&=T_2\left\lceil\brac{\tfrac{2}{n!\eps}}^{\nicefrac{1}{n}}\right\rceil|\ln(\eps)|+T_3\left\lceil\brac{\tfrac{2}{n!\eps}}^{\nicefrac{1}{n}}\right\rceil\\
	&\leq 3T_2\eps^{-\frac{1}{n}}|\ln(\eps)|+3T_3\eps^{-\frac{1}{n}}\\
	&\leq 3T_2\eps^{-\frac{1}{n}}\max\{r,|\ln(\eps)|\}+3T_3\eps^{-\frac{1}{n}}.
	\end{split}\end{align}
	Hence we obtain
	\begin{align}\begin{split}\label{NNsFMest1}
	\sup_{f\in B^n_1,\eps\in(0,\exp(-2n^2))}\!\brac{\frac{\M(\Phi_{f,\eps})}{\eps^{-\frac{1}{n}}\max\{r,|\ln(\eps)|\}}}\leq 3T_2+3T_3\frac{1}{\max\{r,2n^2\}}<\infty.
	\end{split}\end{align}
	Combining \eqref{NNsFMPhi} with the fact that continuous function are bounded on compact sets ensures
	\begin{align}\begin{split}\label{NNsFMest2}
	&\quad\,\sup_{f\in B^n_1,\eps\in[\exp(-2n^2),1]}\!\brac{\frac{\M(\Phi_{f,\eps})}{\eps^{-\frac{1}{n}}\max\{r,|\ln(\eps)|\}}}\\
	&\leq\sup_{f\in B^n_1,\eps\in[\exp(-2n^2),1]}\!\brac{\frac{T_2N(|\ln(\eps)|+|\ln(N)|)+T_3N}{\eps^{-\frac{1}{n}}\max\{r,|\ln(\eps)|\}}}<\infty.
	\end{split}\end{align}
	In addition note
	\begin{align}
	\sup_{f\in B^n_1,\eps\in(1,\infty)}\!\brac{\frac{\M(\Phi_{f,\eps})}{\eps^{-\frac{1}{n}}\max\{r,|\ln(\eps)|\}}}&=
	\sup_{f\in B^n_1,\eps\in(1,\infty)}\!\brac{\frac{\M(\theta)}{\eps^{-\frac{1}{n}}\max\{r,|\ln(\eps)|\}}}\\&
	=\sup_{f\in B^n_1,\eps\in(1,\infty)}\!\brac{\frac{0}{\eps^{-\frac{1}{n}}\max\{r,|\ln(\eps)|\}}}=0<\infty.
	\end{align}
	This, \eqref{NNsFMest1}, and \eqref{NNsFMest2} establish that the neural networks 
        $(\Phi_{f,\eps})_{f\in B^n_1,\eps\in(0,\infty)}$ satisfy \eqref{NNsFi2}. 
        The proof of \Cref{NNsmoothFunctions} is completed.
\end{proof}

\subsection{Proof of \Cref{Yarotsky2}}
  \begin{proof}[Proof of \Cref{Yarotsky2}]
    Throughout this proof assume \Cref{NNSettingPar}, let $c_{a,b}\in\R$, $[a,b]\subseteq\R_+$, 
    be the real numbers given by $c_{a,b}=\min\{1,(b-a)^{-n}\}$,  
    let $\lambda_{a,b}\in\Nn^{1,1}_1$, $[a,b]\subseteq\R_+$, 
    be the neural networks given by $\lambda_{a,b}=(\tfrac{1}{b-a},-\tfrac{a}{b-a})$,
    let $\alpha_f\in\Nn^{1,1}_1$, $f\in \CN$ 
    be the neural networks given by $\alpha_f=(\tfrac{1}{c}\norm{f}_{n,\infty},0)$,
    let $L_{a,b}\colon[0,1]\to[a,b]$, $[a,b]\subseteq\R_+$ 
    be the functions which satisfy for every $[a,b]\subseteq\R_+$, $t\in[0,1]$
    \begin{align}
      L_{a,b}(t)=(b-a)t+a,\label{Y2Ldef}
    \end{align}
    and for every $f\in\CN$ let $f_*\in C^n([0,1],\R)$ be the function which satisfies for every $t\in[0,1]$ 
    \begin{align}\label{Y2fstarDef}
      f_*(t)=\fninf^{-1}c_{a,b}(f(L_{a,b}(t))).
    \end{align}
We claim that for every $[a,b]\subseteq\R_+$, $f\in C^n([a,b],\R)$, $m\in\oneto{n}$, $t\in[0,1]$ it holds
 \begin{align}
  f_*^{(m)}(t)=\fninf^{-1}c_{a,b}(b-a)^m[f^{(m)}(L_{a,b}(t))].\label{Y2claim}
 \end{align}
We now prove \eqref{Y2claim} by induction on $m\in\oneto{n}$. 
For the base case $m=1$, the chain rule implies 
for every $[a,b]\subseteq\R_+$, $f\in C^n([a,b],\R)$, $t\in[0,1]$ 
\begin{align}\begin{split}
f_*'(t)&=\ddt\brac{\fninf^{-1}c_{a,b}f(L_{a,b}(t))}
=\fninf^{-1}c_{a,b}\brac{f'(L_{a,b}(t))L_{a,b}'(t)}
\\
&=\fninf^{-1}c_{a,b}\brac{f'(L_{a,b}(t))(b-a)}
=\fninf^{-1}c_{a,b}(b-a)[f'(L_{a,b}(t))].
\end{split}\end{align}
This establishes \eqref{Y2claim} in the base case $m=1$.

For the induction step $\{1,2,\dots,n-1\}\ni m\to m+1\in\{2,3,\dots,n\}$
observe that the chain rule ensures for every 
$[a,b]\subseteq\R_+$, $f\in C^n([a,b],\R)$, $m\in\N$, $t\in[0,1]$ 
\begin{align}\begin{split}
 \ddt\brac{\fninf^{-1} c_{a,b}(b-a)^m[f^{(m)}(L_{a,b}(t))]}
&=\fninf^{-1} c_{a,b}(b-a)^m[f^{(m+1)}(L_{a,b}(t))L_{a,b}'(t)]\\
&=\fninf^{-1} c_{a,b}(b-a)^{m+1}[f^{(m+1)}(L_{a,b}(t))].
\end{split}
\end{align}
Induction thus establishes \eqref{Y2claim}.

In addition, for every $[a,b]\subseteq\R_+$, $k\in\{0,1,\dots,n\}$ 
\begin{align}\begin{split}
 c_{a,b}(b-a)^k&=\min\{1,(b-a)^{-n}\}(b-a)^k = \min\{(b-a)^k,(b-a)^{-n+k}\}\leq 1.
\end{split}\end{align}
Combining this with \eqref{Y2ninfnorm1}, \eqref{Y2Ldef}, and \eqref{Y2claim} 
    ensures for every $[a,b]\subseteq\R_+$, $f\in C^n([a,b],\R)$ 
    \begin{align}\begin{split}
      \max_{k\in\{0,1,\dots,n\}}\brac{\sup_{t\in[0,1]}\abs{f_*^{(k)}(t)}}&=\max_{k\in\{0,1,\dots,n\}}\brac{\sup_{t\in[a,b]}\abs{\fninf^{-1}c_{a,b}(b-a)^k[f^{(k)}(t)]}}\\
      &\leq  \fninf^{-1}\max_{k\in\{0,1,\dots,n\}}\brac{\sup_{t\in[a,b]}\abs{f^{(k)}(t)}}= 1.
    \end{split}\end{align} 
    \Cref{NNsmoothFunctions} therefore establishes that there exist neural networks 
    $(\Phi_{g,\eta})_{g\in B^n_1,\eta\in(0,\infty)}\subseteq\Nall$ which satisfy
    \begin{enumerate}[(a)]
    \item\label{Y2a}
	$\displaystyle\sup_{g\in B^n_1,\eta\in(0,\infty)}\brac{\frac{\L(\Phi_{g,\eta})}{\max\{r,\abs{\ln(\eta)}\}}}<\infty$,
    \item\label{Y2b}
	$\displaystyle\sup_{g\in B^n_1,\eta\in(0,\infty)}\brac{\frac{\M(\Phi_{g,\eta})}{\eta^{-\frac{1}{n}}\max\{r,|\ln(\eta)|\}}}<\infty$, 
      and
    \item\label{Y2c}
      for every $g\in B^n_1$, $\eta\in(0,\infty)$ that
      \begin{align}
	\sup_{t\in[0,1]}\abs{g(t)-\brac{R_{\relu}(\Phi_{g,\eta})}\!(t)}\leq \eta.
      \end{align}
    \end{enumerate}
    Let $\pa{\Phi_{f,\eps}}_{f\in\CN,\eps\in(0,\infty)}\subseteq\Nall$ 
    denote neural networks which satisfy for every 
    $[a,b]\subseteq\R_+$, $f\in C^n([a,b],\R)$, $\eps\in(0,\infty)$ 
    \begin{align}
      \Phi_{f,\eps}=\alpha_f\odot\phi_{f_*,\frac{c_{a,b}\eps}{\fninf}}\odot\lambda_{a,b}.
    \label{Y2PepsDef}
    \end{align}
    Observe that for every $[a,b]\subseteq\R_+$, $f\in C^n([a,b],\R)$, $t\in[0,1]$ it holds 
    \begin{align}
      [R_{\rho}(\lambda_{a,b})](t)=\brac{\tfrac{1}{(b-a)}}t-\tfrac{a}{(b-a)}=L_{a,b}^{-1}(t)\qquad 
              \text{and}\qquad [R_{\rho}(\alpha_f)](t)=\tfrac{\fninf}{c_{a,b}}t.
    \end{align}
    \Cref{NNconc} therefore demonstrates for every 
           $[a,b]\subseteq\R_+$, $f\in C^n([a,b],\R)$, $\eps\in(0,\infty)$, $t\in[0,1]$ it holds
    \begin{align}\begin{split}\label{Y2R}
      [R_{\rho}(\Phi_{f,\eps})](t)&=[R_{\rho}(\alpha_f\odot\phi_{f_*,\frac{c_{a,b}\eps}{\fninf}}\odot\lambda_{a,b})](t)\\
      &=[\Rr(\alpha_f)\circ R_{\rho}(\phi_{f_*,\frac{c_{a,b}\eps}{\fninf}})\circ R_{\rho}(\lambda_{a,b})](t)\\
      &=\tfrac{\fninf}{c_{a,b}}[\Rr(\phi_{f_*,\frac{c_{a,b}\eps}{\fninf}})](L_{a,b}^{-1}(t)).
    \end{split}\end{align}
    Moreover, note \eqref{Y2fstarDef} ensures that for every $[a,b]\subseteq\R_+$, $f\in C^n([a,b],\R)$, $t\in[a,b]$ it holds 
    \begin{align}
      f(t)=\tfrac{\fninf}{c_{a,b}}f_*(L_{a,b}^{-1}(t)).
    \end{align}
    Combining \eqref{Y2c}, \eqref{Y2PepsDef}, and \eqref{Y2R} implies for 
    every $[a,b]\subseteq\R_+$, $f\in C^n([a,b],\R)$, $\eps\in(0,\infty)$
    \begin{align}\begin{split}
      \sup_{t\in[a,b]}\abs{f(t)-\brac{R_{\relu}(\Phi_{f,\eps})}\!(t)}
     &=\sup_{t\in[a,b]}\abs{\tfrac{\fninf}{c_{a,b}}f_*(L_{a,b}^{-1}(t))-\tfrac{\fninf}{c_{a,b}}[\Rr(\phi_{f_*,\frac{c_{a,b}\eps}{\fninf}})](L_{a,b}^{-1}(t))}\\
      &=\tfrac{\fninf}{c_{a,b}}\brac{\sup_{t\in[0,1]}\abs{f_*(t)-[\Rr(\phi_{f_*,\frac{c_{a,b}\eps}{\fninf}})](t)}}\leq \tfrac{\fninf}{c_{a,b}} \tfrac{c_{a,b}\eps}{\fninf}=\eps.
    \end{split}\end{align}
    This establishes that the neural networks $\pa{\Phi_{f,\eps}}_{f\in\CN, \eps\in(0,\infty)}$ satisfy \eqref{Y2iii}.
    Furthermore, \Cref{NNconc} ensures 
    for every $[a,b]\subseteq\R_+$, $f\in C^n([a,b],\R)$, $\eps\in(0,\infty)$ holds
    \begin{align}\begin{split}\label{Y2T1}
      \L(\Phi_{f,\eps})=\L(\alpha_f\odot\phi_{f_*,\frac{c_{a,b}\eps}{\fninf}}\odot\lambda_{a,b})=\L(\alpha_f)+\L(\phi_{f_*,\frac{c_{a,b}\eps}{\fninf}})+\L(\lambda_{a,b})=\L(\phi_{f_*,\frac{c_{a,b}\eps}{\fninf}})+2.
    \end{split}\end{align}
    In addition, for every $[a,b]\subseteq\R_+$, $f\in C^n([a,b],\R)$, $\eps\in(0,\infty)$ holds 
    \begin{align}\begin{split}\label{Y2T2}
      \max\{r,|\ln(\tfrac{c_{a,b}\eps}{\fninf})|\}
     &=\max\{r,|\ln(\tfrac{\min\{1,(b-a)^{-n}\}\eps}{\fninf})|\}
       =\max\{r,|\ln(\tfrac{\eps}{(\max\{1,(b-a)\})^n\fninf})|\}
    \\
     &\leq n\max\{r,|\ln(\tfrac{\eps}{(\max\{1,(b-a)\})\fninf})|\}.
    \end{split}\end{align}
    Combining this with \eqref{Y2a} and \eqref{Y2T1} implies that
    \begin{align}\begin{split}
      \sup_{f\in\CN, \eps\in(0,\infty)}
          \brac{\frac{\L(\Phi_{f,\eps})}{\max\{r,|\ln(\frac{\eps}{\max\{1,b-a\}\norm{f}_{n,\infty}})|\}}}
      &\leq n\sup_{f\in\CN, \eps\in(0,\infty)}
    \brac{\frac{\L(\phi_{f_*,\frac{c_{a,b}\eps}{\fninf}})+2}{\max\{r,|\ln(\tfrac{c_{a,b}\eps}{\fninf})|\}}}
      \\
      &=n\sup_{g\in B^n_1,\eta\in(0,\infty)}\brac{\frac{\L(\Phi_{g,\eta})+2}{\max\{r,\abs{\ln(\eta)}\}}}
       <\infty.    
    \end{split}\end{align} 
    This establishes that the neural networks 
    $\pa{\Phi_{f,\eps}}_{f\in\CN, \eps\in(0,\infty)}$ satisfy \eqref{Y2i}. 
    Next, \Cref{NNconc} implies that for every 
    $[a,b]\subseteq\R_+$, $f\in C^n([a,b],\R)$, $\eps\in(0,\infty)$ 
    \begin{align}\begin{split}
      \M(\Phi_{f,\eps})
    &=\M(\alpha_f\odot\phi_{f_*,\frac{c_{a,b}\eps}{\fninf}}\odot\lambda_{a,b})
     =\M(\alpha_f)+\M(\phi_{f_*,\frac{c_{a,b}\eps}{\fninf}})+\M(\lambda_{a,b})
     =\M(\phi_{f_*,\frac{c_{a,b}\eps}{\fninf}})+3.
    \end{split}\end{align}
    In addition, note that \eqref{Y2T2} shows for every 
    $[a,b]\subseteq\R_+$, $f\in C^n([a,b],\R)$, $\eps\in(0,\infty)$ 
    \begin{align}\begin{split}
      \quad\brac{\tfrac{c_{a,b}\eps}{\fninf}}^{-\frac{1}{n}}\max\{r,|\ln(\tfrac{c_{a,b}\eps}{\fninf})|\}
      n\leq\max\{1,b-a\}\norm{f}_{n,\infty}^{\frac{1}{n}}\eps^{-\frac{1}{n}}
                \max\{r,|\ln(\tfrac{\eps}{\max\{1,b-a\}\norm{f}_{n,\infty}})|\}.
    \end{split}\end{align}
    Combining this with \eqref{Y2b} and \eqref{Y2PepsDef} therefore ensures
    \begin{align}\begin{split}
      &\quad\sup_{f\in\CN, \eps\in(0,\infty)}\brac{\frac{\M(\Phi_{f,\eps})}{\max\{1,b-a\}\norm{f}_{n,\infty}^{\frac{1}{n}}\eps^{-\frac{1}{n}}\max\{r,|\ln(\frac{\eps}{\max\{1,b-a\}\norm{f}_{n,\infty}})|\}}}\\
      &\leq n\sup_{f\in\CN, \eps\in(0,\infty)}\brac{\frac{\M(\phi_{f_*,\frac{c_{a,b}\eps}{\fninf}})+3}{\quad\brac{\frac{c_{a,b}\eps}{\fninf}}^{-\frac{1}{n}}\max\{r,|\ln(\frac{c_{a,b}\eps}{\fninf})|\}}}\\
      &\leq n\sup_{g\in B^n_1,\eta\in(0,\infty)}\brac{\frac{\M(\Phi_{g,\eta})+3}{\eta^{-\frac{1}{n}}\max\{r,|\ln(\eta)|\}}}<\infty.     
    \end{split}\end{align}
    This establishes that the neural networks $\pa{\Phi_{f,\eps}}_{f\in\CN, \eps\in(0,\infty)}$ 
    satisfy \eqref{Y2ii} and completes the proof.
  \end{proof} 
\end{document}